\documentclass[12pt]{amsart}
\usepackage{amsmath}
\usepackage{amsthm}
\usepackage{comment}
\usepackage{graphicx}
\usepackage{enumerate}
\usepackage{color}
\usepackage{stmaryrd}

\usepackage{hyperref}
\usepackage{todonotes,color}  
\usepackage{caption}
\usepackage{subcaption}
\newtheorem{theorem}{Theorem}[section]

\newtheorem{corollary}[theorem]{Corollary}

\newtheorem{definition}[theorem]{Definition}

\newtheorem{lemma}[theorem]{Lemma}

\newtheorem{proposition}[theorem]{Proposition}
\newtheorem{remark}[theorem]{Remark}

\def\N{\mathbb{N}}
\def\R{\mathbb{R}}
\def\C{\mathbb{C}}
\def\Z{\mathbb{Z}}
\def\exp{{\rm{exp}}}

\def\re{{\rm{Re}}}

\def\e{\varepsilon}

\def\a{\alpha}

\def\ft{\tilde{f}}
\def\vt{\tilde{v}}

\def\vb{\overline{v}}

\def\R{\mathbb{R}}
\def\dd{\,\textrm{d}}

\def\hc{f}
\def\op{\mathcal{F}}
\def\opg{\mathcal{S}}
\def\opF{\mathcal{F}}
\def\opR{\mathcal{R}}

\def\opL{\mathcal{S}}
\def\opN{\mathcal{N}}
\def\opP{\mathcal{P}}

\def\opT{\mathcal{T}}
\def\opTt{\widetilde{\mathcal{T}}}

\def\O{\mathcal{O}}
\def\co{\mathcal{C}}

\def\w{w}

\def\m{\textrm{min}}

\def\dg{\delta g}

\def\dv{\delta v}
\def\dgt{\widetilde{\delta g}}

\def\Xin{\mathcal{X}}
\def\Zin{\mathcal{Z}}
\def\Yin{\mathcal{Y}}

\def\c{c}
\def\esv{}
\def\esvq{}
\def\esvi{}

\def\Cin{\mathbf{b}}
\def\Cout{\mathbf{a}}

\def\fdout{f_0^{\mathrm{out}}}
\def\vdout{v_0^{\mathrm{out}}}
\def\fdin{f_0^{\mathrm{in}}}
\def\vdin{v_0^{\mathrm{in}}}

\def\gout{g^{\mathrm{out}}}
\def\wout{w^{\mathrm{out}}}
\def\gin{g^{\mathrm{in}}}
\def\win{w^{\mathrm{in}}}

\def\fout{f^{\mathrm{out}}}
\def\vout{v^{\mathrm{out}}}
\def\fin{f^{\mathrm{in}}}
\def\vin{v^{\mathrm{in}}}

\newcommand\f[1]{\hc_{#1}}
\renewcommand\v[1]{v_{#1}}

\author{M. Aguareles$^1$}
\address{${}^{1}$Departament de'Inform\`atica, Matem\`atica Aplicada i Estad\'istica, Universitat de Girona, Girona, Spain. }
\email{maria.aguareles@udg.edu}

\author{I. Baldom\'a$^{2}$}
\address{${}^{2}$Departament de Matem\`atiques, Universitat Polit\`ecnica de Catalunya (UPC), IMTECH (UPC),
Centre de Recerca Matem\`atica (CRM), Barcelona, Spain}
\email{immaculada.baldoma@upc.edu}
\author{T. M-Seara$^{3}$}
\address{${}^{3}$Departament de Matem\`atiques, Universitat Polit\`ecnica de Catalunya (UPC), IMTECH (UPC),
Centre de Recerca Matem\`atica (CRM), Barcelona, Spain}
\email{tere.m-seara@upc.edu}

\title{A rigorous derivation of the asymptotic wavenumber of spiral wave solutions of the complex Ginzburg-Landau equation.}

\date{\today}

\begin{document}
\maketitle

\begin{abstract}
 In this work $n$-armed Archimedian spiral wave solutions of the complex Ginzburg-Landau equation are considered. These solutions are showed to depend on two characteristic parameters, the so called \emph{twist parameter}, $q$, and the \emph{asymptotic wavenumber} $k$. The existence and uniqueness of the value of $k=k_*(q)$ for which  $n$-armed Archimedian spiral wave solutions exist is a classical result,
obtained back in the 80’s by Kopell and Howard. In this work we deal with a different problem, that is, the asymptotic expression of $k_*(q)$ as $q \to 0$.
Since the eighties, different heuristic perturbation techniques, like formal asymptotic expansions, have conjectured an asymptotic expression of $k_* (q)$ which is of the form $k_*(q) \sim C q^{-1} e^{-\frac{\pi}{2n |q|}}$ being $C$ a known constant. However, the validity of this expression has remained opened until now, despite of the fact that it has been widely used for more than $40$ years. In this work, using a functional analysis approach, we finally prove the validity of the asymptotic formula for $k_*(q)$, providing a rigorous bound for its relative error, which turns out to be $k_*(q)= C q^{-1} e^{-\frac{\pi}{2nq}} (1+ \mathcal{O}(|\log q|^{-1})$. Moreover, such approach can be used in more general equations such as the celebrated $\lambda-\omega$ systems.
\end{abstract}

\section{Introduction}
In a wide range of physical, chemical and biological systems of different interacting species, one usually finds that the dynamics of each species is governed by a diffusion mechanism along with a reaction term, where the interactions with the other species are taken into account. For instance, one finds these type of systems in the modelling of chemical reaction processes as a model for pattern formation mechanisms (\cite{hohenberg93}), in the description of some ecological systems (\cite{murray01}), in phase transitions in superconductivity (\cite{hoffmann12}) or even to describe cardiac muscle cell performance \cite{erhardt22}, among many others. Mathematically, a reaction-diffusion system is essentially a system of ordinary differential equations to which some diffusion terms have been added:
\begin{equation}
\label{eq:RD}
    \partial_\tau U= D\Delta U + F(U,a),
\end{equation}
where $U=U(\tau,\vec{x})\in \mathbb{R}^N$, $\vec{x}=(x,y)\in \R^2$, $\tau\in \R$, $D$ is a diffusion matrix, $F$ is the reaction term, which is usually nonlinear, $\Delta = \partial_{xx}+\partial_{yy}$ is the Laplace operator and $a$ is a parameter (for instance some catalyst concentration in a chemical reaction) or a group of parameters.

In this paper we deal with a particular type of reaction-diffusion equations which are traditionally denoted as \emph{oscillatory systems}. These are characterised by the fact that they tend
to produce oscillations in homogeneous situations (i.e. when the term $D\Delta U$ vanishes). Of particular interest are oscillatory reaction-diffusion systems which tend
to produce spatial homogeneous oscillations. These are systems like \eqref{eq:RD} where the dynamical system that is obtained when one neglects the spatial derivatives (i.e., the Laplace operator) has an asymptotically stable periodic orbit. To be more precise, we refer to dynamical systems that undergo a non-degenerate supercritical Hopf bifurcation at $(U_0,a_0)$. In this case, one can derive an equation for the amplitude of the oscillations, $A\in \mathbb{C}$, by taking $\e^2=a-a_0>0$ small, $t=\e^2 \tau$ and writing the modulation of local oscillations with frequency $\omega$ as solutions of \eqref{eq:RD} of the form
    $$
    U(\tau,\vec{x},a)=U_0 + \e [ A(t , \vec{x})e^{i\omega \tau} v+\bar{A}(t , \vec{x})e^{-i\omega \tau} \bar{v}]+ \mathcal{O}(\e^2),
    $$
where $\bar{}$ denotes the complex conjugate. Under generic conditions, performing suitable scalings and upon neglecting the higher order terms in $\varepsilon$ (see, for instance, Section 2
in \cite{kuramoto03}, \cite{aranson02}, or \cite{mielke02}), the amplitude, $A(t , \vec{x})$, turns out to satisfy the celebrated complex Ginzburg-Landau equation (CGL)
\begin{equation}
\label{eq:CGL}
\partial_t A = (1+i\alpha) \Delta A + A - (1+i\beta) A |A|^2,
\end{equation}
where $A(t,\Vec{x})\in\mathbb{C}$ and $\alpha,\beta$ are real parameters (depending on $F$ and $D$). The universality and ubiquity of CGL has historically produced a large amount of research and it is one of the most studied nonlinear partial differential system of equations specially among the physics community. The CGL equation is also known to exhibit a rich variety of different pattern solutions whose stability and emergence are still far from being completely understood (see \cite{correia20}, \cite{plec01}, \cite{doelman09}, \cite{scheel03}, \cite{sandstede20}, \cite{dodson19} for some of the latest achievements and open problems).

We note that \eqref{eq:CGL} has two special features: the solutions are invariant under spatial translations, that is, if $A(t,\vec{x})$ is a solution, then $A(t,\vec{x}+ \vec{x}_0)$ does also satisfy equation \eqref{eq:CGL} for any fixed $\vec{x}_0\in\R^2$, and it also has gauge symmetry, that is $\widetilde{A}(t,\vec{x})= e^{i\phi} A(t,\vec{x})$ is a solution for any $\phi\in\R$.

In this work we shall focus on some special rigidly rotating solutions of~\eqref{eq:CGL} called \emph{Archimedian spiral waves}. In order to define these solutions, following~\cite{sandstede20}, we consider first polar coordinates, that is $\vec{x}=(r\cos \varphi,r \sin \varphi) \in \mathbb{R}^2$ in which equation \eqref{eq:CGL} reads:
\begin{equation}
\label{eq:CGL_pol}
\partial_t A = (1+i\alpha) \left(\partial^2_r A+
\frac{1}{r}\partial_r A+\frac{1}{r^2}\partial^2_\varphi A\right)+ A - (1+i\beta) A |A|^2,
\end{equation}
where, abusing notation, we denote by the same letter $A(t,r,\varphi)$ the solution in polar coordinates.
To define spiral waves let us first consider the one dimensional CGL equation:
\begin{equation}
\label{eq:CGL1d}
\partial_t A = (1+i\alpha) \partial ^2_r A + A - (1+i\beta) A |A|^2 ,\qquad r\in\mathbb{R}
\end{equation}
and introduce the notion of \emph{wave train}.
\begin{definition}\label{def:wavetrain}
A \emph{wave train} of~\eqref{eq:CGL_pol} is a non constant solution, $A(t,r)$, of equation~\eqref{eq:CGL1d} of the form:
\begin{equation}
    \label{sol:1d}
    A(t, r)=A_*(\Omega t -k_*r),
\end{equation}
where the \emph{profile} $A_*(\xi)$ is $2\pi$-periodic, $\Omega\in\R\backslash \{0\}$ is the frequency of the wave train and $k_*\in\R$ is the corresponding (spatial) wavenumber.
\end{definition}

The particular case of a single mode wave train, namely $A(t,r)= C e^{i(\Omega t - k_*r)}$ leads to the well-known relations
\begin{equation}\label{eq:disp}
C=\sqrt{1-k_*^2}, \qquad \Omega = \Omega (k_*) = - \beta +k_*^2 (\beta- \alpha).
\end{equation}
The last condition on the frequency is the associated \textit{dispersion relation}. Then, for any pair of the parameter values $(\alpha, \beta)$ there exist a family of {single mode} wave trains of~\eqref{eq:CGL1d} of the form given in~\eqref{sol:1d} satisfying conditions~\eqref{eq:disp}, one for each wavenumber $k_*$.

Now we define (see Definition~\ref{def:archimedian}) an \emph{$n$-armed Archimedian spiral wave} which, roughly speaking, is a bounded solution of~\eqref{eq:CGL_pol} that asymptotically, as $r\to\infty$, tends to a particular wave train (see Figure~\ref{fig:1d}).
From a physical point of view, spiral waves arise when inhomogeneities of the medium force a zero amplitude in particular points in space (\cite{hend61}). These points where the amplitude is forced to vanish are usually known as \emph{defects} (\cite{aranson02}). By virtue of the translation invariance of~\eqref{eq:CGL}, in spiral wave solutions with a single defect, one can place the defect anywhere in space, in particular at the origin, i.e. $ A(t,\vec{0})=0$.

In this work we shall use the following definition of an $n$-armed spiral wave solution of the complex Ginzburg-Landau equation given in \cite{sandstede20}:
\begin{definition}\label{def:archimedian} Let $n\in \mathbb{N}$, we say that $A(t,r,\varphi)$ is a \emph{rigidly rotating Archimedian $n$-armed spiral wave} solution of equation~\eqref{eq:CGL_pol} if it is a bounded solution of the form $A(t,r,\varphi)=A_s(r, n \varphi + \Omega t)$,  defined for $r\geq 0$ satisfying that
$$
\lim_{r\to \infty} \max_{\psi \in [0,2\pi]} |A_s(r,\psi) - A_*(-k_* r + \theta(r) + \psi) | = 0,
$$
and
$$
\lim_{r\to \infty}\max_{\psi \in [0,2\pi]}| \partial_{\psi} A_s(r,\psi) - A_*'(-k_* r + \theta(r) + \psi) |
= 0,
$$
where the profile $ A_*(\Omega t -k_*r )$ is a wave train of the equation~\eqref{eq:CGL1d}, $A_s(r,\cdot)$ is $2\pi$-periodic and $\theta$ is a smooth function such that $\lim_{r\to \infty} \theta'(r) \to 0$.

The parameter $k_*$ is in this case known as the \emph{asymptotic wavenumber} of the spiral.
\end{definition}
Notice that, in a co-rotating frame given by $\psi=n\varphi+ \Omega t$ and considering $r$ as the independent variable, spiral wave solutions can be seen as a heteroclinic orbit, as represented in Figure \ref{fig:1d}, connecting the equilibrium point $A=0$ with the wave train solution $A_*$.
\begin{figure}[ht]
    \centering
    \includegraphics[width=\textwidth/2]{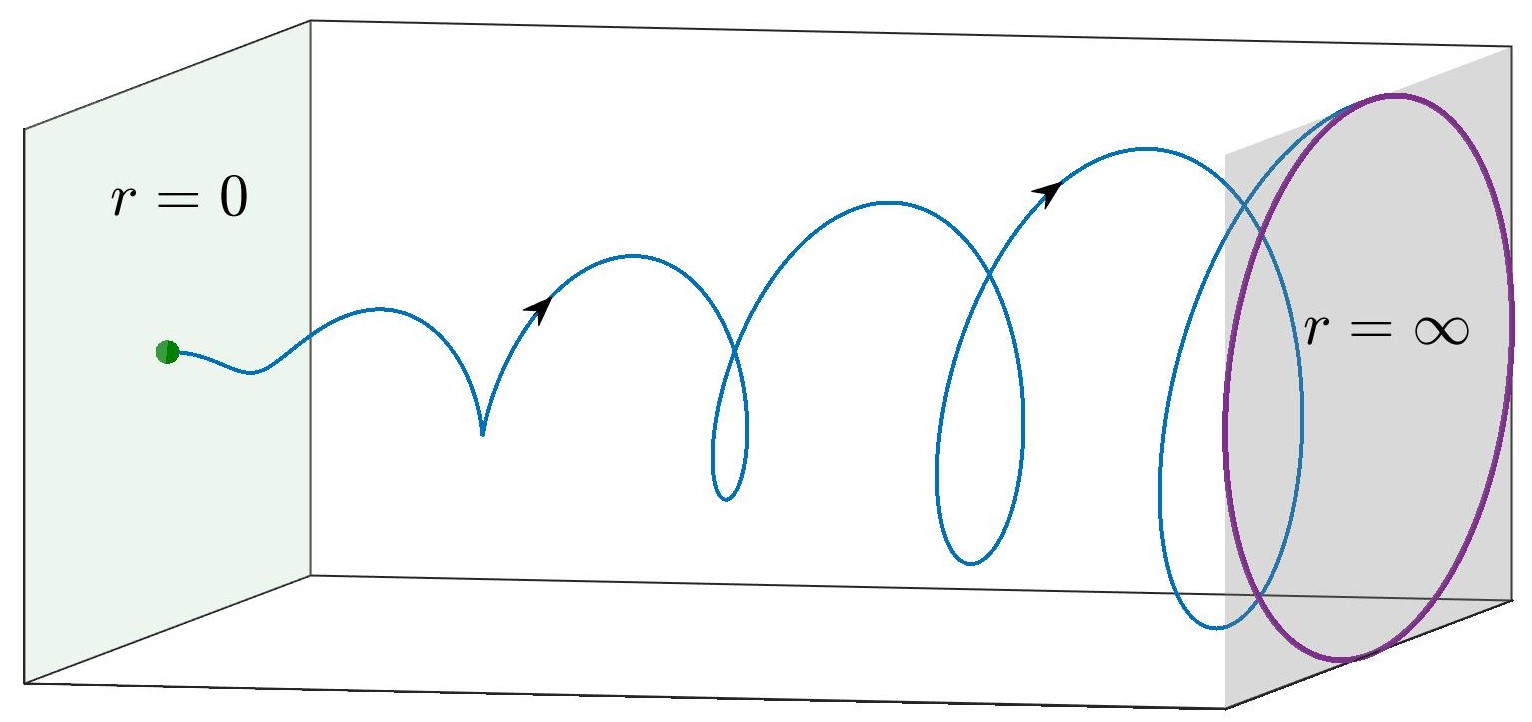}
    \caption{Representation of the spiral wave solutions of \eqref{eq:CGL} as an heteroclinic connection.}
    \label{fig:1d}
\end{figure}

To give the main result of this paper we  introduce  the so-called \emph{twist parameter} $q$:
\begin{equation}\label{intro:eq:q}
 q=\frac{\beta-\alpha}{1+\alpha \beta}
\end{equation}
which, in particular, is well defined for values of $\alpha,\beta $ such that  $|\alpha -\beta| \ll 1$.
 As we shall explain in Section~\ref{secion:Sp}, the shape of the spiral waves strongly depends on this parameter. In fact, when $q=0$,
 the solutions of the Ginzburg-Landau equation~\eqref{eq:CGL} of the form  $A(t,\vec{x}) = e^{-i \alpha t} \hat{A}(t,\vec{x})$ satisfy the ``real'' Ginzburg-Landau equation
 $$
\partial_t \hat{A}= \Delta \hat{A} + \hat{A} - \hat{A}|\hat{A}|^2, \qquad \hat{A}(t,\vec{x}) \in \mathbb{R}.
 $$
Our perturbative analysis considers the case in which we are close to the ``real'' Ginzburg-Landau equation, that is to say, we deal with values of $q$ which are small.

The main result of this paper reads as follows:
\begin{theorem}\label{thm:mainintroduction}
For any $n \in \mathbb{N}$, there exist $q_0>0$, small enough, and a unique function $\kappa_*: (-q_0,q_0) \to \mathbb{R}$ of the form
\begin{equation}\label{intro:k}
\kappa_*(q)= \frac{2}{q}e^{-\frac{C_n}{n^2} -\gamma } e^{-\frac{\pi}{2n|q|}} (1+\mathcal{O}(|\log q|^{-1})),
\end{equation}
with $\gamma$ the Euler's constant and $C_n$ a constant depending only on $n$, satisfying that the complex Ginzburg-Landau equation \eqref{eq:CGL_pol} possesses  rigidly rotating Archimedian $n$-armed spiral wave solutions of the form
\begin{equation} \label{sol:spiral}
A(t,r,\varphi;q)= \mathbf{f}(r;q) \exp \left(i(\Omega  t+\Theta (r;q)\pm n\varphi)\right),
\end{equation}
with a single defect satisfying
\begin{equation*}
\mathbf{f}(0;q)=0,\qquad  \qquad \lim_{r\to \infty} \mathbf{f}(r;q)=\sqrt{1-k_*^2},\qquad
\Theta'(0;q)=0,\qquad
\lim_{r\to \infty} \Theta'(r;q)=-k_*,
\end{equation*}
if and only if the asymptotic wavenumber of the spiral wave is $k_*=\kappa_*(q)$ as given in \eqref{intro:k} and  $\Omega$ satisfies~\eqref{eq:disp}. In addition $\Theta'(r;q)$ has constant sign, that is, for $q$ fixed, $\mathbf{f}(r;q)$ is an increasing function,
$$\mathbf{f}(r;q)>0,\qquad \textrm{for }\; r>0$$
and, as a consequence, $\lim_{r\to \infty} \mathbf{f}'(r;q)=0$.
\end{theorem}
\begin{remark}\label{remark:relative_error}
We emphasize the results of Theorem~\ref{thm:mainintroduction} ensure the existence of a constant $M$ (depending on $q_0$ and $n$) such that for all $q\in (-q_0,q_0)$ one has
$$
\left | \frac{q}{2} e^{\frac{C_n}{n^2} +\gamma } e^{\frac{\pi}{2n|q|}}\kappa_*(q) - 1\right |\leq  \frac{M}{ |\log q|}.
$$
That is, we rigorously bound the relative error of $\kappa_*(q)$ with respect to its dominant term.
\end{remark}
The simple description of spiral wave patterns of \eqref{eq:CGL} clashes with the complexity of obtaining rigorous results on their existence, stability or emergence. In fact, the existence and uniqueness of $\kappa_*(q)$ and, as a consequence, of the rotational frequency of the pattern $\Omega$, is a classical result that was obtained in the 80's by Kopell \& Howard in \cite{kopell}.
At the same time the physics community started showing interest in this type of phenomena and several authors used formal perturbation analysis techniques to describe spiral wave solutions (see for instance \cite{greenberg81}, \cite{cohen78} or \cite{yamada76}).
More relevantly, Greenberg in \cite{Greenberg} and Hagan in \cite{hagan82} used formal techniques of matched asymptotic expansions to conjecture an asymptotic formula for $k_*=\kappa_*(q)$ when $q$ is small. The conjectured expression \eqref{intro:k} of the wavenumber $k_* (q)$, has been widely used in the literature and checked numerically in innumerable occasions (see for instance \cite{hohenberg93}, \cite{aranson02}, \cite{mikhailov2012}, \cite{coullet1989}, \cite{pinto2001} or \cite{Tsai2010}) but it has never been rigorously proved, that is the main purpose of the present paper. Furthermore, and as far as the authors know, in the previous works where expression~\eqref{intro:k} was formally derived the order of the error was either not mentioned or it was considered (without proof) to be $\mathcal{O}(q)$.

The precise computation of the constants in the exponentially small terms arising in~\eqref{intro:k} was already a challenge to overcome when the formal derivation was obtained and, in fact, 30 years later in \cite{AgChWi08}, a new simpler formal asymptotic scheme was used. It is therefore not that surprising that it has taken more than 40 years to finally obtain a rigorous proof of the expression~\eqref{intro:k} (see Remark~\ref{remark:relative_error}).

The novelty of our approach is to introduce a suitable functional setting which allows as to prove
that a necessary and sufficient condition for the spiral waves to exist is that the associated wavenumber, $k_*$, has to be exactly $\kappa_*(q)$ as in~\eqref{intro:k}.
This functional approach has furthermore allowed to provide a very detailed description of the structure of the whole spiral wave solutions, of which several features, such as positivity or monotonicity among many others, have now been rigorously established.

Archimedian spiral wave patterns are present in some other systems. In particular, there is another type of reaction-diffusion systems, the so-called \emph{$\lambda-\omega$ systems}, which have been classically used to investigate rotating spiral wave patterns:
\begin{equation}
\label{eq:lambda-omega}
    \frac{\partial}{\partial t}\begin{pmatrix}
u_1\\
u_2
\end{pmatrix} = \begin{pmatrix} \lambda(f) & -\omega(f)\\
\omega(f) &\lambda(f) \end{pmatrix}
\begin{pmatrix}
u_1\\
u_2
\end{pmatrix} +\Delta \begin{pmatrix}
u_1\\
u_2
\end{pmatrix},
\end{equation}
where $u_1(t,\vec{x}),u_2(t,\vec{x})\in\R$ and $\omega(\cdot ),\lambda(\cdot )$ are real functions of the modulus $f=\sqrt{u_1^2+u_2^2}$.
Actually, this system was first introduced by Kopell \& Howard in \cite{kopell73} as a model to describe plane wave solutions in oscillatory reaction diffusion systems. Not much later the same authors in \cite{kopell74a}, \cite{kopell74b} and \cite{kopell}, under some assumptions on $\lambda, \omega$, rigorously proved the existence and uniqueness of spiral wave solutions of \eqref{eq:lambda-omega} with a single mode.
Later, in~\cite{AgBaSe2016}, the authors proved that, in fact, the asymptotic wavenumber $k_*=k_*(q)$ has to be a flat function of the (small) parameter $q$. The particularity of this system is that the equations satisfied by spiral waves turn out to be exactly the same as the ones for the CGL equation when $\lambda(z)=1-z^2$ and $\omega(z)=\Omega+q(1-k^2-z^2)$, as we show later in Remark \ref{rem:lamda-omega}.

\subsection{Spiral patterns}\label{secion:Sp}

By Definition~\ref{def:archimedian} of Archimedian spiral waves, spiral wave solutions of the form~\eqref{sol:spiral} provided by Theorem~\ref{thm:mainintroduction}, have to tend, as $r\to \infty$, to
\begin{equation*}
A_*(\Omega t-k_*r+\theta(r))=C e^{i(\Omega t - k_*r+\theta(r))}
\end{equation*}
with $A(t,r)=A_*(\Omega t-k_*r)$ a wave train of~\eqref{eq:CGL1d}, that is $C,\Omega\in \mathbb{R}$ satisfying~\eqref{eq:disp} and $\theta'(r) \to 0$ as $r\to \infty$.
We will see in Section~\ref{sec:derivationequation} that, in fact, these are the only possible wave trains of~\eqref{eq:CGL1d}, namely, wave trains of equation~\eqref{eq:CGL1d} only have one mode.
The contour lines of $A_*$, that is to say, $\mathrm{Re} \big (A_*(\Omega t-k_*r +n\varphi) e^{-i\Omega t}\big )=c$ for any real constant $c$ (or equivalently $-k_*r + n\varphi = c'$), are Archimedian spirals whose wavelength $L$ (distance between two spiral arms) is given by
$$
L= \frac{2\pi n}{|k_*|}.
$$
The parameter $n\in\Z$ is known as the \emph{winding number} of the spiral and it represents the number of times that the spiral crosses the positive horizontal axis when $\varphi$ is increased by $2\pi$. In Figure~\ref{fig:spirals} we represent $n$\emph{-armed} archimedian spirals for different winding numbers, $n$.
\begin{figure}[ht]
     \centering
     \begin{subfigure}[b]{0.24\textwidth}
         \centering
         \includegraphics[width=\textwidth]{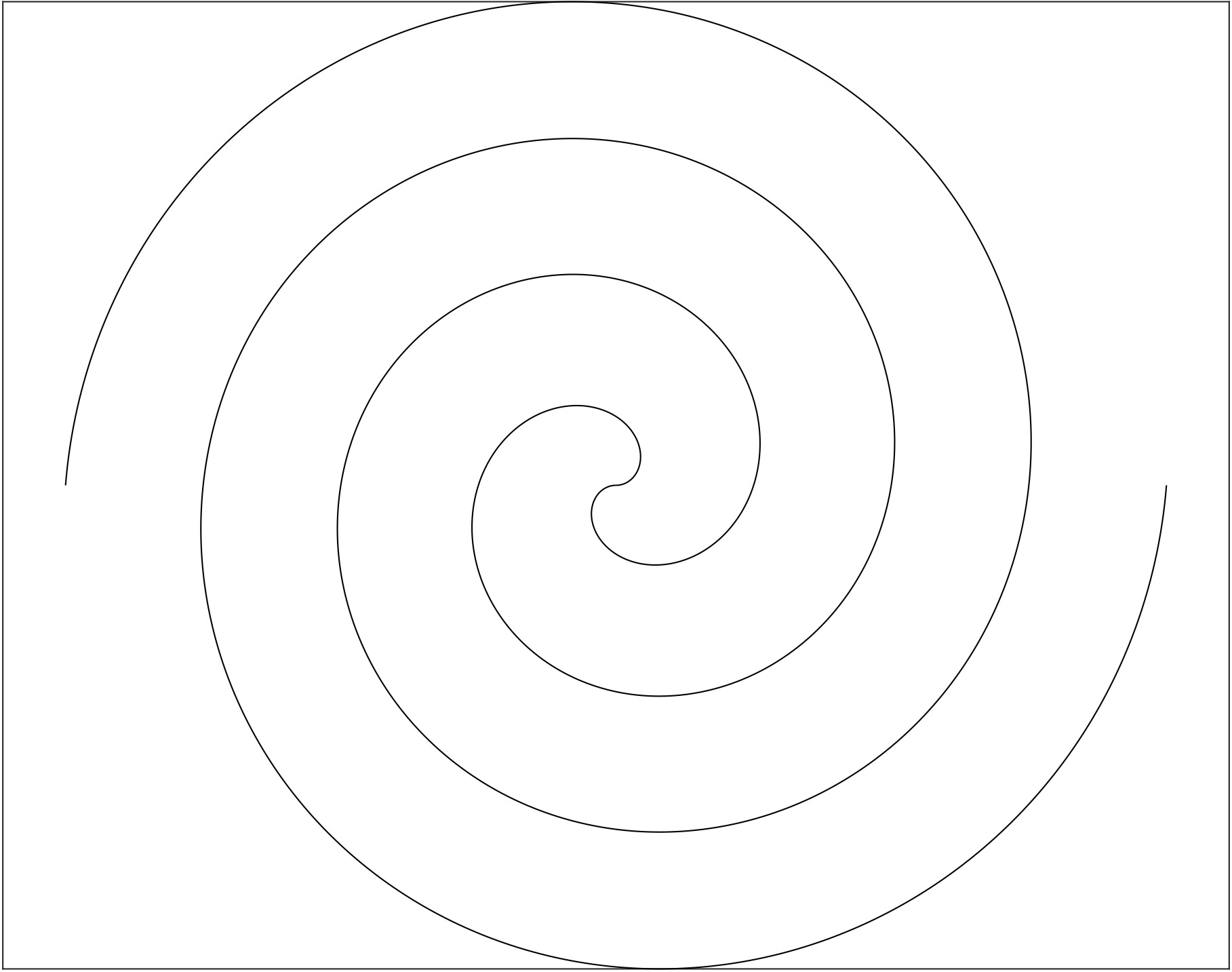}
         \caption{$n=1$}
     \end{subfigure}
     \begin{subfigure}[b]{0.24\textwidth}
         \centering
         \includegraphics[width=\textwidth]{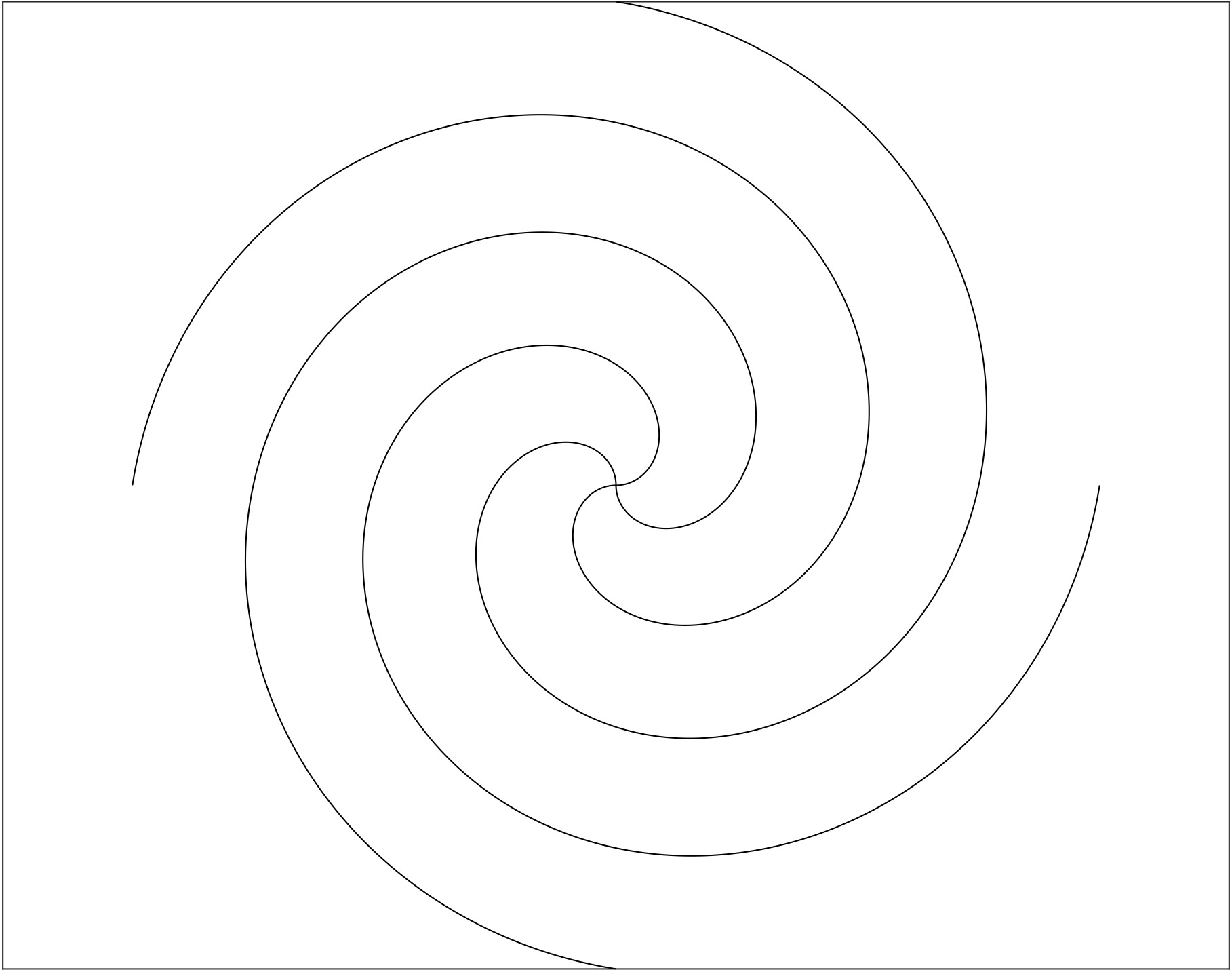}
         \caption{$n=2$}
     \end{subfigure}
     \begin{subfigure}[b]{0.24\textwidth}
         \centering
         \includegraphics[width=\textwidth]{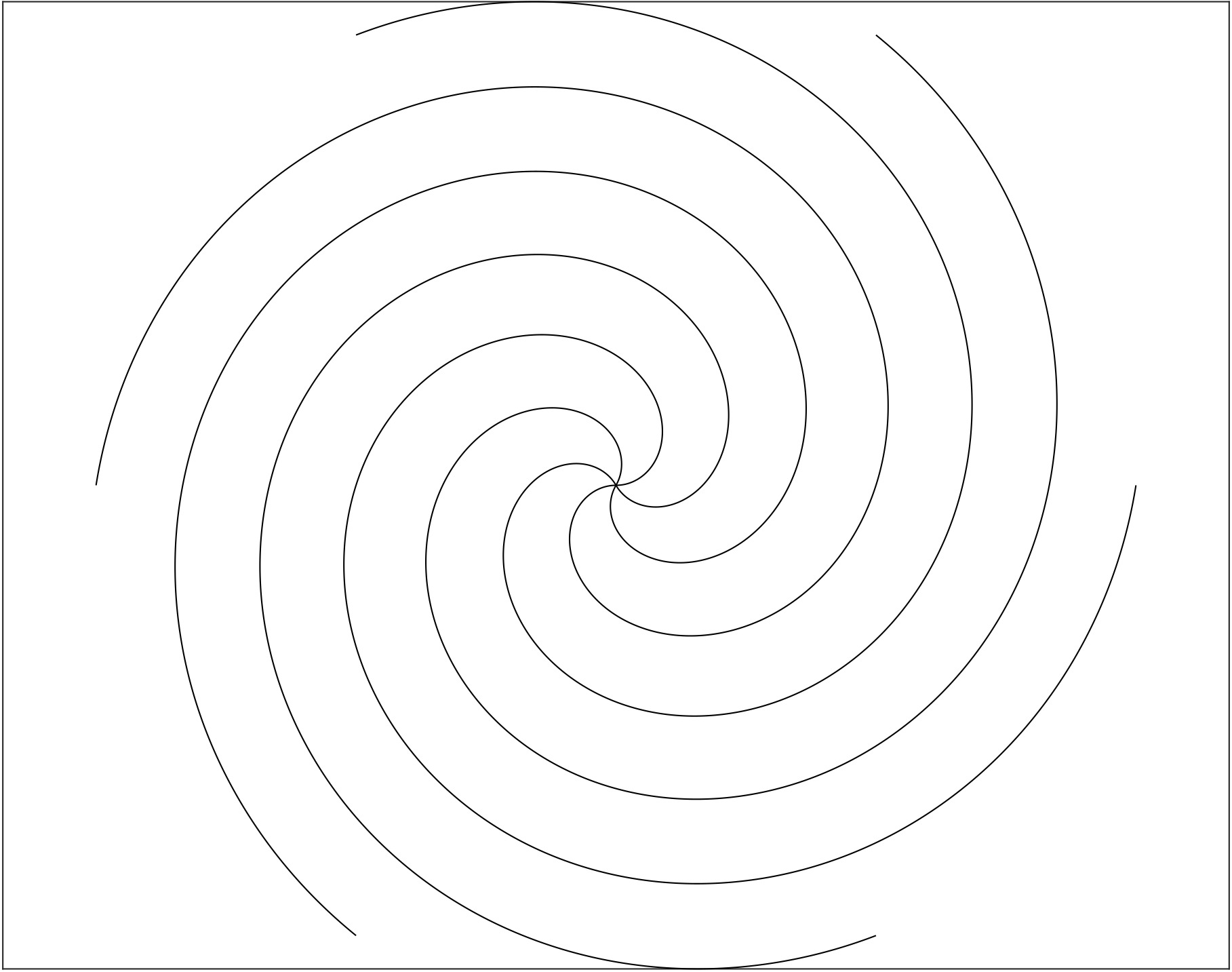}
         \caption{$n=3$}
     \end{subfigure}
     \begin{subfigure}[b]{0.24\textwidth}
         \centering
         \includegraphics[width=\textwidth]{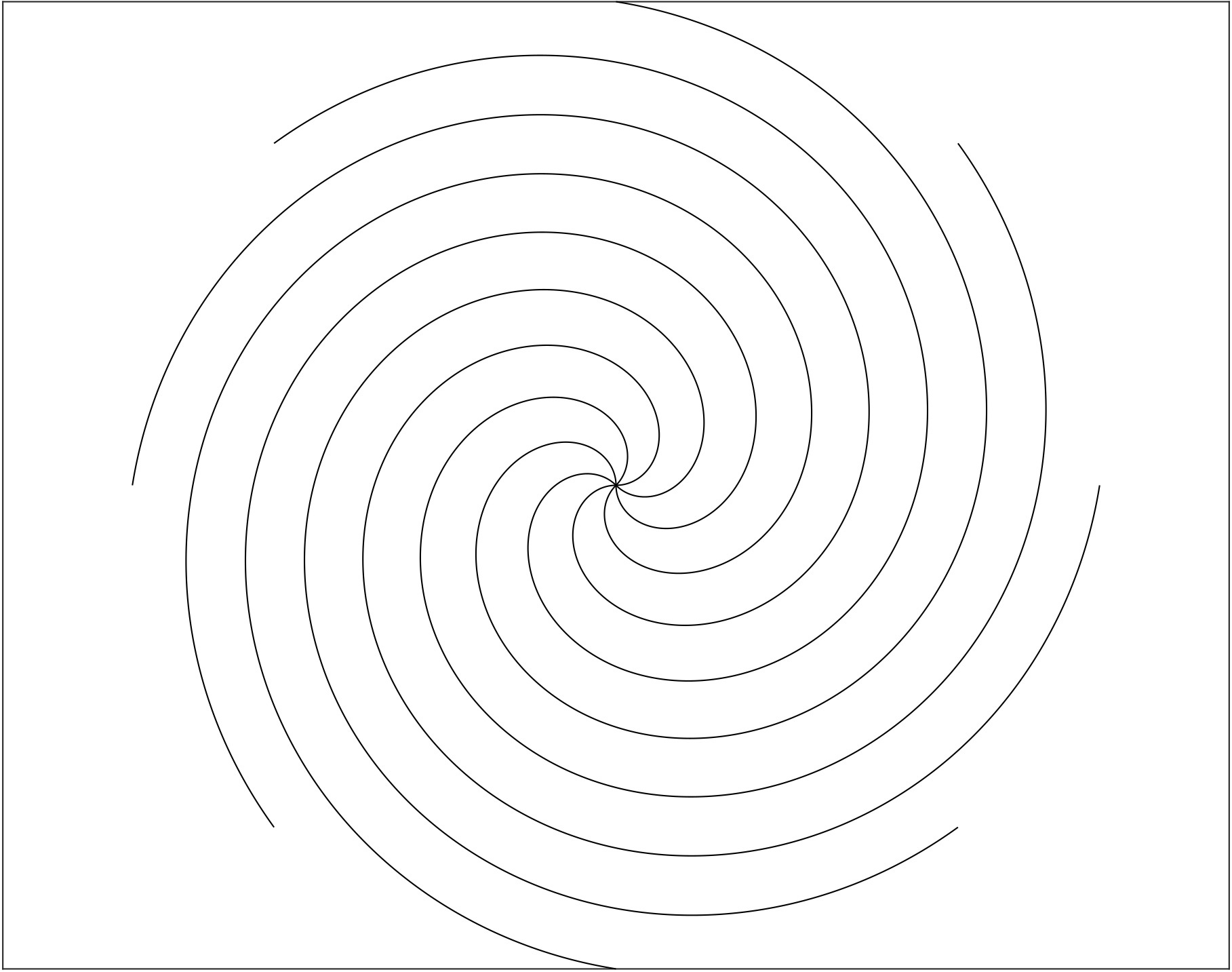}
         \caption{$n=4$}
     \end{subfigure}
        \caption{Representation of archimedian n-armed spiral waves for different winding numbers $n$.}
         \label{fig:spirals}
\end{figure}

At this point we must emphasize the role of the parameter $q$ in~\eqref{intro:eq:q} in the shape of the spiral wave $$
A(t,r,\varphi;q) =\mathbf{f}(r;q) e^{ i(\Omega t + \Theta (r;q) + n\varphi)}
$$
provided in Theorem~\ref{thm:mainintroduction}. Recall that the asymptotic wavenumber of the spiral wave is $k_*=\kappa_*(q)$ with $\kappa_*(q)$ defined in~\eqref{intro:k}. Let $A_*$ be the wave train associated to the spiral wave $A$ as in Definition~\ref{def:archimedian}. Then, from~\eqref{eq:disp},
$$
\lim_{r\to \infty} \mathbf{f}(r;q)= \sqrt{1-k_*^2}.
$$
Moreover, expression ~\eqref{intro:k} shows that $\lim_{q\to 0}\kappa_*(q)=0$, and therefore $\lim_{r\to\infty} \Theta'(r;0)=0$.
In fact, when $q=0$, that is $\alpha= \beta$ (see \eqref{intro:eq:q}), again from the dispersion equation~\eqref{eq:disp} one has that $C=1$ and $\Omega =-\beta$. In this case, the solutions of the Ginzburg-Landau equation~\eqref{eq:CGL_pol} of the form  $A(t,r,\varphi):= e^{i \Omega t}\hat{A}(r,\varphi)$ are such that $\hat{A}$ satisfies
$$
\partial_r^2 \hat{A} + \frac{1}{r} \partial_r \hat{A} + \frac{1}{r^2} \partial_{\varphi}^2 \hat{A} +\hat{A} - \hat{A} |\hat{A}|^2=0.
$$
For any $n\in \mathbb{N}$, this equation has a solution of the form $\hat{A}(r,\varphi)= \mathbf{f}(r) e^{in\varphi}$ with $\mathbf{f}(0)=0$, $\lim_{r\to \infty} \mathbf{f}(r)=1$. Indeed, the equation that $\mathbf{f}$ satisfies,
$$
\mathbf{f}'' + \frac{1}{r} \mathbf{f}' - \frac{n^2}{r^2} \mathbf{f} + \mathbf{f} - \mathbf{f}^3=0,
$$
is a particular case of the equation studied in~\cite{Aguareles2011}, proving that there exists a unique solution satisfying the conditions in Theorem~\ref{thm:mainintroduction} when $q=0$. Therefore, plotting $\re(\hat{A}(r,\varphi))$ one finds the surface depicted in the left image of Figure~\ref{fig:spirals3d}.
\begin{figure}[ht]
     \centering
     \begin{subfigure}[b]{0.48\textwidth}
         \centering
         \includegraphics[width=\textwidth]{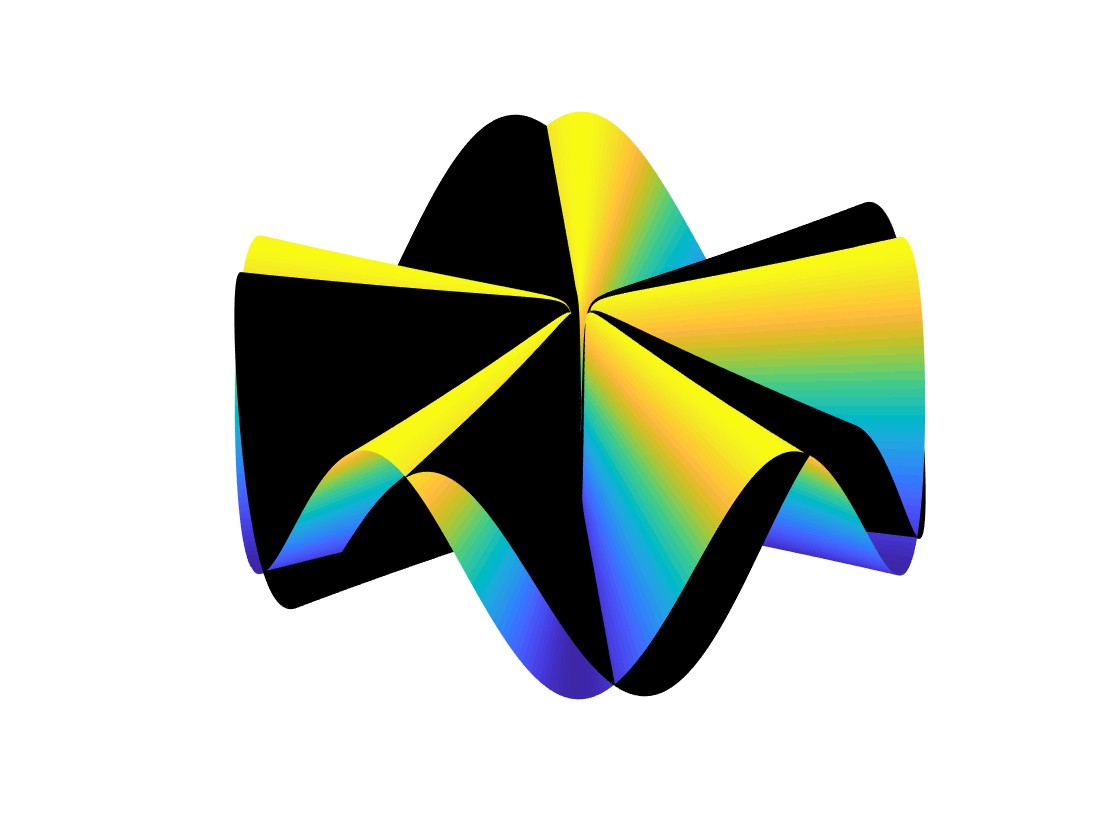}
         \caption{$q=0$}
     \end{subfigure}
     \begin{subfigure}[b]{0.48\textwidth}
         \centering
         \includegraphics[width=\textwidth]{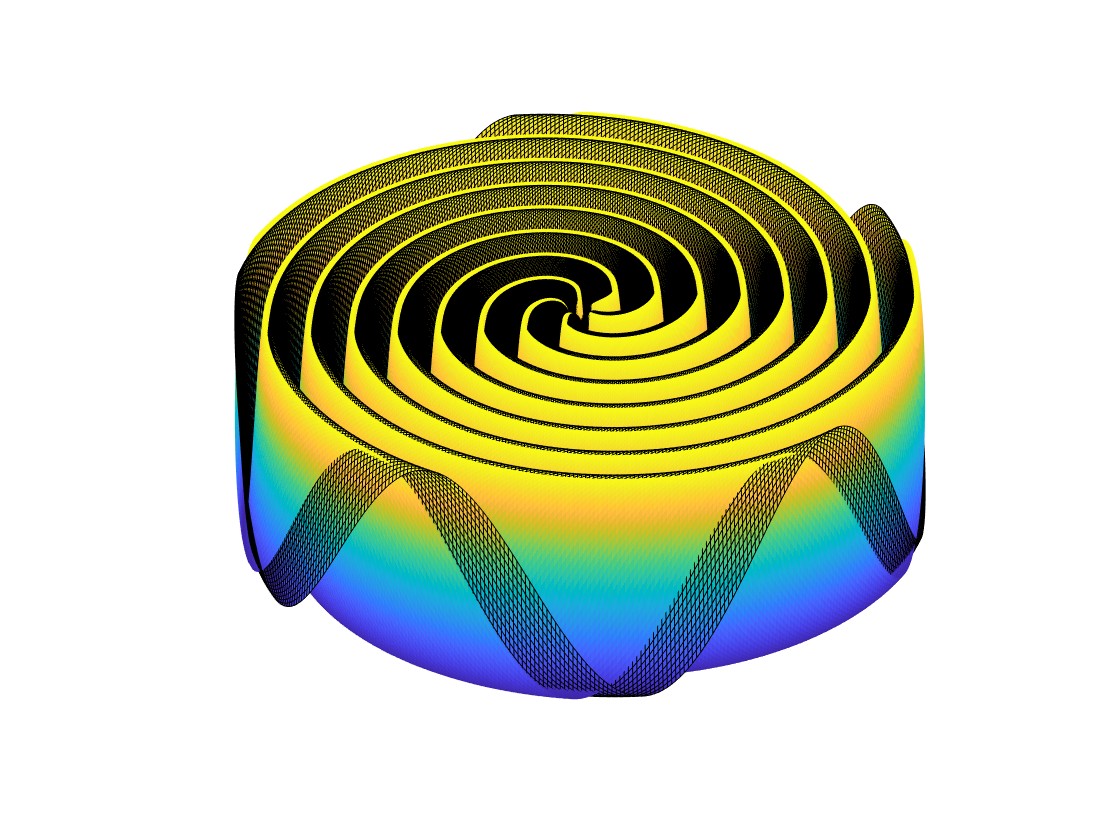}
         \caption{$q \neq 0$}
     \end{subfigure}
        \caption{Real part of the solutions of equation \eqref{eq:CGL} when $q=0$ and for $q\neq 0$. The picture shows $\big (\vec{x},\re (A(t,\vec{x}) e^{i\Omega t}) \big )$}
         \label{fig:spirals3d}
\end{figure}
We note that contour lines of $\re(\hat{A}(r,\varphi)$ are straight lines emanating from the origin.

However, if $q\neq 0$, $\Theta(r;q)$ is not constant and the contour lines bend and become the already mentioned Archimedian spirals, as the ones depicted in the right image of Figure \ref{fig:spirals3d}.
This is why $q$ is usually denoted as the \emph{twist} parameter of the spiral.

The paper is organized as follows. First in Section~\ref{sec:derivationequation} we prove that the only associated wave trains (Definition~\ref{def:wavetrain}) have a single mode (Lemma~\ref{lem:wavetrains}) and we obtain a system of ordinary differential equations that $\mathbf{f}$ and $\Theta$ have to satisfy in order for $A$, as defined in~\eqref{sol:spiral}, to be a rigidly rotating Archimedian $n$-armed spiral wave. In addition we set the boundary conditions which characterize $\mathbf{f}$ and $\Theta'$ (see Lemma~\ref{lem:newequation}). Finally, we enunciate Theorem~\ref{thm:main}, about the existence of such solutions and we prove Theorem~\ref{thm:mainintroduction} as a corollary of Theorem~\ref{thm:main}.

The rest of the paper is devoted to prove Theorem~\ref{thm:main}. First in Section~\ref{sec:main} we explain the strategy we follow to prove Theorem~\ref{thm:main} as well as
some heuristic arguments which motivate the asymptotic expression for the asymptotic wavenumer $k_*$. Section~\ref{sec:matching} is devoted to prove Theorem~\ref{thm:main} using rigorous matching methods. For that, Theorems~\ref{th:outermatching} and~\ref{th:innermatching} prove the existence of families of solutions and, finally, Theorem~\ref{thm:matchingtotal} proves the desired formula for the asymptotic wavenumber. The more technical Sections~\ref{sec:provaOuter} and~\ref{sec:proofInner} deal with the proof of Theorems~\ref{th:outermatching} and~\ref{th:innermatching} respectively.

\section{Spiral waves as solutions of ordinary differential equations}
\label{sec:derivationequation}

Next lemma characterizes the form of the possible wave train solutions of equations~\eqref{eq:CGL1d}:
\begin{lemma}\label{lem:wavetrains}
The wave trains associated to~\eqref{eq:CGL_pol} have a unique mode, namely, they are of the form
$
A(t,r)=C e^{i (\Omega t -k_* r)}
$
with $k_*\in \mathbb{R}$, and the constants  $C,\Omega \neq 0$ satisfy the relations~\eqref{eq:disp}.
\end{lemma}
\begin{proof}
Assume that $A_*(\xi)= \sum_{\ell \in \mathbb{Z}} a^{[\ell]} e^{i \ell \xi}$, $a^{[\ell]}\in \mathbb{C}$, and let $A(t,r)$ be the wave train defined through $A_*$, that is $A(t,r)=A_*(\widehat{\Omega} t - \widehat{k}_* r)$. Since $A(t,r)$ has to be a solution of~\eqref{eq:CGL1d}, we have that, for all $\ell \in \mathbb{Z}$
$$
  i \ell \widehat{\Omega} a^{[\ell]}   =
-(1+ i\alpha) \widehat{k}_*^2   \ell^2 a^{[\ell]} +    a^{[\ell]} - (1+i \beta) |A|^2 a^{[\ell]},
$$
with $|A|^2 = |A(t,r)|^2 = A(t,r) \overline{A(t,r)}$ the complex modulus. Assume that $a^{[\ell_1]}, a^{[\ell_2]} \neq 0$ for some $\ell_1, \ell_2$. Then
\begin{align*}
i \ell_1 \widehat{\Omega} &= - (1+ i\alpha) \widehat{k}_*^2 \ell_1^2 + 1 - (1+i\beta)|A|^2, \\
i \ell_2 \widehat{\Omega} &= - (1+ i\alpha) \widehat{k}_*^2 \ell_2^2 + 1 - (1+i\beta)|A|^2.
\end{align*}
This implies that
\begin{align*}
&\widehat{\Omega} \ell_1 = -\alpha \widehat{k}_*^2 \ell_1^2 -\beta |A|^2, \qquad 0 = - \widehat{k}_*^2 \ell_1^2 + 1 - |A|^2  \\
&\widehat{\Omega} \ell_2 = -\alpha \widehat{k}_*^2 \ell_2^2 -\beta |A|^2, \qquad 0 = -\widehat{k} _*^2 \ell_2^2 + 1 - |A|^2
\end{align*}
and as a consequence $0=-\widehat{k}_*^2 (\ell_1^2 - \ell_2^2)$ so, if $\widehat{k}_* \neq 0$, $\ell_1 = \pm \ell_2$.
If $\widehat{k}_* = 0$, then we have that $\widehat{\Omega} (\ell_1 - \ell_2)=0$ so that $\ell_1= \ell_2$ and we are done (recall that $\widehat{\Omega} \neq 0)$. If $\ell_1 = -\ell_2$, we deduce that
$\widehat{\Omega} \ell_1 = \widehat{\Omega} \ell_2 = -\widehat{\Omega }\ell_1$ which implies that $\ell_1=0$ and hence $A(t,r)$ is constant which is a contradiction with Definition~\ref{def:wavetrain}.
Therefore $\ell_1=\ell_2$ and $A(t,r)$ has only one mode indexed by $\ell$.
Defining $\Omega = \ell \widehat{\Omega}$ and $k_*= \ell \widehat{k}_*$ the wave train is expressed as $A(t,r)=C e^{i(\Omega t- k_* r)}$. Imposing that $A(t,r)$ is a solution of~\eqref{eq:CGL1d}, we obtain
$$
\Omega = - \alpha k_* - \beta |A|^2 , \qquad 0=-k_*^2   + 1- |A|^2.
$$
Using that $|A|= C$, we have that $C= \sqrt{1-k_*^2}$ and $\Omega= - \beta +k_*^2 (\beta - \alpha)$.
\end{proof}

We fix now $C,\Omega$ and $k_*$ such that they satisfy the relations in~\eqref{eq:disp}, namely
\begin{equation}\label{eq:disp:section2}
C^{2}= 1- k^2_*, \qquad \Omega = - \beta + k_*^2 (\beta-\alpha)
\end{equation}
and the associated wave train is
$$
A_*(\Omega t- k_* r)=C e^{i(\Omega t - k_* r)}.
$$
By Lemma~\ref{lem:wavetrains} and Definition~\ref{def:archimedian} of Archimedian spiral wave, in this paper we look for single mode spiral wave solutions of the form
\begin{equation}\label{form:spiralwavessection2}
A(t,r,\varphi)= \mathbf{f}(r;q) e^{i (\Omega t + n \varphi + \Theta(r;q))},
\end{equation}
with
$$
\lim_{r\to \infty} \mathbf{f}(r;q) = \sqrt{1- k_*^2}, \qquad \lim_{r\to \infty} \Theta'(r;q)=-k_*.
$$
\begin{remark}
By Definition~\ref{def:archimedian}, an Archimedian spiral wave associated to the wave train $A_*(\Omega t -k_*r)= C e^{i(\Omega t - k_* r)}$, satisfies
$$
A(t,r,\varphi)=A_s(r,\Omega t +n \varphi) = \sum_{\ell \in  \mathbb{Z}} a^{[\ell]}(r) e^{i \ell (\Omega t+n\varphi)} = \sum_{\ell \in \mathbb{Z}} f^{[\ell]}(r) e^{i\ell (\Omega t + n\varphi) + i\theta_{\ell}(r)}
$$
with $f^{[\ell]}(r) \geq 0$ for all $\ell \in \mathbb{Z}$,
$$
\lim_{r\to \infty} |f^{[1]}(r)-C|= \lim_{r \to \infty} |a^{[1]}(r) e^{-i\theta_1(r)}-C|=0,
$$
with $\theta_1(r)$ such that $\lim_{r\to \infty} \theta_1'(r)=-k_*$, and, for $\ell \neq 1$,
$$
\lim_{r\to \infty} a^{[\ell]}(r) = 0.
$$
The spiral waves we are looking for, that is, of the form provided in~\eqref{form:spiralwavessection2}, are the ones where $a^{[\ell]} \equiv 0$, for $\ell \neq 1$. These single mode solutions are the ones that were studied in previous works by authors~\cite{kopell,Greenberg,hagan82,AgBaSe2016}.
\end{remark}
We look for the equations that $\mathbf{f}$ and $\Theta$ have to satisfy in order for $A(t,r,\varphi)$ of the form in~\eqref{form:spiralwavessection2} to be a solution of~\eqref{eq:CGL_pol}. We recall the definition of $q$ provided in~\eqref{intro:eq:q}
\begin{equation}\label{def:q:section2}
q= \frac{\beta- \alpha}{1+ \alpha \beta}.
\end{equation}
\begin{lemma}\label{lem:newequation}
Assume that $|\alpha - \beta| < 1$. Let $\Omega\neq 0$, $k_*$ be constants satisfying~\eqref{eq:disp:section2}
and $A(t,r,\varphi;q)=\mathbf{f}(r;q) \mathrm{exp}(i (\Omega t + \Theta(r;q) \pm n\varphi ))$ for some functions $\mathbf{f}$ and $\Theta$. We introduce
$$
a=\left (\frac{1+\alpha^2}{1-\Omega \alpha}\right )^{1/2}.
$$
and
\begin{equation*}
 f(r;q)= \left ( \frac{1+ \alpha \beta}{1- \Omega \alpha}\right )^{1/2} \mathbf{f}(a r;q), \qquad
\chi(r;q) = \Theta(a r;q).
\end{equation*}

Then $A(t,r,\varphi;q)$ is a solution of~\eqref{eq:CGL_pol}
 if and only if $f$ and $v=\chi'$ satisfy the ordinary differential equations
\begin{subequations}
\label{originals}
\begin{align}
f''+\frac{f'}{r}-f\frac{n^2}{r ^2}+f(1-f^2-v^2)&=0,\label{eq:f}\\
fv'+f\frac{v}{r }+2 f'v+qf(1-f^2-k^2)&=0. \label{eq:Chi}
\end{align}
\end{subequations}
with $k \in [-1,1]$ satisfying the relations
$$
q(1-k^2) = -\frac{\Omega + \alpha}{1-\Omega \alpha}, \qquad
k_*  = \frac{k }{(1- \alpha q (1-k^2))^{1/2}}.
$$
\end{lemma}
\begin{proof}
We first note that, for $|\alpha - \beta |<1$, we have that $1+\alpha \beta >0$.
In addition, $1- \Omega \alpha >0$. Indeed,  according to \eqref{eq:disp:section2},
$$
1-\Omega \alpha = 1 -\alpha (-\beta + k_*^2 (\beta -\alpha))= 1 + \alpha \beta- \alpha \beta k_*^2 + \alpha^2 k_*^2 = 1+ \alpha \beta (1-k_*^2)+ \alpha^2 k_*^ 2.
$$
Therefore, if $\alpha \beta \geq 0$, using that $k_*<1$ (see again~\eqref{eq:disp:section2}), we have that $1-\Omega \alpha>0$.
When $\alpha \beta <0$, since $1+\alpha \beta>0$,
$$
1-\Omega\alpha =1-|\alpha \beta| (1-k_*^2) + \alpha^2 k_*^2 > 1 -|\alpha \beta| = 1+ \alpha \beta>0.
$$
Consider the rotating frame with the scalings
\begin{equation}\label{eq:BA}
B(r,\varphi)=  \delta e^{-i\Omega t} A\left (t,ar,\varphi \right ) =f(r;q) e^{i(\pm n\varphi + \chi (r;q))},
\end{equation}
where $f(r;q)= \delta \mathbf{f}(ar;q)$ and $\chi(r;q)=\Theta(ar;q)$.

Since $A$ is solution of~\eqref{eq:CGL_pol}, $B$ is a solution of
$$
\partial_r^2 B + \frac{1}{r} \partial_r B + \frac{1}{r^2}\partial_{\varphi}^2 B + a^2 \frac{1- i \Omega}{1+ i\alpha} B - \delta^{-2} a^2 \frac{1+ i\beta}{1+ i \alpha} B |B|^2 =0,
$$
or equivalently
$$
\partial_r^2 B + \frac{1}{r} \partial_r B + \frac{1}{r^2}\partial_\varphi^2 B + a^2\frac{1- \Omega \alpha - i ( \Omega + \alpha)}{  1+ \alpha^2}  B - a^2
\frac{1+\alpha \beta + i (\beta - \alpha)}{\delta^2   (1+ \alpha^2)}   B |B|^2=0.
$$
We notice that, by definition of $a$ and taking $\delta$
$$
a^2= \frac{1+ \alpha^2}{1- \Omega \alpha}, \qquad \delta ^2 = a^2\frac{1+ \alpha \beta}{1+ \alpha^2}= \frac{1+ \alpha \beta}{1- \Omega \alpha},
$$
and we denote
\begin{equation}\label{eq:q:new}
\widehat{\Omega}= - a^2 \frac{\Omega + \alpha}{ (1+ \alpha^2)} = - \frac{\Omega + \alpha}{1- \Omega \alpha}.
\end{equation}
Then, since
$$
a^2 \frac{\beta-\alpha}{\delta^2  (1+ \alpha^2 )} = \frac{\beta- \alpha}{1+ \alpha \beta}=q,
$$
the function $B$ satisfies the equation
$$
\partial_r^2 B+ \frac{1}{r} \partial_r B + \frac{1}{r^2}\partial_\varphi^2 B + (1+ \widehat{\Omega} i ) B
- (1+ q i) B |B|^2 = 0.
$$
and, substituting the form of $B$ in~\eqref{eq:BA},
we obtain that $f$ and $\chi$ satisfy the ordinary differential equations
\begin{align*}
    f'' + \frac{f'}{r} - f \frac{n^2}{r^2} + f \big (1-f^2 - (\chi')^2\big )& =0, \\
    2 f' \chi' + f \chi'' + \frac{1}{r} f \chi' + \widehat{\Omega} f - q f^3 &=0.
\end{align*}
Notice that, by~\eqref{eq:disp:section2},
\begin{equation}
    \label{eq:Omega_bar}
    \widehat{\Omega} = \frac{(\beta - \alpha)}{1-\Omega \alpha} (1-k_*^2)
\end{equation}
and then $\widehat{\Omega}$ and $q$ have the same sign as $\beta - \alpha$. Introducing $v=\chi'$ and  $k\in [-1,1]$ by the relation $\widehat{\Omega} = q (1-k^2)$, the above equations are the ones in~\eqref{originals}.

To finish, we deduce the relation between $k_*$ and $k$. First we note that, using the definition of $q$,
$$
1-\widehat{\Omega} \alpha = 1- q \alpha (1-k^2)=  \frac{1+\alpha \beta - \alpha (\beta -\alpha) (1-k^2) }{1+\alpha \beta }
=\frac{1+ \alpha^2 (1-k^2) + \alpha \beta k^2}{1+\alpha \beta }>0.
$$
Then, since
$$
\Omega =  - \frac{\alpha + \widehat{\Omega}}{ 1- \alpha \widehat{\Omega}} = - \frac{\alpha + q (1-k^2)}{1- \alpha q (1-k^2)},
$$
using that $\Omega = -\beta + k_*^2 (\beta - \alpha)$
$$
k_*^2 (\beta -\alpha)= \frac{\beta - \alpha \beta q (1-k^2) - \alpha - q (1-k^2)}{1 - \alpha q (1-k^2)} =   \frac{\beta -\alpha - q (1-k^2)(1+ \alpha \beta)}{1- \alpha q (1-k^2)}.
$$
When $\alpha \neq \beta$, by definition of $q$, we have that
$$
k_*^2 = \frac{k^2}{1- \alpha q (1-k^2)}.
$$
When $q=0$, we simply define $k=k_*$ which is consistent with the above definitions.
\end{proof}

\begin{remark}\label{rem:lamda-omega}
Spiral wave solutions of $\lambda-\omega$ systems in~\eqref{eq:lambda-omega} can be written in terms of a system of ordinary differential equations by writing the system \eqref{eq:lambda-omega} in complex form. That is, denoting $A=u_1+iu_2$, it satisfies
$$
\partial_t A = (\lambda(f) + i \omega (f)) A + \Delta  A.
$$
Then considering the change to polar coordinates $\vec{x}=(r\cos \varphi, r\sin \varphi)$ and looking for solutions of the form provided in \eqref{sol:spiral} yields the following system of ordinary differential equations:
\begin{equation}
\begin{aligned}
\label{eq:lam-om-ODE}
f''+\frac{f'}{r}-f\frac{n^2}{r^2}+f(\lambda(f)-(\chi')^2)&=0,\\
f\chi''+f\frac{\chi'}{r}+2 f'\chi'+f(\omega(f)-\Omega)&=0.
\end{aligned}
\end{equation}
The equations~\eqref{originals} correspond to equations~\eqref{eq:lam-om-ODE} in the particular case where $\lambda(z)=1-z^2$ and $\omega(z)=\Omega + q(1-k^2 -z^2)$.
\end{remark}

An important observation is that when $q=0$ (see \eqref{eq:q:new} for the definition of $q$) equation \eqref{eq:Chi} simply reads
$$
fv'+f\frac{v}{r}+2 f'v=\frac{(r f^2 v)'}{rf} = 0,
$$
and therefore $r f^2 v = \textrm{ctant}$.
Therefore, given that the solutions that we are looking for must be bounded at $r=0$, the only possible solution is therefore $v\equiv 0$.
Also, substituting in~\eqref{eq:f} one finds that
$$
f(r;0)=f_0(r),
$$
is the solution of
\begin{equation}\label{eqfosection2}
    \begin{split}
       & f_0''+\frac{f_0'}{r}-f_0\frac{n^2}{2r^2}+f_0(1-f_0^2)=0.
    \end{split}
\end{equation}
In the previous paper~\cite{Aguareles2011} (see also~\cite{AgBaSe2016}), the existence of solutions of the above differential equation was stated (in fact a more general set of differential equations was considered) under the boundary conditions
\begin{equation}\label{boundarysplit}
f_0(0)=0, \qquad \lim_{r\to \infty} f_0(r)=1,
\end{equation}
satisfying in addition
\begin{equation}\label{propf0infty:section2}
f_0(r) = 1 - \frac{n^2}{2r^2} + \mathcal{O}(r^{-4}), \qquad r\to \infty.
\end{equation}
In this new setting, Theorem~\ref{thm:mainintroduction} is a straightforward consequence of the following result which, moreover, provides a more detailed information on the constant $C_n$.

\begin{theorem}\label{thm:main} Let $n\in \mathbb{N}\cup \{0\}$. There exist  $q_0>0$ and a function $\kappa:[0,q_0] \to \mathbb{R}$ satisfying $\kappa(0)=0$, and
\begin{equation*}
\kappa(q)=  \frac{2}{q}e^{-\frac{C_n}{n^2} -\gamma } e^{-\frac{\pi}{2n|q|}} (1+\mathcal{O}(|\log q |^{-1})),
\end{equation*}
with $\gamma$ the Euler's constant and
$$
C_n=\lim_{r\to \infty} \left ( \int_{0}^r \xi f_0^2(\xi) (1-f_0^2 (\xi))\,d\xi -  n^2\log r  \right ),
$$
where $f_0$ is the solution of \eqref{eqfosection2} and \eqref{boundarysplit},
such that if $k=\kappa(q)$, then
the system~\eqref{originals} subject to the set of boundary conditions
\begin{equation*}
\begin{split}
&f(0;q)=v(0;q)=0,\\
&\lim_{r\to\infty}f(r;q)=\sqrt{1-k^2},  \qquad \lim_{r\to\infty}v(r;q)=-k,
\end{split}
\end{equation*}
has a solution.

In addition such a solution satisfies that, for $r>0$, $v(r;q)$ has constant sign, for $q$ fixed, $f(r;q)$ is an increasing function, $f(r;q)>0$ and, as a consequence, $\lim_{r\to \infty} f'(r;q)=0$.
\end{theorem}
\begin{remark} We do not need to impose the extra boundary condition $\lim_{r\to \infty} f'(r;q)=0$ which, as we will see along the proof of Theorem~\ref{thm:main}, is a consequence of imposing that the solution satisfies $\lim_{r\to \infty} (f(r;q),v(r;q))=(\sqrt{1-k^2}, -k)$.
 \end{remark}
\begin{proof}[Proof of Theorem~\ref{thm:mainintroduction} as a Corollary of Theorem~\ref{thm:main}]
We first emphasize the fact that equations~\eqref{originals} remain unaltered when $(v,q)$ is substituted by $(-v, -q)$.
Therefore one can consider $q\geq 0$ without loss of generality.
From the property~\eqref{propf0infty:section2} of $f_0$ as $r\to \infty$, it is clear that the constant $C_n \in \mathbb{R}$.

From Theorem~\ref{thm:main} and Lemma~\ref{lem:newequation} there exists a spiral wave of the form~\eqref{sol:spiral} satisfying $\lim_{r\to \infty} \mathbf{f}'(r;q)=0$, $\mathbf{f}(0;q)=\Theta'(0;q)=0$ and
$$
\lim_{r\to \infty} \mathbf{f}(r;q)=  \sqrt{1-\kappa(q)^2} \left (\frac{1-\Omega \alpha}{ 1+ \alpha \beta } \right )^{1/2}, \qquad
\lim_{r\to \infty} \Theta'(r;q)= -\kappa(q) \left ( \frac{1-\Omega \alpha }{1+\alpha^2}\right )^{1/2}.
$$
By Lemma~\ref{lem:newequation},
$$
\kappa_*(q)= \kappa(q) \left (1 - \alpha q (1-\kappa(q) \right )^{-1/2}.
$$

Since $\kappa_*(q)$ has the same first order expression as $\kappa(q)$ provided $q$ is small enough, the expression for $\kappa_*(q)$ in Theorem~\ref{thm:mainintroduction} follows from the one for $\kappa(q)$.

To guarantee that $\mathbf{f}$ and $\Theta$ satisfy the required asymptotic conditions we need to check that $k_*=\kappa_*(q)$ and $k=\kappa(q)$ satisfy
$$
1-k_*^2 = (1-k^2) \frac{1-\Omega \alpha}{ 1+ \alpha \beta }, \qquad -k_* = -k \left ( \frac{1-\Omega \alpha }{1+\alpha^2}\right )^{1/2}.
$$
Indeed, from Lemma~\ref{lem:newequation} and using definition~\eqref{def:q:section2} of $q$, we have that, if $q\neq 0$
$$
 1-k^2 =- \frac{1}{q} \frac{\alpha - \beta + k_*^2 (\beta - \alpha)}{1- \Omega \alpha} = (1-k_*^2) \frac {(1+ \alpha \beta)}{1- \Omega \alpha},
$$
and the first equality is proven.
With respect to the second one, we have to prove that
$$
(1- \Omega \alpha)(1- \alpha q (1-k^2))= 1+\alpha^2.
$$
The equality is satisfied for $\alpha=0$. When $\alpha \neq 0 $ we have to prove that
$$
0=- (\Omega + q (1-k^2)) + \alpha (\Omega q (1-k^2)-1)
= - (\Omega +\alpha) - q(1-k^2)(1- \Omega  \alpha),
$$
which from Lemma~\ref{lem:newequation} is true.

For the uniqueness of the function $\kappa_*(q)$ we use Theorem 3.1 in~\cite{kopell} and Lemma 2.1 in~\cite{AgBaSe2016}, related to $\lambda-\omega$ systems as~\eqref{eq:lam-om-ODE}, with the assumptions $\lambda(1)=0$, $\lambda'(z), \omega'(z) <0$, for $z\in (0,1]$ and $|\omega'(z)|=\mathcal{O}(|q|)$. We note that our case corresponds to $\lambda(z)=1-z^2$ and $\omega(z) = \Omega + q(1-k^2 -z^2)$ that satisfies these conditions.
The result in~\cite{AgBaSe2016} says that, if system~\eqref{eq:lam-om-ODE} has a solution with boundary conditions given by
$$
\lim_{r\to \infty}f(r) = f_\infty, \qquad \lim_{r\to \infty} f'(r)=0, \qquad \lim_{r\to \infty} v(r) = v_\infty,
$$
then $f_\infty$ is such that $\omega(f_\infty) = \Omega $ and $v_\infty^2 = \lambda(f_\infty)$.
The result in~\cite{kopell} states that there exists a unique value, $v_\infty (q)$, for $q$ small enough, such that the system~\eqref{eq:lam-om-ODE} has solution with boundary conditions
$$
\lim_{r\to \infty}f(r) = f_\infty, \qquad \lim_{r\to \infty} f'(r)=0, \qquad \lim_{r\to \infty} \chi'(r) = v_\infty (q),
$$
and $f, v$ regular at $r=0$. Applying these results to our case we obtain that $f_\infty = \sqrt{1-k^2}$ and $v_\infty=-k$ and the results in~\cite{kopell} gives the uniqueness result in Theorem~\ref{thm:mainintroduction}.
\end{proof}

After more than forty years, Theorems~\ref{thm:main} and~\ref{thm:mainintroduction} provide a rigorous proof of the explicit asymptotic expressions widely used for $k=\kappa(q)$ and $k_*=\kappa_*(q)$ as well as rigorous bounds for their relative errors.
Furthermore, the rigorous matching scheme used in this paper opens the door to showing without much extra effort the equivalent result for spiral waves in the more general setting of $\lambda-\omega$ systems.

\section{Main ideas in the proof of Theorem \ref{thm:main}}\label{sec:main}
To prove Theorem~\ref{thm:main} we need to study the existence of solutions of equations~\eqref{originals}
with boundary conditions:
\begin{equation}\label{eq:boundaryconditions}
\begin{aligned}
    &f(0;k,q)=v(0;k,q)=0, \\ &\lim_{r\to \infty } f(r;k,q)= \sqrt{1-k^2},  \qquad \lim_{r\to \infty } v(r;k,q)= -k .
    \end{aligned}
\end{equation}
The strategy of the proof is as follows. We split the domain $r\geq 0$ in two regions limited by a convenient value $r_0\gg 1$:
\begin{itemize}
    \item
A far-field (\textit{outer region}) defined as
\begin{equation}\label{eq:outerconditionint}
r\in [r_0,\infty), \quad \mbox{where} \quad     \lim_{r\to \infty} f(r;k,q)=\sqrt{1-k^2}, \qquad \lim_{r\to \infty} v(r;k,q)=-k,
\end{equation}
are the only boundary conditions that are imposed.
\item
An \emph{inner region} defined as
\begin{equation}\label{eq:innerconditionint}
r\in [0,r_0], \quad \mbox{where} \quad f(0;k,q)=v(0;k,q)=0
\end{equation}
are the boundary conditions.
\end{itemize}
The specific value of  $r_0=\frac{1}{\sqrt{2}}e^{\rho/q}$ with $\rho=(\frac{q}{|\log q|})^\frac13 $ will be explained in Section~\ref{subsec:mathching1}.

We shall obtain two families of solutions (see Theorems \ref{th:outermatching}  and \ref{th:innermatching}), depending on two free parameters $\mathbf{a},\mathbf{b} \in \R$, namely:
\begin{itemize}
\item
$f^{\mathrm{out}}(r,\mathbf{a};k,q)$, $\partial_r f^{\mathrm{out}}(r,\mathbf{a};k,q)$,
$v^{\mathrm{out}}(r,\mathbf{a};k,q)$ for the \emph{outer region} satisfying \eqref{eq:outerconditionint} and
\item
$f^{\mathrm{in}}(r,\mathbf{b};k,q)$,  $\partial_r f^{\mathrm{in}}(r,\mathbf{b};k,q)$, $v^{\mathrm{in}}(r,\mathbf{b};k,q)$ for the \emph{inner region} satisfying \eqref{eq:innerconditionint},
\end{itemize}
 which, upon matching them in the common point $r=r_0=r_0(q)$, provides a system with three equations and three unknowns $(\mathbf{a},\mathbf{b},k)$:
 \begin{equation*}
     \begin{split}
         f^{\mathrm{in}}(r_0,\mathbf{b};k,q)= & f^{\mathrm{out}}(r_0,\mathbf{a};k,q),\\
                  \partial _r f^{\mathrm{in}}(r_0,\mathbf{b};k,q)= & \partial_r f^{\mathrm{out}}(r_0,\mathbf{a};k,q),\\
                           v^{\mathrm{in}}(r_0,\mathbf{b};k,q)= & v^{\mathrm{out}}(r_0,\mathbf{a};k,q).
     \end{split}
 \end{equation*}
 Therefore, having fixed $q$, this system provides a solution  $(\mathbf{a}^*,\mathbf{b}^*, k^*)$.
 Consequently, for the value of $k=k^*$, we have a solution of system~\eqref{originals} defined for all $r\geq 0$ as:
\begin{equation}\label{def:solution:outer_inner}
(f(r;k,q),v(r;k,q)) =
 \begin{cases}
\big (f^{\mathrm{in}}(r,\mathbf{b}^*;k^*,q), v^{\mathrm{in}}(r,\mathbf{b}^*;k^*,q)\big )  & \text{if  }  r\in [0,r_0] \\
\big (f^{\mathrm{out}}(r,\mathbf{a}^*;k^*,q), v^{\mathrm{out}}(r,\mathbf{a}^*;k^*,q)\big )  & \text{if  }  r\geq r_0.
\end{cases}
 \end{equation}
 satisfying the boundary conditions
 \eqref{eq:boundaryconditions}.
 This proves the existence result in Theorem~\ref{thm:main} taking $\kappa(q)=k^*$.

Before stating the main results which provide Theorem~\ref{thm:main}, in Section \ref{sec:matching}, in the next subsection we give some intuition about how we obtain the value of $k=\kappa(q)$.

\subsection{The asymptotic expression for $k=\kappa(q)$}\label{subsec:outer_inner_variables}
One can find in the literature different heuristic arguments, based on (formal) matched asymptotic expansions techniques, which motivate the particular asymptotic expression for the parameter $k$:
\begin{equation}
\label{eq:k_heur}
k=\kappa(q)=\frac{\bar{\mu}}{q} e^{-\frac{\pi}{2nq}} (1+ \mathrm{o}(1)),
\end{equation}
with $\bar{\mu}\in \R$ a parameter independent of $q$ (see for instance \cite{hagan82}).
However, in this section we explain the particular deduction that is more consistent with the rigorous proof provided in the present work which we obtain by performing a change of parameter $k=\frac{\mu}{q} e^{-\frac{\pi}{2nq}}$ and finding the value of $\mu$ that solves the problem. Furthermore, a novelty of our proof is that it also provides that the relative error in expression \eqref{eq:k_heur} is in fact $\mathcal{O}(|\log q|^{-1})$.

We begin, as we explained at the beginning of Section~\ref{sec:main}, by looking for solutions of equations~\eqref{originals} which satisfy the boundary conditions~\eqref{eq:outerconditionint} at $r=\infty$,
which we shall denote as the \emph{outer solutions}.
We  introduce a new parameter
\begin{equation}\label{def:epsilon}
    \e=k q,
\end{equation}
and perform the scaling
\begin{equation}\label{eq:canvi_0}
R=\e r, \ V(R)=k^{-1}v(R/\e), \ F(R)=f(R/\e),
\end{equation}
to equations \eqref{originals}. We obtain
\begin{subequations}
\begin{align}
&\e^2\left(F''+\frac{F'}{R}-F\frac{n^2}{R^2}\right)+F(1-F^2-k^2V^2)=0,\label{outerF1_0}\\
&\e^2\left(V'+\frac{V}{R}+2\frac{V F'}{F}-1\right)+q^2(1-F^2)=0.
\label{outerV1_0}
\end{align}
\end{subequations}
If $\e\ne 0$ one can use the actual value of $1-F^2$ provided by equation~\eqref{outerF1_0} to recombine equations~\eqref{outerF1_0} and~\eqref{outerV1_0} to obtain the equivalent system:
\begin{subequations}
\label{outer_0}
\begin{align}
&\e ^2 \left(F''+\frac{F'}{R}-F\frac{n^2}{R^2}\right)+F(1-F^2-k^2V^2)=0,\label{outerF_0}\\
&V'+\frac{V}{R}+V^2+q^2\frac{n^2}{R^2}-1=\frac{q^2}{F}\left(F''+\frac{F'}{R}\right)-2V\frac{F'}{F}.\label{outerV_0}
\end{align}
\end{subequations}
By virtue of~\eqref{eq:outerconditionint} we look for bounded solutions of equations \eqref{outer_0} satisfying:
\begin{equation}
\label{eq:boundaryconditionsouter_0}
\lim_{R\to\infty}F(R;k,q)= \sqrt{1-k^2},  \qquad \lim_{R\to\infty}V(R;k,q)=-1.
\end{equation}
It is easy to prove (compare with Proposition~\ref{prop:dominantouter}) that the formal asymptotic expansions of bounded solutions when $R\to \infty$ satisfy
\begin{equation}\label{eq:asymptoticsinfinityR_0}
\begin{split}
F(R;k,q)&\sim \sqrt{1-k^2}-\frac{k^2}{2R\sqrt{1-k^2}}+\O\left (\frac{\e^2}{R^2}\right ),\qquad  \mbox{as} \ R\to \infty, \\
V(R;k,q)&\sim -1-\frac{1}{2R}+\O\left (\frac{\e^2}{R^2}\right ),\qquad  \mbox{as} \ R\to \infty .
\end{split}
\end{equation}
We note that equation \eqref{outerF_0} is singular in $\e$. In particular, if $\e=0$, and therefore $k=0$ (recall \eqref{def:epsilon}), either $F=0$, which is a trivial solution we are not interested in, or  $1-F^2(R)=0$, which also gives a non interesting solution.
But, if we write equation \eqref{outerF_0} as
\[
\e ^2 \left(F''+\frac{F'}{R }\right)+F\left (-\frac{\e ^2  n^2}{R^2}+1-F^2-k^2V^2\right )=0,
\]
we observe that the asymptotic expansions~\eqref{eq:asymptoticsinfinityR_0} suggest that the terms $\e^2F'/R$ and $\e^2F''$ are of higher order in $k$, and therefore in $\e$, than the rest.
Therefore we will take as first approximation the solution of
\[
-\frac{\e ^2  n^2}{R^2}+1-F^2-k^2V^2=0,
\]
which gives our candidate to be the main part of the outer solution we are looking for:
\begin{equation}\label{eq:F0_0}
F_0(R)=F_0(r;k,q)=\sqrt{1-k^2V_0^2(R;q)-\e^2 \frac{n^2}{R^2}}.
\end{equation}
Then, neglecting again  the terms of order depending on $F'$ and $F''$ in equation \eqref{outerV_0}, a natural definition for $V_0$ is the solution of the Ricatti equation
\begin{equation}\label{Ricatti_0}
V_0'+\frac{V_0}{R}+V_0^2+q^2\frac{n^2}{R^2}-1=0,
\qquad  \mbox{such that} \qquad \lim _{R\to \infty}V_0(R;q)=-1.
\end{equation}
Observe that the boundary condition for $V_0$ gives:
\[
\lim _{R\to \infty}F_0(R;k,q)=\sqrt{1-k^2},
\]
as expected.

A solution of~\eqref{Ricatti_0} is given by (see, for instance~\cite{abramowitz})
\begin{equation}\label{eq:V0_0}
V_0(R;q) = \frac{K_{inq}'(R)}{K_{inq}(R)},
\end{equation}
with $K_{inq}$ the modified Bessel function of the first kind.
It is a well known fact that (see~\cite{abramowitz}),
\begin{equation*}
K_{\nu}(R) =\sqrt{\frac{\pi}{2R}} e^{-R} \left (1+ \mathcal{O}(R^{-1})\right ),   \qquad  \mbox{as} \ R\to \infty,
\end{equation*}
for any $\nu\in \C$, where $\mathcal{O}(R^{-1})$ is uniform as $\nu\to 0$.
Therefore the functions $(F_0,V_0)$ satisfy the boundary conditions \eqref{eq:boundaryconditionsouter_0}.

We go back to our original variables through the scaling \eqref{eq:canvi_0} and define:
\begin{equation}\label{dominanttermouter_0}
\fdout(r;k,q)=F_0(\e r;k,q)=F_0(kqr;k,q),\qquad \vdout(r;k,q)=kV_0(\e r;q)=kV_0(kqr;q),
\end{equation}
which satisfy
\begin{equation}\label{eq:boundary1_0}
    \lim_{r\to \infty} \vdout(r;k,q) = -k,\qquad \lim_{r\to \infty} \fdout(r;k,q) = \sqrt{1-k^2}.
\end{equation}
The precise properties of the dominant terms $\fdout,\vdout$ will be exposed in Proposition~\ref{prop:dominantouter}.

An important observation if $r\gg 1 $ but $kr$ is small enough, is that the function $\vdout(r;k,q)$ has the following asymptotic expansion (a rigorous proof of this fact will be done in Proposition~\ref{prop:dominantouter}, see~\eqref{vdoutmaching_new}):
\begin{equation*}
\vdout(r;k,q) = -\frac{n}{r} \mathrm{tan} \left (nq \log r + nq \log kq +\frac{\pi}{2} -\theta_{0,nq}\right ) [1+ \mathcal{O}(q^2)] ,
\end{equation*}
with $\theta_{0,nq}=\mathrm{arg} \left(\Gamma(1+inq)\right)=-\gamma nq + \mathcal{O}(q^2)$, $\Gamma$ is the Euler's Gamma function, and $\gamma$ the Euler's constant.

We now deal with the \emph{inner solutions} of~\eqref{originals} departing the origin
satisfying $f(0;k,q)=v(0;k,q)=0$.
For moderate values of $r$, the \emph{inner problem} is perturbative with respect to the parameter $q$. For that reason, to define the dominant term of the inner solutions we first consider the case $q=0$.
Let us now recall that
in~\cite{AgBaSe2016} it was proven that, when $q=0$, system~\eqref{originals} has a solution $(f,v)$ with boundary conditions~\eqref{eq:boundaryconditions} if and only if $k=k(0)=0$.
In this case, $v=v(r;0,0)=0$ and $\f0(r)=f(r;0,0)$ satisfies the boundary conditions~\eqref{boundarysplit} and the second order differential equation~\eqref{eqfosection2}, that is:
\begin{equation}\label{eq:introf0_0}
\f0''+\frac{\f0'}{r}-\f0\frac{n^2}{r^2}+\f0(1-\f0^2)=0,\qquad \f0(0)=0, \qquad \lim_{r\to\infty}\f0(r) = 1.
\end{equation}
As we already mentioned, the existence and properties of $\f0$ were studied in the previous work~\cite{Aguareles2011}.

As $v(r;0,0)\equiv 0$, we write  $v(r;k,q)=q \esv \vb(r;k,q)$ so the system~\eqref{originals} reads
\begin{align*}
f''+\frac{f'}{r}-f\frac{n^2}{r^2}+f(1-f^2-q^2\esvq \vb^2)&=0,\\
f\vb'+f\frac{\vb}{r}+2\vb f'+ \esvi f(1-f^2-k^2)&=0.
\end{align*}
Let us now consider $(\f0(r),\v0(r;k))$, the unique solution of this system when $q\esvq=0$ satisfying ~\eqref{eq:introf0_0}
and
\begin{equation}\label{eq:introv0_0}
 \v0'+\frac{\v0}{r} + 2 \v0 \frac{\f0'}{\f0} + \esvi (1-\f0^2-k^2)=0,\qquad
\v0(0;k)=0.
\end{equation}
In~\cite{Aguareles2011} it was proven that
 $\f0(r)>0$, for $r>0$ and $\f0(r)\sim \alpha_0 r^n$, as $r\to 0$,
thus, the function
\begin{equation}\label{expr:v_0}
\v0(r;k) = \frac{1}{r \f0^2(r)} \int_{0}^r \xi \f0^2 (\xi)(1- \f0^2(\xi)-k^2)\dd \xi,
\end{equation}
satisfies~\eqref{eq:introv0_0} and $\v0(0;k)=0$.
We then define the functions, whose properties are stated in Proposition~\ref{prop:f0v0_new}
\begin{equation}\label{dominantterminner_0}
\fdin(r)=\f0(r),\qquad  \vdin(r;k,q) = q \v0(r;k).
\end{equation}

In Proposition~\ref{prop:f0v0_new} it will be proven that, if $r\gg 1 $ but $kr$ is small enough, the function $\vdin(r;k,q)$ has the following asymptotic expansion, see~\eqref{prop:v0innermatching}:
\begin{equation}\label{prop:v0innermatchingheu}
\vdin(r;k,q)=-q\frac{n^2(1+k^2)}{r} \log r + \frac{qC_n}{r} -\frac{k^2q}{2}r +   q\mathcal{O}(r^{-3}\log r) + qk^2 \mathcal{O}(r^{-1}),
\end{equation}
with $C_n$ defined in Theorem~\ref{thm:main}.

We emphasize that we expect the functions $\vdout$ and $\vdin$ to be the first order of the functions $v^{\rm{out}}$ and $v^{\rm{in}}$ in the outer and inner domains of $r$.  Therefore,  a natural request is that they ``coincide up to first order" in some large enough intermediate point, $r_0$, such that $kr_0$ and $q\log r_0$ are still small enough quantities.
With these hypotheses and using the previous asymptotic expansion~\eqref{prop:v0innermatchingheu} we obtain:
\[
\vdin (r_0;k,q)=\frac{q}{r_0}\left[
-n^2\log r_0+C_n+ HOT\right],\]
where the terms in $HOT$ are small provided $kr_0$ is small. With respect to $\vdout$, using that  $\theta_{0,nq}=-\gamma nq+ \mathcal{O}(q^2)$, we have that
\[
\vdout(r_0;k,q) =\frac{q}{r_0}\left[ -\frac{n}{q} \mathrm{tan} \left (nq \log r_0 + nq \log kq +\frac{\pi}{2} + nq \gamma +\mathcal{O}(q^2)\right ) [1+ \mathcal{O}(q^2)] \right].
\]
Observe that if $nq \log kq +\frac{\pi}{2}=\mathcal{O}(q)=mq$, upon Taylor expanding the tangent function one obtains:
\[
\vdout(r_0;k,q) =-\frac{q}{r_0}\left[ n^2 \log r_0 + n
m + n^2 \gamma + HOT \right]
\]
and then it is possible to have $\vdout(r_0)-\vdin(r_0)=0$ because the ``large" term
$n^2 \log r_0$ is canceled.

The last observation of this section is that taking
$kq=\mu e^{-\frac{\pi}{2nq}}$ gives $nq \log kq +\frac{\pi}{2}=nq \log \mu =\mathcal{O}(q)$. For this reason, during the proof of Theorem \ref{thm:main} in the rest of the paper, we will rewrite the parameter $k$ using this expression:
\begin{equation}\label{def:formak}
 kq=\mu e^{-\frac{\pi}{2nq}},
\end{equation}
and we will prove that, for $q$ small enough, there exists a value of $\bar{\mu}$ independent of $q$ such that, for $k$ given by \eqref{def:formak} with $\mu = \bar{\mu} + \mathcal{O}(|\log q|^{-1})$,
\eqref{originals} has a solution satisfying the required asymptotic conditions~\eqref{eq:boundaryconditions}.

\section{Proof of Theorem \ref{thm:main}: Matching argument}\label{sec:matching}

In order to prove Theorem~\ref{thm:main} following the strategy explained in Section~\ref{sec:main}, we provide the precise statements about the existence of the families of solutions $(f^{\mathrm{out}},v^{\mathrm{out}})$ in the \emph{outer region}~\eqref{eq:outerconditionint} (Section~\ref{subsec:outer}) and $(f^{\mathrm{in}}, v^{\mathrm{in}})$  in the \emph{inner region}~\eqref{eq:innerconditionint} (Section~\ref{subsec:inner}). Moreover, since our method relies on finding  $(f^{\mathrm{out}},v^{\mathrm{out}})$  and $(f^{\mathrm{in}}, v^{\mathrm{in}})$
near the dominant terms $(f_0^{\mathrm{out}}, v_0^{\mathrm{out}})$ and $(f_0^{\mathrm{in}}, v_0^{\mathrm{in}})$,  given in \eqref{dominanttermouter_0} and \eqref{dominantterminner_0} respectively, we set all the properties of these dominant terms in Proposition~\ref{prop:dominantouter} and~\ref{prop:f0v0_new} respectively.
After that, in Sections~\ref{subsec:mathching1} and~\ref{sec:constants}, the rigorous matching of the dominant terms is done. Finally in Section~\ref{sec:matchingfinal}, we finish the proof of Theorem~\ref{thm:main}.

The modified Bessel functions $I_\nu, K_\nu$, see~\cite{abramowitz}, play an important role in our proofs.
From now on we shall use that for any $\nu\in \C$, there exists $z_0>0$ (see~\cite{abramowitz}), such that
\begin{equation}\label{asymKIn_new}
K_\nu(z)= \sqrt{\frac{\pi}{2 z}}e^{-z} \left (1+ \frac{4\nu^2 -1}{8z}+\O\left (\frac{1}{z^2}\right )\right ) ,\quad
I_\nu(z)= \sqrt{\frac{1}{2\pi z}}e^{z} \left (1+ \O\left (\frac{1}{z}\right ) \right ),
\qquad |z|\geq z_0 ,
\end{equation}
where, for $|\nu| \leq \nu_0$ the $\O(1/z)$ terms are bounded by $\frac{M}{|z|}$ for $|z|\geq z_0$ and $M,z_0$ only depend on $\nu_0$. In addition, when $\nu \in \mathbb{N}$,
\begin{equation}
\label{asymKIn0_new}
    K_\nu(z) = \mathcal{O}(z^{-\nu}), \qquad I_\nu (z) = \mathcal{O}(z^{\nu}), \qquad |z| \to 0,
\end{equation}
where, again, $\O(z^{\nu})$ is uniform for $\nu\leq \nu_0$.

From now on we denote by $M$ a constant independent on $q,k$ that can (and will) change its value along the proofs. In addition when the notation $\O(\cdot)$ is used, it means that the terms are bounded uniformly everywhere the function is studied.

\subsection{Outer solutions}\label{subsec:outer}

We begin the proof of Theorem~\ref{thm:main}, studying the dominant terms $f_0^{\mathrm{out}}, v_0^{\mathrm{out}}$ defined in~\eqref{dominanttermouter_0}in the \emph{outer region} (see~\eqref{eq:outerconditionint}).
\begin{proposition}\label{prop:dominantouter}
For any $0<\mu_0<\mu_1$, there exists $q_0=q_0(\mu_0,\mu_1)>0$ such that for any
$\mu\in [\mu_0,\mu_1]$ and $q\in (0,q_0]$, the functions
$v_0^{\mathrm{out}}(r;k,q) $ and $f_0^{\mathrm{out}}(r;k,q)$ defined in~\eqref{dominanttermouter_0}
with $k=\mu e^{-\frac{\pi}{2nq}}$, satisfy the following properties:
\begin{enumerate}
    \item There exists $R_0>0$ such that for $kqr\geq R_0$,
\begin{equation}\label{asymptoticandbound:V0_new}    \begin{aligned}
v_0^{\mathrm{out}}(r;k,q)&= -k -\frac{1}{2qr} + k \O\left (\frac{1}{(kq r)^{2}}\right ), \\
f_0^{\mathrm{out}}(r,k,q) &= \sqrt{1-k^2} \left (1 - \frac{k}{2qr(1-k^2)}\right ) + \O\left (\frac{1}{(qr)^2}\right ).
\end{aligned}
\end{equation}
\item \label{item2dominantouter}
For $2  e^{-\frac{\pi}{2nq}}\leq kqr\leq (qn)^2 $, we have:
\begin{equation}\label{vdoutmaching_new}
\vdout(r;k,q) = -\frac{n}{r} \mathrm{tan} \left (nq \log r + nq \log \left (\frac{\mu}{2}\right ) -\theta_{0,nq}\right ) [1+ \mathcal{O}(q^2)] ,
\end{equation}
with $\theta_{0,nq}=\mathrm{arg} \left(\Gamma(1+inq)\right)=-\gamma nq + \mathcal{O}(q^2)$ where $\Gamma$ is the Euler's Gamma function and $\gamma$ the Euler's constant.
\item \label{item3dominantouter}
For $2 e^{-\frac{\pi}{2nq}}\leq kqr$,
we have:
\begin{equation}\label{propertiesvout}
\partial_r v_0^{\mathrm{out}}(r;k,q)>0, \quad  v_0^{\mathrm{out}}(r;k,q) <-k, \quad
\partial_r \fdout(r;k,q)>0.
\end{equation}
\item \label{item4dominantouter}
Let $\alpha\in (0,1)$. There exists $\bar{q}_0=\bar{q}_0(\alpha,\mu_0,\mu_1)$ and a constant $M=M(\alpha,\mu_0,\mu_1)>0$ such that if $r_{\min}$ satisfies $2e^{2} e^{-\frac{\pi}{2qn}} \leq kq r_{\min} \leq (kq)^\alpha$ then, for $r\geq r_{\min}$,
$v^{\mathrm{out}}_0$ satisfies:
\begin{equation}\label{morepropertiesvout}
\begin{aligned}
&|v^{\mathrm{out}}_0(r;k,q)|, \quad  |r \partial_r v^{\mathrm{out}}_0(r;k,q)|, \quad |r^2 \partial_r^2 v^{\mathrm{out}}_0(r;k,q)| \leq M r_{\min}^{-1},\\
&|r(\vdout(r;k,q)+k)|, \quad |r^2 \partial_r \vdout(r;k,q)|,\quad  |r^3 \partial_r^2 \vdout(r;k,q)|\leq M q^{-1}.
\end{aligned}
\end{equation}
With respect to $\fdout$, we have;
\begin{equation}\label{morepropertiesfout}
\begin{aligned}
&\fdout(r;k,q)\geq 1/2, \quad |r^2 \partial_r \fdout(r;k,q)|,\quad  |r^3 \partial_r^2 \fdout(r;k,q)|\leq Mq^{-1} r_\m^{-1},
\\
&|1-\fdout(r;k,q)|,\quad |r \partial_r \fdout(r;k,q)|, \quad  |r^2 \partial_r^2 \fdout(r;k,q)|\leq M r_\m^{-2}.
\end{aligned}
\end{equation}
\end{enumerate}
\end{proposition}
The proof of this proposition is postponed to Appendix~\ref{sec:dominantouter} and it involves a careful study of some properties of the Bessel functions $K_{inq}$.

Once $(\fdout,\vdout)$ are studied, we look for solutions in the \emph{outer region} satisfying boundary conditions~\eqref{eq:outerconditionint}.
This is the contain of the following theorem which gives the existence and bounds of a one parameter family of solutions of equations~\eqref{originals},
which stay close to the approximate solutions $(\fdout(r;k,q),\vdout(r;k,q))$ given in~\eqref{dominanttermouter_0}
for all $r\geq r_2$, being $r_2$ any number such that $r_2=\O(\varepsilon^{\a-1})$ with {$0<\a<1$ satisfying that $q^{-1}\e ^{1-\alpha}\to 0$} when $q\to 0$.
\begin{theorem} \label{th:outermatching}
For any  $\eta >0$, $0<\mu_0<\mu_1$, there exist $q_0=q_0(\mu_0,\mu_1,\eta)>0$, $e_0=e_0(\mu_0,\mu_1,\eta)>0$ and $M=M(\mu_0,\mu_1,\eta)>0$ such that,
 for any $\mu \in [\mu_0,\mu_1]$ and $q\in[0,q_0]$ if we take $\e=\mu e^{-\frac{\pi}{2nq}}$
  and $\alpha\in(0,1)$
satisfying
\begin{equation}\label{cond:thouter0}
q^{-1}\e ^{1-\alpha}<e_0,
\end{equation}
taking $r_2$ as
\begin{equation}\label{cond:r2}
r_2 =  \e^{\a-1},
\end{equation}
and $\Cout$ satisfying
\begin{equation}
    \label{cond:a}
|\Cout|\leq \eta r_2^{-3/2}e^{r_2 \sqrt{2}}\, ,
\end{equation}
equations~\eqref{originals} have a family of solutions $(\fout(r,\Cout;k,q),\vout(r,\Cout;k,q))$
defined for $r \geq r_2$
which are of the form
\begin{equation}\label{def:fvout}
\begin{split}
&\fout(r,\Cout;k,q)=\fdout(r;k,q)+\gout(r,\Cout;k,q), \\ &\vout(r,\Cout;k,q)=\vdout(r;k,q)+\wout(r,\Cout;k,q).
\end{split}
\end{equation}
where $\fdout,\vdout$ are defined in~\eqref{dominanttermouter_0}. The functions $\gout,\wout$ satisfy
$$
|r^2 \gout(r,\Cout;k,q)|,\, |r^2 \partial\gout(r,\Cout;k,q)|\le  M ,  \qquad
|r^2 \wout(r,\Cout;k,q)|\le   M q^{-1} (\eta + q^{-1} \e^{1-\alpha}).
$$
We can also  decompose
\begin{equation}
\label{def:gout}
\gout(r,\Cout;k,q)= K_0(r\sqrt{2}) \Cout+\gout_0(r;k,q)+\gout_1(r,\Cout;k,q),
\end{equation}
where $K_0$ is the modified Bessel function of the first kind (\cite{abramowitz}), and $\gout_0(r;k,q)$ is an explicit function independent of $\eta$.
Moreover,
\begin{itemize}
    \item[(i)]
    there exists $q_0^*=q_0^*(\mu_0, \mu_1)>0$, and $M_0=M_0(\mu_0,\mu_1)$ such that, for $q\in[0,q_0^*]$,
    \begin{equation}\label{eq:fitagout}
        |r^2 \gout_0(r;k,q)|,\,  |r^2 \partial \gout_0(r;k,q)|\leq M_0  \e^{1-\a}q^{-1}  ,
    \end{equation}
    \item[(ii)]
    and for $q\in[0,q_0]$,
    \begin{equation} \label{eq:fitagout_bis}
        |r^2 \gout_1(r, \Cout;k,q)|,\,  |r^2 \partial \gout_1(r, \Cout;k,q)|\leq M_1\e^{1-\a}q^{-1} e^{-r_2 \sqrt{2}} r_2^{3/2} |\Cout |,
    \end{equation}
    where $M_1=M_1(\mu_0,\mu_1,\eta)$ depends on $\mu_0,\mu_1$, and $\eta$.
\end{itemize}
With respect to $\wout$, it can be decomposed as $\wout = \wout_0 + \wout_1$ satisfying that for $q\in [0,q_0]$
$$
|r^2 \wout_0(r,\Cout;k,q)|   \leq M_2 q^{-1} e^{-r_2 \sqrt{2}} r_2^{3/2} |\Cout |
, \qquad
|r^2 \wout_1 (r,\Cout;k,q)| \leq M_2 \e^{1-\alpha}q^{-2} ,
$$
with  $M_2=M_2(\mu_0,\mu_1,\eta)$.
\end{theorem}

Theorem \ref{th:outermatching} is proved in Section~\ref{sec:provaOuter}
by performing the scaling \eqref{eq:canvi_0} and studying the solutions of the outer equations \eqref{outer_0} with boundary conditions \eqref{eq:boundaryconditionsouter_0} near the functions $F_0,V_0$ given in~\eqref{eq:F0_0} and~\eqref{eq:V0_0}. The proof is done through  a fixed point argument in a suitable Banach space.

We emphasize that given that when $r\to \infty$, $\gout$ and $\wout$ have limit zero, and $\fdout$ and $\vdout$ satisfy~\eqref{eq:boundary1_0}, then $(\fout,\vout)$ satisfy the boundary conditions~\eqref{eq:boundaryconditions}.
With this result in mind we now proceed with the study of the behaviour of solutions of~\eqref{originals} departing $r=0$, also called \emph{inner solutions}.

\subsection{Inner
solutions}\label{subsec:inner}
We now deal with the families of solutions of~\eqref{originals} departing the origin, satisfying the boundary condition $f(0)=v(0)=0$  that are defined for values of $r$ in the \emph{inner region} (see~\eqref{eq:innerconditionint}).

We first set the properties of $\f0^{\mathrm{in}}, \v0^{\mathrm{in}}$, the dominant terms in the \emph{inner region} defined in~\eqref{dominantterminner_0}, that will mostly be used throughout this proof.
\begin{proposition}\label{prop:f0v0_new}
For any $0<\mu_0<\mu_1$, there exists $q_0=q_0(\mu_0,\mu_1)>0$ such that for any $\mu\in [\mu_0,\mu_1]$ and $ q\in [0,q_0]$, the functions $\f0^{\mathrm{in}}(r), \v0^{\mathrm{in}}(r;k,q)$ defined in~\eqref{dominantterminner_0} with $kq=\mu e^{-\frac{\pi}{2nq}}$, satisfy the following properties:
\begin{enumerate}
\item \label{item1dominant_inner}
For all $r>0$ we have $\f0^{\mathrm{in}}(r), \partial_r \f0^{\mathrm{in}}(r)>0$ and there exists $c_f>0$ such that:
\begin{equation*}
\begin{aligned}
 & \f0^{\mathrm{in}}(r)\sim c_f r^n, \quad r\to 0,\qquad
\f0^{\mathrm{in}}(r) = 1 - \frac{n^2}{2r^2} + \O(r^{-4}),\quad r\to \infty,\\
&\partial_r \f0^{\mathrm{in}}(r)\sim nc_fr^{n-1},\quad r\to 0,
\quad \partial_r \f0^{\mathrm{in}}(r) = \frac{n^2}{r^3} + \O(r^{-5}),\quad r\to \infty.  &&
\end{aligned}
\end{equation*}
\item
For  $0<r\leq \frac{n}{k\sqrt{2}}$, $\v0^{\mathrm{in}}(r;k,q)<0$ and there exists a positive function $c_v(k)=c_v^0+ \O(k^2)$ such that
\begin{equation*}
\begin{aligned}
&\v0^{\mathrm{in}}(r;k,q)\sim -qc_v (k)r ,\quad r\to 0, \qquad
|\v0^{\mathrm{in}}(r;k,q)|\leq Mq \frac{|\log r|}{r},\quad 1\ll r< \frac{n}{k\sqrt{2}},\\
&\partial_r \v0^{\mathrm{in}}(r;k,q) \sim -q c_v(k),\quad r\to 0, \qquad
|\partial_r \v0^{\mathrm{in}}(r;k,q)|\leq M q \frac{\log r}{\esv r^2},\quad 1\ll r< \frac{n}{k\sqrt{2}}.
\end{aligned}
\end{equation*}
\item
For $1\ll r \leq \frac{n}{k\sqrt{2}}$, we have that  %
\begin{equation}\label{prop:v0innermatching}
\vdin(r;k,q)=-q\frac{n^2(1+k^2)}{r} \log r + \frac{qC_n}{r} -\frac{k^2q}{2}r +   q\mathcal{O}(r^{-3}\log r) + qk^2 \mathcal{O}(r^{-1}),
\end{equation}
with $C_n$ defined in Theorem~\ref{thm:main} and
$$
\partial_r \vdin(r;k,q) = q \frac{n^2}{r^2} \log r + q \O(r^{-2}).
$$
\end{enumerate}
\end{proposition}
The proof of this proposition is referred to Appendix~\ref{sec:inner_dominant} and mostly relies on previous works~\cite{Aguareles2011} and~\cite{AgBaSe2016}.

The following theorem, whose proof is provided in Section~\ref{sec:proofInner}, states that there exists a family of solutions of~\eqref{originals}, satisfying the boundary conditions at the origin, which remains close to the approximate solutions $(\fdin(r),\vdin(r;k,q))$ given in~\eqref{dominantterminner_0}, for all $r\in[0,r_1]$, being $r_1=\O(e^{\rho/q})$ for some $\rho>0$ small enough.
\begin{theorem}\label{th:innermatching}
For any $\eta >0$, $0<\mu_0<\mu_1$, there exist $q_0=q_0(\mu_0,\mu_1,\eta)>0$,
$\rho_0=\rho_0(\mu_0,\mu_1,\eta)>0$ and $M=M(\mu_0,\mu_1,\eta)>0$ such that
for any $\mu \in [\mu_0, \mu_1]$, $q\in[0,q_0]$
and
\begin{equation}\label{eq:rho0th}
\rho\in(0, \rho_0),
\end{equation}
taking $\e=\mu e^{-\frac{\pi}{2nq}}$,
$r_1$ as
\begin{equation}\label{cond:r1}
r_1=\frac{e^{\rho/q}}{\sqrt{2}},
\end{equation}
and $\Cin$ satisfying
\begin{equation}\label{cond:b}
|\Cin| r_1^{3/2} e^{\sqrt{2} r_1}\leq \dfrac{\eta}{(\sqrt{2})^{3/2}} q^2 (\log \sqrt{2}r_1)^2= \dfrac{\eta}{(\sqrt{2})^{3/2}} \rho^2,
\end{equation}
the system~\eqref{originals} has a family of solutions $(\fin(r,\Cin;k,q),\vin(r,\Cin;k,q))$
defined for $r\in [0,r_1]$ satisfying boundary conditions \eqref{eq:innerconditionint}, that is,
$\fin(0,\Cin;k,q)=\vin(0,\Cin;k,q)=0$.
Moreover, these functions satisfy:
\begin{equation}
\label{def:fvin}
\fin(r,\Cin;k,q)=\fdin(r)+\gin(r,\Cin;k,q),\qquad \vin(r,\Cin;k,q)= \vdin(r;k,q)+\win(r,\Cin;k,q),
\end{equation}
with $\fdin,\vdin$ defined in~\eqref{dominantterminner_0}. The functions
$\gin$, $\win$ satisfy for all $r\in [0,r_1]$
\begin{equation*}
 \big | \gin (r,\Cin;k,q)\big |\leq M q^2,  \qquad  \big |\win (r,\Cin;k,q)\big |\leq M q^3 ,
\end{equation*}
for $0\leq r <1$
\begin{align*}
&\big | \gin(r,\Cin;k,q) \big | \leq Mq^2 r^{n}, \qquad \big | \partial \gin(r,\Cin;k,q)  \big | \leq M q^2r^{n-1}, \\
&\big | \win(r,\Cin;k,q) \big | \leq M q^3 r  ,  \qquad \;
 \big | \partial \win(r,\Cin;k,q) \big | \leq  M q^3 ,
\end{align*}
and for $1\ll r \leq r_1$
\begin{equation*}
 \big | \gin(r,\Cin;k,q)  \big | \leq M q^2     \frac{|\log r|^2}{r^2}, \qquad
 \big | \win(r,\Cin;k,q) \big | \leq  M q^3 \frac{|\log r|^3}{r}.
\end{equation*}
In addition, there exists a function $I$ satisfying
\begin{equation}\label{propertyImathing}
\begin{split}
&I'(r_1\sqrt{2})K_n(r_1\sqrt{2}) - I(r_1\sqrt{2}) K_n'(r_1\sqrt{2})= \frac{1}{r_1 \sqrt{2}}, \\ &|I(r_1\sqrt{2})|,|I'(r_1\sqrt{2})| \leq M_I \frac{1}{\sqrt{r_1}} e^{r_1 \sqrt{2}},
\end{split}
\end{equation}
for some constant $M_{I}$, and where $K_n$ is the modified Bessel function of the first kind  (\cite{abramowitz}), such that
\begin{equation}
\label{def:gin}
    \gin(r,\Cin;k,q)= I(r\sqrt{2})\Cin+\gin_0(r;k,q)+\gin_1(r,\Cin;k,q),
\end{equation}
where $\gin_0(r;k,q)$ is an explicit function which is independent of $\eta$. Also, for $1\ll r \leq r_1$,
\begin{itemize}
    \item[(i)] there exists $q_0^*=q_0^*(\mu_0, \mu_1)>0$, and $M_0(\mu_0,\mu_1)$ such that, for $q\in[0,q_0^*]$,
    \begin{equation}\label{eq:fitagin}
        |\gin_0(r;k,q)|, |\partial \gin_0(r;k,q)| \leq M_0 q^2    \frac{|\log r|^2}{r^2},
    \end{equation}
    \item[(ii)] and for $q\in[0,q_0]$,
    \begin{equation*}
       |\gin_1(r,\Cin;k,q)|, |\partial \gin_1(r,\Cin;k,q)| \leq M_1 q^2
       \rho^2   \frac{|\log r|^2}{r^2},
    \end{equation*}
    where $M_1=M_1(\mu_0,\mu_1,\eta)$ depends on $\mu_0,\mu_1$, and $\eta$.
\end{itemize}
\end{theorem}

\subsection{Matching point and matching equations}\label{subsec:mathching1}

Observe that, given $0<\mu_0<\mu_1$, the results of Theorems \ref{th:outermatching} and \ref{th:innermatching} are valid for any value of $k$ of the form $k=\frac{\e}{q}=\frac{\mu}{q}e^{-\frac{\pi}{2nq}}$, $\mu\in [\mu_0,\mu_1]$ and $q$ small enough.
To finish the proof of  Theorem \ref{thm:main} we need to select the value of $\mu$, and therefore of $k$, which connects an outer  solution (given by  a particular value of $\Cout$ ) with an inner one (given by a particular value of $\Cin$).
To this end we need to have a non-empty matching region, for which we shall impose $r_2=r_1$, that is to say, $\e^{\a -1}=e^{\rho/q}/\sqrt{2}$.
Then, using that $\e=\mu e^{-\frac{\pi}{2qn}}$, one obtains
\begin{equation}\label{def:alpha}
\a=\alpha(\rho, \mu, q)=1-\frac{2n\rho}{\pi}\frac{1-\frac{q\ln(\sqrt{2})}{\rho}}{1-\frac{2nq\log(\mu)}{\pi}}.
\end{equation}
But, according to Theorem~\ref{th:outermatching}, it is also required that $\e^{1-\a}/q < e_0\ll 1$, which is equivalent to imposing that $q,\rho$ satisfy:
$$
q|\ln(e_0q \sqrt{2})|< \rho\, .
$$
Therefore, fixing any $\eta>0$, since by \eqref{eq:rho0th}, $0<\rho <\rho_0$, the condition for $q,\rho$ becomes:
\begin{equation}
    \label{ineq:rho_q}
q|\ln(e_0q / \sqrt{2})|< \rho <\rho_0\, .
\end{equation}
We rename
\begin{equation}
    \label{def:r0}
    r_0:=r_1=r_2=\frac{e^{\rho/q}}{\sqrt{2}}=\e^{\alpha -1}=\mu ^{\alpha-1}e^{\frac{\pi (1-\alpha)}{2qn}}\, ,
\end{equation}
and we take
\begin{equation}\label{def:rho}
\rho=\left (\frac{q}{|\log q|} \right )^{1/3},
\end{equation}
which satisfies the required inequalities in~\eqref{ineq:rho_q}.
Therefore Theorems~\ref{th:outermatching} and~\ref{th:innermatching} are in particular valid when taking $\alpha$ and $\rho$ as given in  \eqref{def:alpha} and \eqref{def:rho}, and $r_1=r_2$ as given in \eqref{def:r0}, since all these values satisfy conditions \eqref{cond:thouter0}, \eqref{cond:r2}, \eqref{eq:rho0th}, and \eqref{cond:r1}, if we take any $\Cout$ and $\Cin$ satisfying \eqref{cond:a}, \eqref{cond:b}, provided $q_0=q_0(\mu_0,\mu_1, \eta)$ is small enough (we take the minimum of both theorems).

Once we have chosen the parameters $\rho$ and $\alpha$ and the value of the matching point $r_0$, the next step is to prove that there exist $\Cout,\Cin,k$ or equivalently, since $k=\e/q=\mu e^{-\frac{\pi}{2qn}}$, $\Cout,\Cin,\mu$,  such that, for $q$ small enough,
\begin{equation}\label{eq:machingreal}
\begin{split}
\fout(r_0,\Cout;k,q)=& \fin(r_0,\Cin;k,q), \qquad \partial _r \fout(r_0,\Cout;k,q)=\partial_r \fin(r_0,\Cin;k,q),\\
 \vout(r_0,\Cout;k,q)=& \vin(r_0,\Cin;k,q).
\end{split}
\end{equation}
We stress that the existence results, Theorems~\ref{th:outermatching} and~\ref{th:innermatching}, depend on the set of constants $\mu_0,\mu_1,\eta$ that are not defined yet.
We shall fix them, in Section~\ref{sec:constants}, as follows:
\begin{itemize}
\item
First, we match the explicit dominant terms of the outer functions $\fout$, $\vout$,  (see \eqref{def:fvout} and \eqref{def:gout})  with dominant terms of the inner functions $\fin$, $\vin$ (see \eqref{def:fvin} and \eqref{def:gin}):
\begin{equation}\label{eq:firstordermatching}
\begin{aligned}
K_0(r_0\sqrt{2}) \Cout_0+\fdout(r_0;k,q)+\gout_0(r_0;k,q) &= I(r_0\sqrt{2}) \Cin_0+ \fdin(r_0) + \gin_0(r_0;k,q),
\\ \vdout(r_0;k,q) &= \vdin(r_0;k,q)
\end{aligned}
\end{equation}
and
\begin{equation}\label{eq:firstordermatchingder}
\begin{aligned}
    \sqrt{2}  K_0'(r_0\sqrt{2}) \Cout_0 +  \partial_r \fdout(r_0;k,q)+ &\partial_r \gout_0(r_0;k,q) \\ &= \sqrt{2} I'(r_0\sqrt{2}) \Cin_0 + \partial_r \fdin(r_0) + \partial_r \gin_0(r_0;k,q)
\end{aligned}
\end{equation}
This is done in Section \ref{sec:constants}, where, in Proposition \ref{prop:matching0}  we find $\Cout_0,\Cin_0$ and $\bar{\mu}$ such that, taking the approximate value of $k=\bar{\mu} q^{-1} e^{-\frac{\pi}{2qn}}$,  equations~\eqref{eq:firstordermatching} and~\eqref{eq:firstordermatchingder} are solved.
Moreover we fix two values $0<\mu_0<\mu_1$ such that, $\bar{\mu}\in [\mu_0,\mu_1]$.

\item The obtained solutions $\Cout_0,\Cin_0$ satisfy conditions \eqref{cond:a} and \eqref{cond:b} for a particular value of $\eta$. We will use these values, $\mu_0,\mu_1,\eta$ in Theorems~\ref{th:outermatching} and~\ref{th:innermatching} to obtain families of solutions $\fout,\vout$, $\fin, \vin$ of equations \eqref{originals}.
\end{itemize}

Finally, the existence of the constants $\Cout$, $\Cin$ and $\mu$ (that will be found to be close to $\Cout_0$, $\Cin_0$, $\bar{\mu}$) satisfying the matching conditions~\eqref{eq:machingreal} is provided
by means of a Brouwer's fixed point argument in Section~\ref{sec:matchingfinal} (see Theorem~\ref{thm:matchingtotal}).

\subsection{Matching the dominant terms: setting the constants $\mu_0,\mu_1,\eta $.} \label{sec:constants}
As we explained in the previous section, the purpose of this section is to choose the constants $\mu_0,\mu_1,\eta$ which appear in Theorems~\ref{th:outermatching} and~\ref{th:innermatching} to obtain the families of solutions $\fout,\vout$, $\fin, \vin$ of equations~\eqref{originals} satisfying the suitable boundary conditions.

Next proposition gives the existence of solutions of equations \eqref{eq:firstordermatching} and~\eqref{eq:firstordermatchingder}.
\begin{proposition}\label{prop:matching0}
Take $\mu_0= e^{-\frac{C_n}{n^2} - \gamma }$, $\mu_1 =3  e^{-\frac{C_n}{n^2} - \gamma }
$, where $C_n$ and $\gamma$ are given in Theorem~\ref{thm:main}.
Then, there exists $q_1^*=q_1^*(\mu_1\,\mu_2)$ and $\hat{M}(\mu_!,\mu_2)$ such that for $0<q<q_1^*$, equations~\eqref{eq:firstordermatching} and~\eqref{eq:firstordermatchingder} have a solution $(\Cout_0,\Cin_0,\bar \mu)$ satisfying:
$$
\bar  \mu \in [\mu_0,\mu_1], \quad |\Cout_0 | \leq \hat{M} \rho^2 r_0^{-3/2} e^{r_0 \sqrt{2}}, \qquad |\Cin_0 |\leq \hat{M} \rho^2 r_0^{-3/2} e^{-r_0 \sqrt{2}},
$$
where $\rho$ is given in \eqref{def:rho}.
\end{proposition}
\begin{proof}
We first note that, by definitions of $\rho$ and $r_0$ in \eqref{def:r0} and \eqref{def:rho}, respectively, we have
$$
|nq\log r_0 +nq\log (\mu/2) - \theta_{0,nq}|=\mathcal{O}(\rho) =\mathcal{O} \left (q/|\log q| \right )^{1/3} \ll 1.
$$
Then, using the asymptotic expressions~\eqref{vdoutmaching_new} and~\eqref{prop:v0innermatching}
for $\vdout$ and $\vdin$ at $r=r_0$ and recalling that $k=\e /q=\bar{\mu} q^{-1} e^{-\frac{\pi}{2nq}}$, we have that
\begin{align} \label{v0inoutmatching}
\vdin(r_0;k,q) &-\vdout(r_0;k,q) =
-qn^2 \frac{1+k^2}{r_0}\log r_0 +q \frac{C_n}{r_0} - q\frac{k^2}{2} r_0 \notag \\ &+  \frac{n}{r_0} \left (nq \log r_0 + n q \log \left (\frac{\bar{\mu}}{2}\right )-\theta_{0,nq}\right )  \notag
\\
& +q\mathcal{O}\left (\frac{ {\log r_0}}{r_0^{3}}\right ) + qk^2 \mathcal{O}(r_0^{-1}) +
\frac{1}{r_0}\mathcal{O}\left (\left |nq \log r_0 +n q \log \left (\frac{\bar{\mu}}{2}\right ) - \theta_{0,nq}\right |^3,q^2\right ) \notag
\\=
& -\frac{n^2 k^2 \rho}{r_0} + \frac{q}{r_0} \left (C_n+n^2 \log \left( \frac{\bar{\mu}}{2}\right) -n\theta_{0,nq}q^{-1} \right ) - q  \frac{k^2}{2}r_0 + q^3\mathcal{O}\left (\frac{ { (\log r_0)^3 }}{r_0}\right )\notag \\ & + \frac{1}{r_0} \mathcal{O}(q^2, q k^2) \notag \\
=& \frac{q}{r_0} \left (C_n+n^2   \log \left (\frac{\bar{\mu}}{2}\right ) - n \theta_{0,nq} q^{-1}\right) +{
\frac{q}{r_0} \mathcal{O}(|\log q|^{-1})}.
\end{align}
Therefore, the only possibility for $\bar{\mu}$ to solve $\vdin(r_0;k,q)-\vdout(r_0;k,q)=0$ is that
$$
C_n + n^2  \log \left (\frac{\bar{\mu}}{2}\right ) - n\theta_{0,nq}q^{-1} = \mathcal{O}(|\log q|^{-1}) \Longleftrightarrow  \bar{\mu} \\
= 2 e^{-\frac{C_n}{n^2} - \gamma + \mathcal{O}(|\log q|^{-1})},
 $$
 where we have used that $\theta_{0,nq}=-\gamma nq + \mathcal{O}(q^2)$, or equivalently
$$
\bar{\mu}= 2 e^{-\frac{C_n}{n^2} - \gamma}\left(1+\mathcal{O}(|\log q |^{-1})\right).
$$
This last equality suggests that the parameter $\bar{\mu}$ has to belong to $[\mu_0,\mu_1]$ with, for instance
\begin{equation}\label{defmu01:matching}
\mu_0= e^{-\frac{C_n}{n^2} - \gamma },\qquad \mu_1 =3  e^{-\frac{C_n}{n^2} - \gamma }.
\end{equation}

For any $\bar{\mu}\in [\mu_0,\mu_1]$, we introduce now the (independent of $\eta$) function
\begin{equation}\label{def:Delta0}
\Delta_0(r;k,q) =\fdin(r)-\fdout(r;k,q) + \gin_0(r;k,q)-\gout_0(r;k,q).
\end{equation}
Then $\Cout_0,\Cin_0$ satisfying~\eqref{eq:firstordermatching} and~\eqref{eq:firstordermatchingder} are given by
\begin{equation}\label{fixedpointequationab}
\left (\begin{array}{c} \Cout_0 \\ \Cin_0 \end{array}\right ) =\frac{1}{d(r_0)} \left ( \begin{array}{c}
I'(r_0\sqrt{2}) \Delta_0 (r_0;k,q) -
\frac{1}{\sqrt{2}}I(r_0 \sqrt{2}) \Delta'_0(r_0;k,q) \\
K'_0(r_0\sqrt{2}) \Delta_0(r_0;k,q) - \frac{1}{\sqrt{2}}K_0(r_0\sqrt{2}) \Delta'_0(r_0;k,q)
\end{array} \right ),
\end{equation}
with
$ d(r_0)=K_0(r_0\sqrt{2}) I'(r_0\sqrt{2}) - K_0'(r_0\sqrt{2})I(r_0\sqrt{2})$.

We first notice that by property~\eqref{propertyImathing} of the function $I$ and using the asymptotic expansion~\eqref{asymKIn_new} for $K_0(r)$ and $K_n(r)$ for $r\gg 1$, there exists $\hat{M}_1$ a constant such that
\begin{equation}\label{prop1matchingf}
0< \frac{1}{d(r_0)} = r_0\sqrt{2} \left (1 + \mathcal{O}\left (\frac{1}{r_0}\right )\right ) \leq r_0 \sqrt{2} + \hat{M}_1.
\end{equation}

Now we estimate $\Delta_0$. We first note that, by estimate~\eqref{vdoutmaching_new} of $\vdout$, if $q$ is small enough,
$$
|\vdout (r_0;k,q)|\leq \hat{M}_2 \frac{\rho}{r_0} \leq \frac{1}{4},
$$
with $\hat{M}_2$ a constant that only depends on $\mu_0,\mu_1$. Then, by item~\ref{item1dominant_inner} of Proposition~\ref{prop:f0v0_new} along with the  definition~\eqref{dominanttermouter_0}  of $\fdout$, we have that, for $q$ small enough,
\begin{align*}
\left | \fdin(r_0)-\fdout(r_0;k,q) \right |  & \leq \left | 1 - \frac{n^2}{2r_0^2} - \sqrt{1- (\vdout(r_0;k,q))^2 - \frac{n^2}{r_0^2}}\right | + \left | \fdin(r_0)-1 + \frac{n^2}{2r_0^2} \right | \\
& \leq  \hat{M}_3 |\vdout(r_0,k)|^2 + \frac{\hat{M}_4}{r_0^4}  \leq \hat{M}_5 \frac{\rho^2}{r_0^2}.
\end{align*}
The constant $\hat{M}_5$ only depends on $\mu_0,\mu_1$.
Therefore, by bounds \eqref{eq:fitagout} and \eqref{eq:fitagin} in Theorems~\eqref{th:outermatching} and~\eqref{th:innermatching}
\begin{equation}\label{prop2mathchingf}
\begin{aligned}
|\Delta_0(r_0;k,q)|&\leq |\fdin(r_0)-\fdout(r_0;k,q)| + |\gin_0(r_0;k,q)| + |\gout_0(r_0;k,q)| \\ &
\leq \hat{M}_5 \frac{\rho^2}{r_0^2} +
 M_0 q^2 \frac{|\log r_0|^2}{r_0^2} + M_0 \frac{\e^{1-\a}}{q r_0^2} \\ & \leq \hat{M}_6 \frac{\rho^2}{r_0^2},
\end{aligned}
\end{equation}
where we have used that
\begin{equation*}
 r_0^{-1}=\e^{1-\a}= e^{-\rho/q}=e^{-1/(q^{2/3}|\log(q)|^{1/3})}= \mathcal{O}(q^\ell),\quad \textrm{for any $\ell> 0$}.
\end{equation*}
Moreover, since, as established in Theorems~\ref{th:outermatching} and~\ref{th:innermatching}, for $0<q\leq q^*_0(\mu_0,\mu_1)$, $M_0$ only depends on $\mu_0,\mu_1$, again, the same happens to $\hat{M}_6$.
Analogously, one can check that, if $0<q\leq q^*_0(\mu_0,\mu_1)$, then
\begin{equation}\label{prop3matchingf}
|\partial \Delta_0(r_0;k,q)|\leq \hat{M}_7 \frac{\rho^2}{r_0^2}.
\end{equation}

By using estimates~\eqref{prop1matchingf}, \eqref{prop2mathchingf} and~\eqref{prop3matchingf}, the estimates~\eqref{propertyImathing} of $I$ and that, if $r\gg 1$, one has  $|K_0(r\sqrt{2})|, |K_0'(r\sqrt{2})| \leq M_K e^{-r\sqrt{2}} r^{-1/2}$,
we have that, as $k=\bar{\mu} q^{-1}e^{-\frac{\pi}{2nq}}$ with $\bar{\mu}\in [\mu_0,\mu_1]$, the solution  $(\Cout_0,\Cin_0)$ of~\eqref{fixedpointequationab} has to satisfy, for $q$ small enough,
\begin{align*}
|\Cout_0| &\leq \rho^2 \frac{1}{r_0^{3/2}}e^{r_0 \sqrt{2}} (\sqrt{2} + \hat{M}_1 r_0^{-1}) M_I\left [ \hat{M}_6 + \frac{1}{\sqrt{2}} \hat{M}_7\right ],
\\ |\Cin_0|& \leq \rho^2 \frac{1}{r_0^{3/2}} e^{-r_0\sqrt{2}} (\sqrt{2} + \hat{M}_1 r_0^{-1}) M_K \left [ \hat{M}_6 + \frac{1}{\sqrt{2}} \hat{M}_7\right ].
\end{align*}
Taking $q$ small enough, $M_1 r_0^{-1} \leq \sqrt{2}$ and, defining
$$
\hat{M}=2 \sqrt{2} \left [\hat{M}_6 + \frac{1}{\sqrt{2}}\hat{M}_7\right ]\max\{ M_I,M_K\} ,
$$
we conclude that there exist $q_1^*=q_1^*(\mu_1\,\mu_2)$ and $\hat{M}(q_1^*)$ such that for $0<q<q_1^*$,
$$
|\Cout_0 | \leq \hat{M} \rho^2 r_0^{-3/2} e^{r_0 \sqrt{2}}, \qquad |\Cin_0 |\leq \hat{M} \rho^2 r_0^{-3/2} e^{-r_0 \sqrt{2}},
$$
where $\rho$ is given in \eqref{def:rho}.
\end{proof}
We stress that, since $r_0=r_1=r_2$, the constants $\Cout_0,\Cin_0$, provided by proposition \ref{prop:matching0}
satisfy the conditions \eqref{cond:a} and \eqref{cond:b} in Theorems~\eqref{th:outermatching} and~\eqref{th:innermatching}
for any  $\eta \geq (\sqrt{2})^{3/2}\hat{M}$.
Recalling that $\hat{M}$ only depends on $\mu_0$ and $\mu_1$, we may set now
\begin{equation}\label{defeta:matching}
\eta = 2 \hat{M}.
\end{equation}
Proposition \ref{prop:matching0} provides good candidates to be approximate values for the  solutions $\Cout,\Cin,\mu$ of the matching equations~\eqref{eq:machingreal}. In particular they set the constants $\mu_0,\mu_1,\eta$ in~\eqref{defmu01:matching} and~\eqref{defeta:matching}.

Since $\Cout_0,\Cin_0$ have different sizes, for technical reasons we define the scaled constants $\hat{\Cout}_0, \hat{\Cin}_0$  as
$$
\hat{\Cout}_0= \Cout_0\,e^{-r_0 \sqrt{2}}r_0^{3/2} , \qquad
\hat{\Cin}_0 = \Cin_0\, \rho^{-2} e^{r_0 \sqrt{2}}r_0^{3/2},
$$
and we observe that they satisfy
\begin{equation}\label{eq:ab0}
|\hat{\Cout}_0|\leq \frac{\eta}{2} \rho^2 \leq \frac{\eta}{2},\qquad  |\hat{\Cin}_0|\leq \frac{\eta}{2}.
\end{equation}

\subsection{Matching the outer and inner solutions: end of the proof of Theorem \ref{thm:main} }
\label{sec:matchingfinal}

The main goal of this section is to obtain the parameters $\Cout$, $\Cin$ and $\mu$ which solve the matching equations~\eqref{eq:machingreal}.
Having solved these equations, which is the content of next Theorem \ref{thm:matchingtotal}, we have a value of $\mu$, and therefore of $k$ as defined in \eqref{def:formak}, for which the original system~\eqref{originals} has a solution $(f,v)$ satisfying the required boundary conditions~\eqref{eq:boundaryconditions}. Once this result is proven, in order to prove Theorem~\ref{thm:main} it will only remain to check that $f$ is a positive increasing function and that $v<0$ (see Proposition~\ref{prop:fincreasing} below).

We begin our construction by considering the families of solutions provided by Theorems~\ref{th:outermatching} and~\ref{th:innermatching} for the constants $\mu_0,\mu_1 ,\eta$, fixed in the previous section (Section~\ref{sec:constants}) and any values
$\Cout$ and $\Cin$ satisfying \eqref{cond:a} and \eqref{cond:b}. Namely, we consider $\mu\in [\mu_0,\mu_1]$, $\eta$, $r_0$, $\rho$ and $\a$ as given in~\eqref{defmu01:matching}, \eqref{defeta:matching}, \eqref{def:r0}, \eqref{def:rho}, and~\eqref{def:alpha} respectively, and $q\in [0,q_0]$. Along this section we call $q_0$ the minimum value provided by all the previous results, that is
Propositions~\ref{prop:dominantouter}, \ref{prop:f0v0_new},  and~\ref{prop:matching0} and Theorems~\ref{th:outermatching}, \ref{th:innermatching}.

Next theorem gives the desired result:
\begin{theorem}\label{thm:matchingtotal}
Take $\mu_0= e^{-\frac{C_n}{n^2} - \gamma }$,
$ \mu_1 =3  e^{-\frac{C_n}{n^2} - \gamma }$,
where $C_n$ and $\gamma$ are given in Theorem~\ref{thm:main} and $\eta$ as given in~\eqref{defeta:matching}.
Then, there exists $q^*$ such that for $q \in [0,q^*]$
equations~\eqref{eq:machingreal} have a solution $\Cout(q)$, $\Cin(q)$, $k(q)$ satisfying~\eqref{cond:a} and~\eqref{cond:b} and $k(q)=\mu e^{-\frac{\pi}{2nq}}$ with $\mu\in [\mu_0,\mu_1]$.

In addition,
$$
|\Cout(q) | \leq   \eta \rho^2 e^{r_0 \sqrt{2}} r_0^{-3/2}, \qquad | \Cin(q)  |\leq \eta \rho^2 e^{-r_0 \sqrt{2}} r_0^{-3/2}, \qquad \mu=\mu(q)= 2 e^{-\frac{C_n}{n^2}- \gamma} (1+ \mathcal{O}(|\log q|^{-1})).
$$
\end{theorem}
\begin{proof}
We define
\begin{equation}\label{def:coutcinhat}
\hat{\Cout}:= \Cout\, e^{-r_0 \sqrt{2}} r_0^{3/2} , \qquad
\hat{\Cin}:= \Cin\, e^{r_0 \sqrt{2}} r_0^{3/2} \rho^{-2} ,
\end{equation}
satisfying
$$
|\hat{\Cout}|, |\hat{\Cin}| \leq \eta.
$$

We impose that $\vin(r_0,\Cin;k,q)=\vout(r_0,\Cout;k,q)$ or equivalently
\begin{equation}\label{eqvoutvin}
\vdin(r_0;k,q) - \vdout(r_0;k,q) = \wout(r_0,\Cout;k,q) - \win(r_0,\Cin;k,q).
\end{equation}
By the results involving $\wout,\win$ in  Theorems~\ref{th:outermatching} and~\ref{th:innermatching} we have that
\begin{align*}
|\wout(r_0;k,q)-\win(r_0;k,q) |& \leq |\wout(r_0;k,q)|+|\win(r_0;k,q)|
\leq M\frac{1}{qr_0^2}   +Mq^3\frac{|\log r_0|^{3}}{r_0}  \\ &\leq M \frac{1}{qr_0^2} + M \frac{\rho^3}{r_0^2} \leq M\frac{1}{qr_0^2}.
\end{align*}
Therefore, by~\eqref{v0inoutmatching}, $\vin(r_0,\Cin;k,q)=\vout(r_0,\Cout;k,q)$
if and only if
$$
\log \left (\frac{\mu}{2}\right ) = -\frac{C_n}{n^2} - \gamma + \mathcal{C}_3(\Cout,\Cin,k;q), \qquad |\mathcal{C}_3(\Cout,\Cin,k;q) |\leq M |\log q|^{-1}.
$$

We recall definition~\eqref{def:coutcinhat} of $\hat{\Cout}, \hat{\Cin}$ and we introduce the function
$$
\mathcal{H}_3(\hat{\Cout},\hat{\Cin},\mu;q) = 2 e^{-\frac{C_n}{n^2} - \gamma}\left[ \text{exp }
\left (\mathcal{C}_3\left ( \hat{\Cout}e^{r_0\sqrt{2}} r_0^{-3/2}, \hat{\Cin} e^{-r_0\sqrt{2}}r_0^{3/2} \rho^2, \mu q^{-1} e^{-\frac{\pi}{2nq}};q\right )\right)-1\right].
$$
It is clear that equation~\eqref{eqvoutvin} is satisfied if and only if
\begin{equation}\label{def:H3matching}
\mu= 2 e^{-\frac{C_n}{n^2}-\gamma} +  \mathcal{H}_3(\hat{\Cout},\hat{\Cin}, \mu;q), \qquad  | \mathcal{H}_3(\hat{\Cout},\hat{\Cin}, \mu;q) |\leq M \leq M |\log q|^{-1}.
\end{equation}

We deal now with the (non-linear) system,
$$
\fout(r_0;k,q)=\fin(r_0;k,q),\qquad \partial_r \fout(r_0;k,q)=\partial_r\fin(r_0;k,q),
$$
which can be rewritten, using expressions for $\fout ,\fin $ in Theorems~\ref{th:outermatching} and~\ref{th:innermatching} as
\begin{equation*}
\begin{aligned}
K_0(r_0\sqrt{2}) \Cout
-I (r_0\sqrt{2}) \Cin &=  \Delta(r_0,\Cout,\Cin;k,q)=\Delta _0(r_0;k,q) + \Delta _1(r_0,\Cout,\Cin;k,q)   ,  \\
K_0'(r_0\sqrt{2}) \Cout
-I  '(r_0\sqrt{2}) \Cin &=  \frac{1}{\sqrt{2}} \partial_r\Delta(r_0,\Cout,\Cin;k,q)=
\frac{1}{\sqrt{2}} \left (  \partial_r\Delta_0(r_0;k,q) + \partial_r\Delta_1(r_0,\Cout,\Cin;k,q) \right )  ,
\end{aligned}
\end{equation*}
with $\Delta_0$ defined in~\eqref{def:Delta0} and
$$
\Delta_1(r,\Cout,\Cin;k,q)=\gin_1(r,\Cin;k,q)-\gout_1(r,\Cout;k,q).
$$
Therefore, $\Cout,\Cin$ satisfy the fixed point equation
\begin{equation}\label{fixedpointequationabtot}
\begin{aligned}
\left (\begin{array}{c} \Cout \\ \Cin \end{array}\right )
&= \mathcal{C}(\Cout,\Cin,k;q) \\ & :=\frac{1}{d(r_0)} \left ( \begin{array}{c}
I'(r_0\sqrt{2}) (\Delta (r_0,\Cout,\Cin;k,q)) -
\frac{1}{\sqrt{2}}I(r_0 \sqrt{2}) \partial_r\Delta(r_0,\Cout,\Cin;k,q) \\
-\partial_rK_0(r_0\sqrt{2}) \Delta(r_0,\Cout,\Cin;k,q) + \frac{1}{\sqrt{2}}K_0(r_0\sqrt{2}) \partial_r\Delta(r_0,\Cout,\Cin;k,q)
\end{array} \right ).
\end{aligned}
\end{equation}

Using the estimates in Theorems~\ref{th:outermatching} and~\ref{th:innermatching} for $\gout_1,\gin_1$, we obtain that
$$
|\Delta_1 (r_0,\Cout,\Cin;k,q)|\leq |\gin_1(r_0,\Cin;k,q)|+ |\gout_1(r_0,\Cout;k,q)|\leq M  \frac{\rho^4 }{r_0^2 },
$$
and $|r_0^2 \partial_r \Delta_1 (r_0,\Cout,\Cin;k,q) |\leq M| \rho^4 $, for any $\Cout$ and $\Cin$ satisfying \eqref{cond:a} and \eqref{cond:b}.

Recalling $\Cout_0$, $\Cin_0$ are defined in \eqref{fixedpointequationab} and using the above bounds for $\Delta_1$ and $\partial_r\Delta_1$ along with~\eqref{propertyImathing} and~\eqref{asymKIn_new} for $I$ and $K_0$,  and the bound for $d(r_0)$ \eqref{prop1matchingf} gives
\begin{equation}\label{propertyC1C2:matching}
| \mathcal{C}_1(\Cout,\Cin,k;q)- \Cout_0|\leq M e^{r_0 \sqrt{2}}  \rho^4 r_0^{-3/2} , \qquad
| \mathcal{C}_2(\Cout,\Cin,k;q)- \Cin_0|\leq M e^{-r_0 \sqrt{2}} \rho^4 r_0^{-3/2}.
\end{equation}
Recalling the definition of $\hat{\Cout}, \hat{\Cin}$  in \eqref{def:coutcinhat} we introduce
\begin{align*}
\mathcal{H}_1(\hat{\Cout},\hat{\Cin},\mu;q) &= e^{-r_0\sqrt{2}}r_0^{3/2} \mathcal{C}_1 (\hat{\Cout}e^{r_0\sqrt{2}} r_0^{-3/2},\hat{\Cin} \rho^2 e^{-r_0 \sqrt{2}} r_0^{-3/2},\mu q^{-1} e^{-\frac{\pi}{2nq}};q)-\hat{\Cout}_0,
\\
\mathcal{H}_2(\hat{\Cout},\hat{\Cin},\mu;q) &= e^{ r_0\sqrt{2}}r_0^{3/2} \rho^{-2} \mathcal{C}_2 (\hat{\Cout}e^{r_0\sqrt{2}} r_0^{-3/2},\hat{\Cin}  \rho^2 e^{-r_0 \sqrt{2}} r_0^{-3/2},\mu q^{-1} e^{-\frac{\pi}{2nq}};q)-\hat{\Cin}_0,
\end{align*}
and the fixed point equation~\eqref{fixedpointequationabtot} becomes
\begin{equation}\label{def:H2matching}
\left (\begin{array}{c} \hat{\Cout} \\ \hat{\Cin} \end{array}\right )
= \left (\begin{array}{c} \hat{\Cout}_0 + {\mathcal{H}}_1(\hat{\Cout},\hat{\Cin},\mu;q)\\
\hat{\Cin}_0 +
{\mathcal{H}}_2(\hat{\Cout},\hat{\Cin},\mu;q)\end{array}\right ).
\end{equation}
Using the bound~\eqref{propertyC1C2:matching} of $\mathcal{C}_1,\mathcal{C}_2$
\begin{equation}\label{boundmatchingaposteriori}
| {\mathcal{H}}_1(\hat{\Cout},\hat{\Cin},\mu;q)| \leq M\rho^4
,\qquad
|{\mathcal{H}}_2(\hat{\Cout},\hat{\Cin},\mu;q)| \leq  M\rho^4|.
\end{equation}

From~\eqref{def:H2matching} and~\eqref{def:H3matching} we have that the constants
$\hat{\Cout},\hat{\Cin}$ and $\mu$ have to satisfy the fixed point equation
\begin{equation}\label{fpematching}
(\hat{\Cout},\hat{\Cin},\mu) = H (\hat{\Cout},\hat{\Cin},\mu;q) :=
\left (\hat{\Cout}_0, \hat{\Cin}_0, 2 e^{\frac{-C_n}{n^2} - \gamma }\right ) + \mathcal{H}(\hat{\Cout},\hat{\Cin},\mu;q),
\end{equation}
with $\mathcal{H}=(\mathcal{H}_1,\mathcal{H}_2,\mathcal{H}_3)$.
We recall that as defined in \eqref{def:rho}, $\rho^3 = q|\log q|^{-1}$ and that the constants $\mu_0,\mu_1$ and $\eta$ were fixed at~\eqref{defmu01:matching} and~\eqref{defeta:matching} respectively.
The function $\mathcal{H}$ satisfies, for $|\hat{\Cout}|,|\hat{\Cin}| \leq \eta $ and $\mu \in [\mu_0,\mu_1]$:
$$
\|\mathcal{H}(\hat{\Cout},\hat{\Cin},\mu;q) \|\leq \max\{M\rho q|\log q|^{-1}, M|\log q|^{-1}\} =
M|\log q|^{-1}.
$$
As a consequence, since $\hat{\Cout}_0$ and $\hat{\Cin}_0$ satisfy \eqref{eq:ab0},  for $|\hat{\Cout}|,|\hat{\Cin}| \leq \eta $ and $\mu \in [\mu_0,\mu_1]$:
 $$
|H_{1,2} (\hat{\Cout},\hat{\Cin},\mu;q)| \leq \frac{\eta}{2} + M|\log q|^{-1} \leq \eta,
$$
and, taking $\mu_0,\mu_1$ as defined in \eqref{defmu01:matching}, one finds
$$
H_3(\hat{\Cout},\hat{\Cin},\mu;q) =   2 e^{\frac{-C_n}{n^2} - \gamma }+\mathcal{O}(|\log q|^{-1})\in [\mu_0,\mu_1].
$$
Therefore, for $q$ small enough, the map $H$ sends the closed ball
$$
B=\{ (\hat{\Cout},\hat{\Cin},\mu) \in \R^3\,:\,
|\hat{\Cout}|,|\hat{\Cin}|\leq \eta, \, \mu \in [\mu_0,\mu_1]\},
$$
into itself and the Brouwer's fixed point theorem concludes the existence of the parameters
$(\hat{\Cout},\hat{\Cin},\mu)=(\hat{\Cout}(q),\hat{\Cin}(q), \mu(q))$ satisfying the fixed point equation~\eqref{fpematching} and
$$
|\hat{\Cout}|\le \eta, \qquad |\hat{\Cin}|\le \eta, \qquad \mu \in [\mu_0,\mu_1].
$$
In addition, for this solution, using the bounds in~\eqref{boundmatchingaposteriori} and~\eqref{eq:ab0}, we have that, for $q$ small enough
$$
|\hat{\Cout}| \leq |\hat{\Cout}_0|+ |\mathcal{H}_1(\hat{\Cout},\hat{\Cin},\mu,q)|\leq \frac{\eta}{2}\rho^2+ M\rho^4 |\log q| \leq \eta \rho^2.
$$
and from~\eqref{def:H3matching},
$$
\left  | \mu - 2 e^{-\frac{C_n}{n^2} - \gamma} \right | \leq M |\log q|^{-1}.
$$
Going back to the original constants $\Cout$ and $\Cin$ using \eqref{def:coutcinhat} completes the proof.
\end{proof}

By Theorem~\ref{thm:matchingtotal}, we can define the solutions of~\eqref{originals} satisfying the boundary conditions~\eqref{eq:boundaryconditions} as in~\eqref{def:solution:outer_inner}:
\begin{equation}\label{defsolution}
(f(r;q),v(r;q)) :=
 \begin{cases}
\big (f^{\mathrm{in}}(r,\mathbf{b}(q);k(q),q), v^{\mathrm{in}}(r,\mathbf{b}(q);k(q),q)\big )  & \text{if  }  r\in [0,r_0], \\
\big (f^{\mathrm{out}}(r,\mathbf{a}(q);k(q),q), v^{\mathrm{out}}(r,\mathbf{a}(q);k(q),q)\big )  & \text{if  }  r\geq r_0.
\end{cases}
 \end{equation}
Therefore, in order to prove Theorem~\ref{thm:main} it only remains to check the additional properties on the solution $(f,v)$.
\begin{proposition}\label{prop:fincreasing} Let $(f(r;q),v(r;q))$ be the solution of~\eqref{originals} defined by~\eqref{defsolution}. There exists $q^*$ such that, for $q\in [0,q^*]$, and $r>0$, $f(r;Q)$ is an increasing function,
$$
0<f(r;q)<\sqrt{1-k^2(q)}, \qquad v(r;q)<0.
$$
\end{proposition}
\begin{proof}
We first prove that $f(r;q)>0$ for $r>0$. We start with the \emph{outer region}. In item~\ref{item4dominantouter} of Proposition~\ref{prop:dominantouter} we proved that
$f_0^{\mathrm{out}}(r;k(q),q)  \geq \frac{1}{2}$ for $r\geq r_0 =e^{\rho/q}/\sqrt{2}$. Therefore, by Theorem~\ref{th:outermatching}, when $r\geq r_0$,
\begin{equation}\label{lowerboundfout}
f (r;q) \geq f_0^{\mathrm{out}}(r;k(q),q) - |g ^{\mathrm{out}}(r,\Cout(q);k(q),q)| \geq \frac{1}{2} - M r^{-2} \geq \frac{1}{2} - M r_0^{-2} >0.
\end{equation}

In the \emph{inner region}, using item~\ref{item1dominant_inner} of Proposition~\ref{prop:f0v0_new} and Theorem~\ref{th:innermatching} we deduce that there exists $\varrho$ small enough but independent of $q$ such that if $r \in [0,\varrho]$,
$$
f (r ; q)= f_0^{\mathrm{in}}(r)+ g^{\mathrm{in}}(r,\Cin(q);k(q),q) = c_f r^n + o(r^{n}) + q^2 \O(r^n) >0,
$$
provided the constant $c_f$ is positive. Then, since $f_0^{\mathrm{in}}$ is positive, increasing and independent of $q$, again using Theorem~\ref{th:innermatching}, for $\varrho\leq r \leq r_0$,
$$
f (r ; q) \geq f_0^{\mathrm{in}}(\varrho) - |\gin(r,\Cin(q);k(q),q)| \geq f_0^{\mathrm{in}}(\varrho)+ \O(q^2) >0,
$$
if $q$ is small enough. This finishes the proof of $f$ being positive.

Now we check that $f(r;q) <\sqrt{1-k^2(q)}$.
We first note that, by \eqref{def:gout}, \eqref{eq:fitagout} and~\eqref{eq:fitagout_bis} in Theorem~\ref{th:outermatching} and using Theorem~\ref{thm:matchingtotal} to bound $\Cout(q)$ we have that
$
g(r;q):= f(r;q)- f_0^{\mathrm{out}}(r,\Cout(q);k(q),q)
$
satisfies that, for $r\geq r_0$:
$$
|r^2 g(r;q)|\leq |r^2\Cout(q) K_0(r)| + M \varepsilon^{1-\alpha} q^{-1} \leq \rho^2 \eta e^{-\sqrt{2}(r-r_0)} r^{3/2} r_0^{-3/2}+ M \varepsilon^{1-\alpha} q^{-1}   \leq M \rho^2,
$$
where we have used that, from definition~\eqref{def:rho} of $\rho$, $\varepsilon^{1-\alpha}q^{-1} = q^{-1} \sqrt{2} e^{-\rho/q} \ll \rho^2$ and the asymptotic expansion~\eqref{asymKIn_new} for the Bessel function $K_0$. Therefore,
$$
f(r;q) \leq \sqrt{1- (v_0^{\mathrm{out}}(r;k(q),q))^2  - \frac{n^2}{r^2}} + M\rho^2 \frac{1}{r^2} \leq \sqrt{1- (v_0^{\mathrm{out}}(r;k(q),q))^2}   -  M \frac{1 }{r^2}  ,
$$
where we have used that $v_0^{\mathrm{out }}(r;k(q),q) \leq M r_0^{-1} =M\varepsilon ^{1-\alpha} \ll 1$ and that $\rho \ll 1$. Then, $f(r;q) \leq \sqrt{1- (v_0^{\mathrm{out}}(r;k(q),q))^2}$ and as a consequence, since $v_0^{\mathrm{out}} \to -k(q)$ as $r\to \infty$ and it is increasing and negative (see item~\ref{item3dominantouter} in Proposition~\ref{prop:dominantouter}), we have that
$$
f(r;q) \leq \sqrt{1- k^2 (q)}, \qquad r\geq r_0.
$$

With respect to the \emph{inner region}, namely $r\in [0,r_0]$, using Proposition~\ref{prop:f0v0_new} there exists $\varrho \gg 1$ independent on $q$ such that for all $\varrho\leq r \leq r_0$,
$  (f_0^{\mathrm{in}})^2 (r) \leq  1 - \frac{n^2}{2 r^2}$.
Then, since by Theorem~\ref{th:innermatching}, $|\gin(r,\Cin;k,q)| \leq Mq^2 |\log r|^2 r^{-2}$ for $\varrho \leq r \leq r_0$ we have that
$$
f^2 (r;q) \leq 1 - \frac{n^2}{2 r^2} + M\rho^2\frac{1}{r^2}  \leq 1 - \frac{1}{2 r_0^2} (n^2 + M\rho^2) \leq  1 - M \varepsilon^{2(1-\alpha)} ,
$$
where we have used that $r_0=\varepsilon^{\alpha -1}$. Then, since  $\e^{\alpha-1} = \frac{1}{\sqrt{2}} e^{\rho/q}$ (see~\eqref{def:r0}) and using definition~\eqref{def:rho} of $\rho$, we conclude that $1-M \varepsilon^{2(1-\alpha)} \leq 1- k^2(q)$, taking if necessary $q$ small enough. As a consequence
$f(r;q) \leq \sqrt{1-k^2(q)}$ if $\varrho\leq r\leq r_0$.
It remains to check the property when $r\in [0,\varrho]$. From the fact that $f_0^{\mathrm{in}}(r)$ is an increasing function and using Theorem~\ref{th:innermatching},
$$
f(r;q) = f_0^{\mathrm{in}}(r) + g^{\mathrm{in}}(r;\Cin(q);k(q),q) \leq f_0^{\mathrm{in}}(\varrho) + M q^{2} < \sqrt{1-k^2(q)},
$$
provided $f_0^{\mathrm{in}}(\varrho) <1$, $\varrho$ is independent on $q$ and $q$ is small enough.

The negativeness of $v(r;q)<0$ for $r>0$ is straightforward from the previous property, $f(r;q) <\sqrt{1-k^2(q)}$. Indeed, using that $v(0;q)=0$, from the differential equations~\eqref{originals}, we have that
$$
v(r;q)= -q \frac{1}{r f^2 (r;q) } \int_{0}^r \xi f^2(\xi;q) (1- f^2(\xi;q) -k^2(q)) \, d\xi < 0.
$$

To finish we prove that $\partial_r f(r;q)>0$. We start with the \emph{inner region}. From Proposition~\ref{prop:f0v0_new}, there exist $0<\varrho_0 \ll  \varrho_1$ satisfying that
$$
\partial_r f_0^{\mathrm{in}}(r)\geq \frac{n}{2} c_f r^{n-1}, \quad \text{if }r\in[0,\varrho_0] \qquad\text{and}\qquad
\partial_r f_0^{\mathrm{in}}(r) \geq \frac{n^2}{2 r^3}, \quad \text{if }r\in [\varrho_1,r_0].
$$
Let $\overline{\varrho} \in [\varrho_0,\varrho_1]$ be such that $\partial_r f_0^{\mathrm{in}}(r) \geq \partial_r f_0^{\mathrm{in}}(\overline{\varrho})>0$ for all $r\in [\varrho_0,\varrho_1]$. Notice that the values of $\varrho_0,\varrho_1$ and $\overline{\varrho}$ are independent on $q$. Therefore, using Theorem~\ref{th:innermatching},
if $r\in [0, \varrho_0]$,
$$
\partial_r f(r;q) = \partial_r f_0^{\mathrm{in}}(r) + \partial_r g^{\mathrm{in}}(r,\Cin(q);k(q),q) \geq \frac{n}{2} c_f r^{n-1} - Mq^2 r^{n-1} >0.
$$
When $r\in [\varrho_0,\varrho_1]$
$$
\partial_r f(r;q) = \partial_r f_0^{\mathrm{in}}(r) + \partial_r g^{\mathrm{in}}(r,\Cin(q);k(q),q) \geq  \partial_r f_0^{\mathrm{in}}(\overline{\varrho}) - Mq^2 >0,
$$
taking, if necessary, $q$ small enough. When $r\geq \varrho_1$, Theorem~\ref{th:innermatching} says that
$$
\partial_r f(r;q) \geq \frac{n^2}{2r^3 } - Mq^2 \frac{|\log r|^2}{r^2},
$$
that is positive if $\varrho_1 \leq r \leq q^{-2}|\log q|^{-3}$, if $q$ small enough. In conclusion
$$
\partial_r f(r;q)>0, \qquad 0\leq r \leq \frac{1}{q^2 |\log q|^3}.
$$
To see that $\partial_r f(r;q)>0$ for bigger values of $r$ we first need to check that
\begin{equation}\label{lowerboundffin}
f(r;q) \geq  \frac{1}{\sqrt{3}}, \qquad r\geq \frac{1}{q^2 |\log q|^3}.
\end{equation}
Indeed, if $q^{-2} |\log q|^{-3} \leq r \leq r_0$, that is, when $r$ belongs to the \emph{inner region}, from Theorem~\ref{th:innermatching}
$$
f(r;q) \geq \fdin(r) - |\gin(r;\Cin(q);k(q),q)| \geq 1- \frac{n^2}{2r^2} + \O(r^{-4}) - M q^2 \frac{|\log r|^2}{r^2}
\geq 1- \O(q^2 |\log q|^3)\geq \frac{1}{\sqrt{3}}.
$$
When $r\geq r_0$ (that is in the \emph{outer region}), by~\eqref{lowerboundfout}, $f(r;q) \geq \frac{1}{3}$ and~\eqref{lowerboundffin} is proven.

We finish the argument to prove that $f$ is an increasing function for $r>0$, by contradiction.
Since we have proved that $\partial_r f(r;q)>0$ for $r>q^{-2}|\log q|^{-3}$ and $f^2(r;q) \leq 1-k^2(q) = \lim_{r\to \infty} f^2(r;q)$, if $f$ has an extreme at $r^*$, it has to have a maximum at some point less than $r^*$. Let $r_*\geq q^{-2} |\log q|^{-3}$ be the minimum value such that $f(r;q)$ has a maximum at $r=r_*$. That is $\partial_r f(r_*,q)=0$ and $\partial_r^2 f(r_*;q)\leq 0$. Therefore, since $f$ is a solution of~\eqref{originals}, we deduce that
\begin{equation}\label{original_r_ast}
f(r_*;q) \left [-\frac{n^2}{r_*^2}  + (1- f^2(r_*;q) - v^2(r_*;q)) \right ]\geq 0.
\end{equation}

Now we use the following comparison result: (see~\cite{protter})
\begin{lemma}\label{wonderful} \cite{protter} Let $(a, b)$ be an interval in $\R$, let $\Omega = \R^2 \times  (a, b)$, and let $\mathcal{H} \in \co^1(\Omega, \R)$. Suppose $h \in \co^2((a, b))$ satisfies $h''(r)+\mathcal{H}(h'(r), h(r), r) = 0$. If $\partial_h \mathcal{H} \leq  0$ on $\Omega$ and if there exist functions $M, m \in \co^2 ((a, b))$ satisfying $M''(r) + \mathcal{H}(M'(r), M(r), r) \leq 0$ and $m''(r) + \mathcal{H}(m'(r), m(r), r) \geq  0$, as well as the boundary conditions $m(a) \leq h(a) \leq  M(a)$ and $m(b) \leq h (b) \leq M(b)$, then for all $r\in (a, b)$ we have $m(r) \leq h (r) \leq M(r)$.
\end{lemma}

We set $(a,b)=(r_*,\infty)$ and we define
$$
\mathcal{H}(h',h,r)= \frac{h'}{r} - h \frac{n^2}{r^2} + h(1-h^2-v^2(r;q)), \qquad h\geq \frac{1}{\sqrt{3}},
$$
with $v(r;q)$ the solution we  have already found, and
$$
\mathcal{H}(h',h,r)= \frac{h'}{r} - h \frac{n^2}{r^2} -hv^2(r;q)) + \frac{2}{3\sqrt{3}} ,   \qquad h\leq \frac{1}{\sqrt{3}}.
$$
We have that $\mathcal{H} \in \mathcal{C}^1(\Omega,\mathbb{R})$ and $\partial_h\mathcal{H}\leq 0$.
By~\eqref{lowerboundffin}, for $r\geq r_* \geq q^{-2} |\log q|^{-3}$, $f(r;q)\geq \frac{1}{\sqrt{3}}$ so that $f(r;q)$ is a solution of $h'' + \mathcal{H}(h'(r),h(r),r)=0$.
Taking $m(r)=f(r_*;q)$ we have that $\lim_{r\to \infty} m(r) = f(r_*;q) \leq \lim_{r\to \infty} f(r;q)= \sqrt{1-k^2}$ and
\begin{align*}
m'' + \mathcal{H}(m'(r),m(r),r)&= -f(r_*;q) \frac{n^2}{r^2} + f(r_*;q) (1- f^2 (r_*;q) - v^2(r_*;q)) \\  &\geq -f(r_*;q) \frac{n^2}{r_*^2} + f(r_*;q) (1- f^2 (r_*;q) - v^2(r_*;q))\geq 0,
\end{align*}
where we have used~\eqref{original_r_ast} in the last inequality.
Then Lemma~\ref{wonderful} concludes that $f(r_*;q)=m(r) \leq f(r;q)$ for $r\geq r_*$. Therefore $r_*$ can not be a maximum and we have a contradiction.
\end{proof}

The rest of the work is devoted to prove the results about the existence of families of solutions in the outer and inner regions.
From now on, to avoid cumbersome notation, we will skip the dependence on the parameters $k,q$.

\section{Existence result in the {outer region}. Proof of Theorem~\ref{th:outermatching}}
\label{sec:provaOuter}

In this section we prove Theorem~\ref{th:outermatching}. To do so, by means of a fixed point equation setting, we look for solutions of equations~\eqref{outer_0} which are written in the \textit{outer variables} introduced in Section~\ref{subsec:outer_inner_variables} (see~\eqref{eq:canvi_0}). Namely, we look for solutions of the equations~\eqref{outer_0} with boundary conditions~\eqref{eq:boundaryconditionsouter_0} that are of the form $F_0 + G$, $V_0+W$ with $F_0,V_0$ defined in~\eqref{eq:F0_0} and~\eqref{eq:V0_0} respectively, that is, taking $\e=kq$,
\begin{equation}\label{defV0F0secouter}
V_0(R)=\frac{K_{inq}'(R)}{K_{inq}(R)},\qquad F_0(R) = \sqrt{1 - k^2 V_0^2(R) - \frac{\e^2 n^2}{R^2}}.
\end{equation}

We first introduce the Banach spaces we will work with.
For any given $R_\m>0$, we introduce the Banach spaces:
\begin{equation}
\label{def:banach}
    \mathcal{X}_{\ell}=\{ f:[R_\m,\infty)\to \R\,:\,\text{continuous}\,,\; \|f\|_\ell:=\sup_{R\in [R_\m,\infty)} |R^{\ell} f(R)| <\infty\},
\end{equation}
being $\mathcal{X}_0$ the Banach space of continuous bounded functions with the supremmum norm.

Notice that $\mathcal{X}_{\ell}=\mathcal{X}_\ell(R_\m)$ depends on $R_\m$ and so the norm of a function does.
However, if $R_{\m}\leq R_{\m}'$,
$\mathcal{X}_\ell(R_\m) \subset \mathcal{X}_\ell(R_\m')$ and
$$
\sup_{R\in [R_{\m},\infty)}|R^\ell f(R)| \geq \sup_{R\in [R_{\m}',\infty)}|R^\ell f(R)|.
$$
This fact allows us to take $R_\m'\geq R_\m$, if we are working in $\mathcal{X}_\ell(R_\m)$.
We will use this property along the work without any special mention.
\subsection{The fixed point equation}
Our goal in this section is to transform equations~\eqref{outerF_0},\eqref{outerV_0} in a fixed point equation in suitable Banach spaces.
For that, the first step is to write such equations in a suitable way.

Let $F=F_0+G$ and $V=V_0+W$.
The term $F(1-F^2-k^2V^2)$ in equation~\eqref{outerF_0} is:
\begin{align*}
F(1-F^2-k^2V^2) =&  -2 F_0^2 G  - 3 F_0 G^2 -G^3 - W k^2\big [2V_0 F_0 + F_0 W + 2V_0 G + WG\big ] \\
&+ (F_0+G) \frac{n^2 \e^2}{R^2}.
\end{align*}
Therefore, equation~\eqref{outerF_0} becomes
\begin{align*}
\e^2 \left (G'' + \frac{G'}{R} \right ) -&2F_0^2 (R) G =
-\e^2 \left (F_0''(R) + \frac{F_0'(R)}{R}  \right )
\\&+
3 F_0(R) G^2 + G^3 +Wk^2 \big [2V_0 (R)F_0(R) + F_0(R) W + 2V_0(R) G + WG\big ].
\end{align*}
In view of\eqref{asymptoticandbound:V0_new}, that in \emph{outer variables} reads as
$$
F_0(R)= \sqrt{1-k^2} \left (1 - \frac{k^2}{2R(1-k^2)} + \O\left (\frac{k^2}{R^2}\right )\right ),
$$
we introduce
\begin{equation}\label{def:F0hatouter}
F_0^2 (R)= 1+\frac{1}{2} \widehat{F}_0(R).
\end{equation}
Therefore we may write the above equation for $G$ as
\begin{equation}\label{eqfinal:G}
G'' + \frac{G'}{R} - G  \frac{2 }{\e^2}
=- \e^{-2}\mathcal{N}_1[G,W].
\end{equation}
with
\begin{equation}\label{def:N1outer}
\begin{aligned}
\mathcal{N}_1[G,W](R)=&\e^2 \left (F_0''(R) + \frac{F_0'(R)}{R} \right )  -\widehat{F}_0(R) G - 3 F_0(R) G^2 - G^3 \\
&-Wk^2 \big(2V_0 (R)F_0 (R)+ F_0(R) W + 2V_0 (R)G + WG\big).
\end{aligned}
\end{equation}

Now we compute the equation for $W$ from~\eqref{outerV_0}. We have that
\begin{equation*}
\begin{aligned}
&W' + \frac{W}{R} + 2 V_0 (R)W + W^2 +  V_0'(R) + \frac{V_0(R)}{R} +V_0(R)^2 -1 +\frac{n^2}{R^2}q^2 \\
 &=
\frac{q^2}{(F_0(R)+G)} \left (F_0''(R) + \frac{F_0'(R)}{R} + G''+ \frac{G'}{R}\right )
 - 2 (V_0(R)+W) \frac{F_0'(R) + G'}{F_0(R) + G}.
\end{aligned}
\end{equation*}
We recall that $V_0$ is a solution of~\eqref{Ricatti_0}.
Then
\begin{equation}\label{eqfinal:W}
W' + \frac{W}{R} + 2 V_0 W =  - \mathcal{N}_2(G,W)(R).
\end{equation}
with
\begin{equation}\label{def:N2outer}
\begin{aligned}
\mathcal{N}_2[G,W](R) =& W^2 -\frac{q^2}{F_0(R)+G} \left (F_0''(R) + \frac{F_0'(R)}{R} + G''+ \frac{G'(R)}{R}\right ) \\ &+2 (V_0(R)+W) \frac{F_0'(R) + G'}{F_0(R) + G} .
\end{aligned}
\end{equation}

We define the linear operators:
\begin{equation*}
\begin{aligned}
\mathcal{L}_1[G](R) &= G'' + \frac{G'}{R} - G \frac{2 }{\e^2}, \\
\mathcal{L}_2[W](R) & =W' + \frac{W}{R} + 2 V_0(R) W,
\end{aligned}
\end{equation*}
and rewrite equations~\eqref{eqfinal:G} and~\eqref{eqfinal:W} as
\begin{equation}\label{eqGW}
\mathcal{L}_1[G] = -\e^{-2}   \mathcal{N}_1[G,W] ,\qquad \mathcal{L}_2[W] = - \mathcal{N}_2[G,W].
\end{equation}

The strategy to prove the existence of solutions of~\eqref{eqGW} is to write them as a Gauss-Seidel fixed point equation and to prove that the fixed point theorem can be applied in suitable Banach spaces. For that, first, we need to compute a right inverse of $\mathcal{L}_1,\mathcal{L}_2$.

We start with $\mathcal{L}_1$.
Assume that we have
\begin{equation}\label{eq:L1houter}
\mathcal{L}_1 [G](R)=-h(R),
\end{equation}
where $h$ satisfies some conditions that we will specify later. We are interested in solutions of this equation such that
$\lim_{R\to \infty} G(R)=0$.

Just for doing computations, we  perform the scaling:
\begin{equation*}
s= \frac{R}{\e}\sqrt{2}, \qquad g(s) = G(s\e/\sqrt{2}),
\end{equation*}
and~\eqref{eq:L1houter} is transformed into
\begin{equation}\label{auxlinear}
g'' + \frac{g'}{s} - g= -\frac{\e^2}{2} h(s \e/\sqrt{2}).
\end{equation}
The homogeneous linear system associated, has a fundamental matrix
$$
\left (\begin{array}{cc} K_0(s) & I_0(s) \\ K_0'(s) & I'_0(s)\end{array}\right ),
$$
where $K_0,I_0$ are the modified Bessel functions \cite{abramowitz} of the first and second kind.
The Wronskian is given by $W(K_0(s),I_0(s)) = s^{-1}$ so that the solutions of~\eqref{auxlinear} are given by
$$
g(s) = K_0(s) \left [\Cout + \frac{\e^{2}}{2} \int_{s_0}^s \xi I_0(\xi) h(\xi \e/\sqrt{2})\dd \xi \right ]+
I_0(s) \left [\mathbf{b} - \frac{\e^{2}}{2}\int_{s_0}^s \xi K_0(\xi) h(\xi \e/\sqrt{2}) \dd \xi \right ].
$$
It is well known that $K_0(s) \to 0$ and $I_0(s) \to \infty$ as $s\to \infty$ (see \eqref{asymKIn_new}).
Then, in order to have solutions bounded as $s\to \infty$,
we have to impose
$$
\textbf{b} - \frac{\e^{2}}{2} \int_{s_0}^\infty \xi K_0(\xi) h(\xi \e/\sqrt{2}) \dd \xi=0.
$$
Therefore,
$$
g(s) = K_0(s) \left [\Cout + \frac{\e^{2}}{2} \int_{s_0}^s \xi I_0(\xi) h\left (\frac{\xi\e}{\sqrt{2}}\right )\dd \xi \right ]+
\frac{\e^{2}}{2} I_0(s) \int_{s}^\infty \xi K_0(\xi) h\left (\frac{\xi\e}{\sqrt{2}}\right ) \dd \xi,
$$
and, proceeding in the same way,
$$
g'(s) = K_0'(s) \left [\Cout + \frac{\e^{2}}{2} \int_{s_0}^s \xi I_0(\xi) h\left (\frac{\xi\e}{\sqrt{2}}\right )\dd \xi \right ]+
\frac{\e^{2}}{2} I_0'(s) \int_{s}^\infty \xi K_0(\xi) h\left (\frac{\xi\e}{\sqrt{2}}\right ) \dd \xi.
$$
Now we undo the change of variables that is: $R=s\e/\sqrt{2}$ and $G(R) = g(R\sqrt{2}/\e) $. We obtain the solution of~\eqref{eq:L1houter}
\begin{equation*}
G(R)=K_0\left (\frac{R\sqrt{2}}{\e}\right ) \left [\Cout+  \int_{R_\m}^{R} \xi I_0\left (\frac{\xi\sqrt{2}}{\e} \right ) h(\xi)\dd \xi \right ]+
I_0\left (\frac{R\sqrt{2}}{\e}\right ) \int_{R}^\infty \xi K_0\left (\frac{\xi \sqrt{2}}{\e} \right ) h(\xi) \dd \xi,
\end{equation*}
with $R_\m=s_0\e/\sqrt{2}$ to be determined later.

We introduce the linear operator
\begin{equation}\label{def:inverse1outer}
\mathcal{S}_1[h](R)=K_0\left (\frac{R\sqrt{2}}{\e}\right )\int_{R_\m}^{R} \xi I_0\left (\frac{\xi\sqrt{2}}{\e} \right ) h(\xi)\dd \xi  +
I_0\left (\frac{R\sqrt{2}}{\e}\right ) \int_{R}^\infty \xi K_0\left (\frac{\xi\sqrt{2}}{\e} \right ) h(\xi) \dd \xi.
\end{equation}
We have proven:
\begin{lemma}\label{lem:S1outer}
For any $\Cout\in \R$ we define
\begin{equation}\label{def:G0cteouter}
\mathbf{G}_0(R)=K_0\left (\frac{R\sqrt{2}}{\e}\right )\Cout.
\end{equation}
Then, if $G$ is a solution of~\eqref{eqfinal:G} satisfying $G(R) \to 0$ as $R\to \infty$ then there exists a constant $\Cout$ such that
$$
G=\mathbf{G}_0 +\mathcal{S}_1[\e^{-2}\mathcal{N}^{-1}[G,W]].
$$
\end{lemma}

Now we compute the right inverse of $\mathcal{L}_2$.
We consider the linear equation
\begin{equation}\label{eq:L2houter}
\mathcal{L}_2[W]=W' + W\left (\frac{1}{R} + 2 V_0\right ) = h.
\end{equation}
Since $V_0(R) = K_{inq}'(R) /K_{inq}(R)$, the solutions are given by:
$$
W(R) = \frac{1}{R K_{inq}^2(R)} \left ( c_0 + \int_{R_0}^R \xi K_{inq}^2(\xi) h(\xi)\right ),
$$
for any constant $c_0$.
In order for $W$ to be bounded as $R\to \infty$ it is required that
$$
c_0 + \int_{R_0}^\infty \xi K_{inq}^2(\xi) h(\xi)\dd \xi =0.
$$
Therefore
$$
W(R) = \frac{1}{R K_{inq}^2(R)} \int_{\infty}^R \xi K_{inq}^2(\xi) h(\xi) \dd \xi.
$$
As a result we have the following Lemma:
\begin{lemma}\label{lem:S2outer}
Any solution of~\eqref{eq:L2houter} bounded as $R\to \infty$ is of the form
$W=\mathcal{S}_2[h]$ with
\begin{equation}\label{def:inverse2outer}
\mathcal{S}_2[h]=\frac{1}{R K_{inq}^2(R)} \int_{\infty}^R \xi K_{inq}^2(\xi) h(\xi) \dd \xi.
\end{equation}
\end{lemma}
From Lemmas~\ref{lem:S1outer} and~\ref{lem:S2outer} we can rewrite~\eqref{eqGW} as a fixed point equation $(G,W)=\op[G,W]$ defined by
\begin{equation*}
\begin{aligned}
G&= \op_1[G,W]: = \mathbf{G}_0 + \opg_1[ \e^{-2} \mathcal{N}_1[G,W]], \\
W&= \op_2[G,W]:=- \mathcal{S}_2 [ \mathcal{N}_2[G,W]],
\end{aligned}
\end{equation*}
where $\mathbf{G}_0$ depends on a constant $\Cout$. Notice that the nonlinear operator $\mathcal{N}_2$ defined in \eqref{def:N2outer} involves the derivatives $G',G''$. In order to avoid working with norms involving derivatives, we will take advantage of the differential properties of $\op_1$ and using that $G=\op_1[G,W]$ we rewrite the fixed point equation as
\begin{equation}\label{eqfxp}
\begin{aligned}
G&= \op_1[G,W]: = \mathbf{G}_0 + \opg_1[ \e^{-2} \mathcal{N}_1[G,W]], \\
W&= \op_2[G,W]:=- \mathcal{S}_2 [ \mathcal{N}_2[\op_1[G,W],W]],
\end{aligned}
\end{equation}
where $\mathcal{S}_1$ is defined in \eqref{def:inverse1outer}, $\mathcal{S}_2$ in \eqref{def:inverse2outer}, $ \mathcal{N}_1$ in \eqref{def:N1outer} and $\mathcal{N}_2$ in \eqref{def:N2outer}.

In Section~\ref{subsec:linear} we study the linear operators $\mathcal{S}_1$ and $\mathcal{S}_2$ (see~\eqref{def:inverse1outer} and~\eqref{def:inverse2outer}) and prove that they are bounded operators in $\mathcal{X}_\ell$ for $\ell\geq 0$.

Our goal is now to prove the following result which is a  reformulation of Theorem~\ref{th:outermatching}.
\begin{theorem} \label{prop:outerfixedpoint}
Let $\eta >0$, $0<\mu_0<\mu_1$
and take $\e=\mu e^{-\frac{\pi}{2nq}}$ with $\mu_0\le \mu \leq \mu_1$.
There exist $q_0=q_0(\mu_0,\mu_1,\eta)>0$ and $e_0=e_0(\mu_0,\mu_1,\eta)>0$, $M=M(\mu_0,\mu_1,\eta)>0$ such that, for any $q\in[0,q_0]$, $\alpha\in(0,1)$
satisfying
\begin{equation*}
q^{-1}\e ^{1-\alpha}<e_0,
\end{equation*}
and for any constant $\Cout$ satisfying
\begin{equation}
\label{eq:condicioc1}
 (\e^{\a} )^{3/2}e^{-\frac{\sqrt{2}}{ \e^{1-\a}}}|\Cout|\le  \eta \e^{3/2} ,
 \end{equation}
there exists a family of solutions $(G(R,\Cout),W(R,\Cout))$ of the fixed point equation~\eqref{eqfxp} defined for $R \geq R_\m^*=  \e^{\a}$ which satisfy
$$
\|G\|_2 + \e \|G'\|_2+\e \|W\|_2 \leq M \e^2 .
$$

Moreover $G(R,\Cout)=G^0(R) + G^1(R,\Cout)$ and $W(R,\Cout)=W^0(R,\Cout)+ W^1(R,\Cout)$ satisfying
\begin{itemize}
\item[(i)] there exists $q_0^*=q_0^*(\mu_0, \mu_1)>0$, and $M_0=M_0(\mu_0,\mu_1)$ such that, for $q\in[0,q_0^*]$,
\[
\|G^0\|_2 +\e\|(G^0)'\|_2\leq M_0 \e^{3-\a}q^{-1},
\]
\item[(ii)]  for $q\in[0,q_0]$, we can decompose
$G^1(R,\Cout)=K_0\left (\frac{R \sqrt{2} }{\e}\right )\Cout + \widehat{G}^1(R,\Cout)$
with
\begin{equation*}
\|\widehat{G}^1\|_2+\e \|(\widehat{G}^1)'\|_2 \leq M \frac{\e^{1-\a}}{q} \left \|
K_0\left (\frac{R \sqrt{2}}{\e }\right )\right \|_2 |\Cout| \leq M_1\e^{2}.
\end{equation*}
\item[(iii)] and for $q\in [0,q_0]$
$$
\e\|W^0\|_2  \leq M   \left \|K_0\left (\frac{R \sqrt{2}}{\e }\right )\right \|_2 |\Cout| \leq M_1\e^2
, \qquad
\e \|W^1\|_2 \leq M_1\frac{\e^{3-\alpha}}{q} .
$$
\end{itemize}
where $M_1=M_1(\mu_0,\mu_1,\eta)$ depends on $\mu_0,\mu_1$, and $\eta$.
\end{theorem}

The rest of this section is devoted to prove this theorem.
In Section~\ref{subsec:linear} we prove that the linear operators $\mathcal{S}_1$ and $\mathcal{S}_2$, defined in~\eqref{def:inverse1outer} and~\eqref{def:inverse2outer}, are bounded in $\mathcal{X}_\ell $, $\ell \geq 0$.
In Section~\ref{subsec:independent} we study
$\op[0,0]$
and finally, in Section~\ref{subsec:lip} we check that the operator $\op$ is Lipschitz in a suitable ball.

It is worth mentioning that the more technical part in this procedure comes from the study of the function
$V_0$ (and $K_{inq}$) done in Proposition~\ref{prop:dominantouter}.

From now on, we fix $\eta, \mu_0,\mu_1$, we will take $\e,q$ as small as needed, and $\Cout$ satisfying~\eqref{eq:condicioc1}. We also will denote by $M$ any constant independent of $\e,q$.

\subsection{The linear operators}\label{subsec:linear}
We prove that, $\mathcal{S}_1, \mathcal{S}_2$ are bounded operators in the Banach spaces $\mathcal{X}_\ell$ defined in \eqref{def:banach} along with important properties of such operators.

\subsubsection{The operator $\mathcal{S}_1$}
In this section we prove that $\mathcal{S}_1 :\mathcal{X}_\ell \to \mathcal{X}_\ell$ is a bounded operator. In addition we also provide bounds for
$(\mathcal{S}_1[h])', (\mathcal{S}_1[h])''$.

\begin{lemma}\label{lem:G1}
Take $R_\m \geq  \e z_0/\sqrt{2}$, with $z_0$ given in~\eqref{asymKIn_new}, and $\ell\ge 0$. Then, if $\e$ is small enough, the linear operator $\mathcal{S}_1:\mathcal{X}_\ell \to \mathcal{X}_\ell$  defined in~\eqref{def:inverse1outer} is a bounded operator. Moreover there exists a constant $M>0$ such that for $h\in \mathcal{X}_\ell$,
$$
\|\mathcal{S}_1[h]\|_\ell \leq M \e ^2 \|h\|_\ell.
$$
\end{lemma}
\begin{proof}
Since $R_\m$ is such that $\frac{R_\m\sqrt{2}}{\e}  >z_0$, by~\eqref{asymKIn_new}, for any $R\geq R_\m$
\begin{equation}\label{asymptoticsKn}
K_0\left (\frac{R\sqrt{2}}{\e}\right )= \sqrt{\frac{\pi\e}{2\sqrt{2}R}} e^{-\frac{R\sqrt{2}}{\e}} \left (1+ \O\left (\frac{\e }{R}\right )\right ) ,
\end{equation}
and
$$
I_0\left (\frac{R\sqrt{2}}{\e}\right )= \sqrt{\frac{\e}{2\sqrt{2}R\pi}} e^{\frac{R\sqrt{2}}{\e}} \left (1+ \O \left (\frac{\e}{R}\right )\right ).
$$

Let now $h\in \mathcal{X}_\ell$, that is $|h(\xi) | \leq  \xi^{-\ell} \|h\|_\ell$. Then:
\begin{align*}
\left | R^\ell {\opg}_1[h](R) \right |
&\leq C R^{\ell-1/2} \left(\frac{\e}{\sqrt{2}}\right)  \| h\|_\ell
\left [ e^{-\frac{R\sqrt{2}}{\e }}\int_{R_\m}^R \frac{e^{\frac{\xi\sqrt{2}}{\e}}}{\xi^{\ell-1/2}} \dd \xi
+ e^{\frac{R\sqrt{2}}{\e}}\int_{R}^\infty \frac{e^{\frac{-\xi\sqrt{2}}{\e }}}{\xi^{\ell-1/2}} \dd \xi\right ] \\
& \leq  C \left (\sqrt{2}\frac{R}{\e}\right )^{\ell-1/2} \left(\frac{\e}{\sqrt{2}}\right) ^{2}  \| h\|_\ell\left [ e^{-\frac{R\sqrt{2}}{\e }} \int_{z_0}^{\frac{R\sqrt{2}}{\e}} \frac{e^t}{t^{\ell-1/2}} \dd t
+ e^{\frac{R\sqrt{2}}{\e}}\int_{\frac{R\sqrt{2}}{\e}}^\infty \frac{e^{-t}}{t^{\ell-1/2}} \dd t\right ]\\
&= C  \left(\frac{\e}{\sqrt{2}}\right)^{2}  \| h\|_\ell \mathcal{M}\left (\frac{R\sqrt{2}}{\e}\right ).
\end{align*}
where
\begin{align*}
\mathcal{M}(z)&=
z^{\ell-1/2}\left [ e^{-z} \int_{z_0}^{z} \frac{e^t}{t^{\ell-1/2}} \dd t
+ e^{z}\int_{z}^\infty \frac{e^{-t}}{t^{\ell-1/2}} \dd t\right ],
\end{align*}
and one can easily see that $\lim_{z\to \infty}\mathcal{M}(z)=1$. Therefore there exists a constant $M>0$ such that $|\mathcal{M}(z)|\le M$ for $z\geq z_0$ and
consequently:
$$
\left | R^\ell  {\opg}_1[h](R) \right | \leq C M\e^{2} \|h \|_\ell.
$$
\end{proof}

\begin{corollary}\label{cor:L1}
Let $R_\m\geq  \e z_0 /\sqrt{2}$ and $\ell\geq 0$. Then for $\e$  small enough and $h\in \mathcal{X}_\ell$, the function  ${\opg}_1[h]$ belongs to $\mathcal{C}^2 ([R_\m,\infty))$. In addition, there exists a constant $M>0$ such that:
$$
\left \|\big ( {\opg}_1[h]\big )' \right \|_\ell \leq M\e  \|h\|_\ell , \qquad \left \|\big ( {\opg}_1 [h]\big )'' \right \|_\ell \leq M \|h\|_\ell.
$$
\end{corollary}
\begin{proof}
Let $\varphi=  {\opg}_1(h)$.
Notice that
\begin{align*}
\varphi'(R) = \frac{\sqrt{2}}{\e} \left [K_0'\left (\frac{R\sqrt{2}}{\e}\right )   \int_{R_\m}^{R} \xi I_0\left (\frac{\xi\sqrt{2}}{\e} \right ) h(\xi)\dd \xi  +
I_0'\left (\frac{R\sqrt{2}}{\e}\right ) \int_{R}^\infty \xi K_0\left (\frac{\xi\sqrt{2}}{\e} \right ) h(\xi) \dd \xi \right ].
\end{align*}
That is $\varphi$ is differentiable if $h$ is continuous (by definition). Moreover, since $K_0'(z),I_0'(z)$ have the same asymptotic expansions as $K_0,I_0$ (in~\eqref{asymKIn_new}) performing the same computations as in the proof of Lemma~\ref{lem:G1} we obtain the result for $\varphi'$.

We note that $\varphi'$ is differentiable if $h$ is continuous (again simply by definition). Then $\varphi$ is $\mathcal{C}^2$. Moreover,
$$
\varphi'' + \frac{\varphi'}{R} -   2\frac{\varphi}{\e^2} =-h,
$$
and therefore
$$
|R^\ell \varphi''(R) | \leq   M\| h \|_\ell \left (3+ \frac{\e }{R} \right ) \leq M \|h\|_\ell.
$$
\end{proof}

\subsubsection{The operator $\mathcal{S}_2$}
Let us first provide a technical lemma.
\begin{lemma}\label{lem:Kinq1_new}
There exists $q_0>0$, such that for any $0<q<q_0$,
if $R\geq 2e^2  e^{-\frac{\pi}{2qn}}$:
\begin{equation*}
\frac{1}{K_{inq}^2(R)} \int_{R}^{\infty} K_{inq}^2(\xi) \dd \xi \leq \frac{1}{2}.
\end{equation*}
\end{lemma}
\begin{proof}
The proof is straightforward from item~\ref{item3dominantouter} of  Proposition~\ref{prop:dominantouter}. Indeed, we first recall that $V_0(R)=\vdout(R/\e)$ and hence $V_0(R)<-1$. Then, we consider the function $\psi(R) = \int_{R}^\infty K_{inq}^2(\xi)\dd \xi - \frac{1}{2} K_{inq}^2(R)$ and we point out that we just need to prove that $\psi(R)\leq 0$ if $R\geq 2 e^2 e^{-\frac{\pi}{2n q}}$. We have that
\begin{align*}
\psi'(R) & = - K_{inq}^2(R) - K_{inq}(R) K_{inq}'(R) = -K_{inq}^2(R)  \left [ 1 + \frac{K'_{inq}(R)}{K_{inq}(R)}\right ]
\\ &=-K_{inq}^2(R) [1 +V_0(R)].
\end{align*}
Therefore, since $V_0(R)<-1$ for $R\geq 2 e^2 e^{-\frac{\pi}{2 nq }}$, then $\psi'(R)>0$ and using that $\psi(R) \leq \lim_{R\to \infty} \psi(R) = 0$ the result is proven.
\end{proof}

The following lemma, provides bounds for the norm of the linear operator
$\opg_2$, defined in \eqref{def:inverse2outer}.
\begin{lemma}\label{lem:W2}
There exists $q_0>0$, such that for any $0<q<q_0$,
taking $R_\m\geq 2 e^2 e^{-\frac{\pi}{2qn}}$, the operator $\mathcal{S}_2 : \mathcal{X}_\ell \to \mathcal{X}_\ell$, defined in \eqref{def:inverse2outer} is bounded for all $\ell \geq 1$. Moreover, if $h\in \mathcal{X}_\ell$, $\ell=1,2$
$$
\|\opg_2[h] \|_\ell\leq \frac{1}{2} \|h \|_\ell .
$$

In addition, when $h\in \mathcal{X}_3$,
\begin{equation}
\label{eq:lem47b}
 \|\opg_2[h]\|_2 \leq  \|h\|_3.
\end{equation}

\end{lemma}
\begin{proof}
Let $\ell \geq 1$ and $h\in \mathcal{X}_\ell $. Then , by Lemma~\ref{lem:Kinq1_new}
$$
\left |R^\ell \opg_2[h](R) \right | \leq \frac{R^{\ell-1}\|h\|_\ell}{K_{inq}^2(R)} \int_{R}^\infty \frac{K_{inq}^2(\xi)}{\xi^{\ell-1}}\dd \xi  \leq
 \frac{\|h\|_\ell}{K_{inq}^2(R)} \int_{R}^\infty K_{inq}^2(\xi)\dd \xi \leq \frac{1}{2} \|h\|_\ell.
$$

When $h\in \mathcal{X}_3$, then since $K_{inq}>0$ and decreasing:
$$
\left |R^2 \opg_2[h](R) \right | \leq \frac{R\|h\|_3}{K_{inq}^2(R) } \int_{R}^\infty \frac{K_{inq}^2(\xi)}{\xi^2} \dd \xi \leq
  \|h \|_3 R \int_{R}^\infty \frac{1}{\xi^2}\dd \xi \leq  \|h\|_3.
$$
\end{proof}

Because in the definition of the operator $\mathcal{N}_2 $ (see \eqref{def:N2outer}) there are  some derivatives involved, we need a more accurate control about how the operator
$\opg_2$ acts on a special type of functions. In particular we shall need to control $\mathcal{S}_2[h V_0]$, where we recall that $V_0= K'_{inq}(R) (K_{inq}(R))^{-1}$. For this reason we study first the auxiliary linear operator defined by
\begin{equation}
\label{def:A}
\mathcal{A}[h](R)=\mathcal{S}_2[h V_0](R)= \frac{1}{R K_{inq}^2(R)} \int_\infty^{R} \xi h(\xi) K_{inq}'(\xi) K_{inq}(\xi) \dd \xi.
\end{equation}
\begin{lemma}\label{lem:AS2}
With the same hypotheses as in Lemma~\ref{lem:W2}, for any $h\in \mathcal{X}_\ell $,
$$
\|\mathcal{A}[h]\|_\ell \leq \frac{1}{2}\|h\|_\ell.
$$
\end{lemma}
\begin{proof}
Let $h\in \mathcal{X}_\ell$. Then
\begin{align*}
|R^\ell \mathcal{A}[h](R)| &\leq \frac{R^{\ell-1} \|h\|_\ell  }{K_{inq}^2(R)} \int_R^\infty (-K_{inq}'(\xi) K_{inq}(\xi) )\frac{1}{\xi^{\ell -1}} \dd \xi
\leq \frac{\|h\|_\ell }{K_{inq}^2 (R) }\int_R^\infty -K_{inq}'(\xi) K_{inq}(\xi) \dd \xi \\ & = \frac12 \|h\|_\ell .
\end{align*}
\end{proof}

\begin{lemma}\label{lem:S22} Let $h_1,h_2$ be bounded differentiable functions. Then:
\begin{align*}
\opg_2[h_1 h_2'](R)  =  h_1(R) h_2(R) - \opg_2[h_1' h_2] - \opg_2[\hat h](R)
-2\mathcal{A}[h_1h_2](R),
\end{align*}
where $\hat{h}(R)= h_1(R) h_2(R) R^{-1}$.
If $(\xi \hat{h}_1)'=\xi \hat{h}_2$ and $h$ is a differentiable bounded function, then
$$
\opg_2[\hat{h}_2 h] (R)= \hat{h}_1(R) h(R) - \opg_2[h'\hat{h}_1](R) - 2\mathcal{A}[\hat{h}_1 h](R).
$$
\end{lemma}
\begin{proof}
We prove both properties by integrating by parts. Indeed, since $h_1,h_2$ are bounded functions
\begin{align*}
\int_\infty^R \xi h_1(\xi) & h_2'(\xi) K_{inq}^2(\xi) \dd \xi =  R h_1 (R) h_2(R) K_{inq}^2 (R)
\\ & - \int_{\infty}^R h_2(\xi) \big[h_1 (\xi)K_{inq}^2 (\xi) + \xi h_1'(\xi)K_{inq}^2(\xi) + 2 \xi h_1(\xi ) K_{inq}'(\xi)K_{inq}(\xi) \dd \xi\big].
\end{align*}
Therefore
$$
\opg_2[h_1 h_2'](R) = \frac{1}{R K_{inq}^2(R)}\int_\infty^R \xi h_1(\xi) h_2'(\xi) K_{inq}^2 (\xi) \dd \xi,
$$
satisfies the statement.

With respect to the second equality. Again by doing parts:
\begin{align*}
\opg_2[\hat{h}_2 h] (R)=& \frac{1}{R K_{inq}^2(R)} \int_\infty^R (\xi \hat{h}_1(\xi))' h(\xi) K_{inq}^2 (\xi) \dd \xi
\\ =& \hat{h}_1(R)h(R)  - \frac{1}{R K_{inq}^2(R)} \int_{\infty}^R \xi \hat{h}_1(\xi) \big [h'(\xi) K_{inq}^2(\xi) + 2 h(\xi) K_{inq}'(\xi) K_{inq}(\xi) \big ] \dd \xi.
\end{align*}
\end{proof}
\subsection{The independent term}\label{subsec:independent}
We study now which the independent term of the fixed point equation~\eqref{eqfxp}, that is $ \op[0,0]=( \op_1[0,0], \op_2[0,0])$. We recall that
\begin{equation}\label{eq:primeraitouter}
\begin{aligned}
\op_1[0,0]&= \mathbf{G}_0+\opg_1[ \e^{-2}\mathcal{N}_1[0,0] ], \\
\op_2[0,0]&=-\opg_2[\mathcal{N}_2[ \op_1[0,0],0]]\end{aligned},
\end{equation}
and $\mathcal{N}_1,\mathcal{N}_2,\mathbf{G}_0,\mathcal{S}_1,\mathcal{S}_2$ are defined in~\eqref{def:N1outer} and~\eqref{def:N2outer}, \eqref{def:G0cteouter}, \eqref{def:inverse1outer}, \eqref{def:inverse2outer} respectively.

Before starting with the study of \eqref{eq:primeraitouter} we state a straightforward corollary of items~\ref{item3dominantouter} and~\ref{item4dominantouter} of Proposition~\ref{prop:dominantouter} about the behaviour of $F_0,V_0$ (see~\eqref{defV0F0secouter}).
\begin{corollary}\label{lem:Kinq2_new}
Let $R_\m = \e^{\a}$ with $\a \in (0,1)$.
Then there exist $q_0 >0$ and a constant $
M>0$ such that for any $0<q<q_0$ and $R\in [R_\m,+\infty)$, $V_0'(R)>0$, $V_0(R)<-1$,
$$
|kV_0(R)|, |k V_0'(R)R|,|kV''(R)R^2|\leq M \e^{1-\a},
$$
and
$$
|R (V_0(R)+1)|,|R^2V_0'(R)|, |R^3 V_0''(R)|\leq M .
$$
With respect to $F_0$, we have that $F_0(R)\geq 1/2$, $F_0'(R)>0$ and
$$
|F'_0(R) R^2|, |F''_0(R) R^3| \leq  C k \e^{1-\a},\qquad |1-F_0(R)|, |F_0'(R)R|, |F_0''(R)R^2|\leq C \e^{2(1-\a)}.
$$
\end{corollary}

From now on we then take $R_\m =\e^\a$ with $0<\alpha<1$ satisfying $\e^{1-\alpha}/q$ small enough. These conditions will ensure that $\e/R_\m \ll 1$. The following proposition provides the size of $\mathcal{F}[0,0]$ in~\eqref{eq:primeraitouter}.

\begin{lemma}
\label{lem:independentouter}
Let $0<\mu_0<\mu_1$
and take $\e=\mu e^{-\frac{\pi}{2nq}}$ with $\mu_0\le \mu \leq \mu_1$.
There exist $q^*_0=q^*_0(\mu_0,\mu_1)>0$, $M=M(\mu_0,\mu_1)>0$ such that, for any $q\in[0,q^*_0]$ and $\alpha\in(0,1)$
satisfying $\e^{1-\alpha}/q<1$, $R_\m= \e^\a$, given $\eta >0$ and $\Cout$ satisfying~\eqref{eq:condicioc1} in the definition of $\mathbf{G}_0$ provided in \eqref{def:G0cteouter} we have:
\begin{equation}
\label{eq:boundsG0}
    \|G_0\|_2+ \e\|G_0'\|_2 + \e^{2}\|G_0''\|_2 \leq \|\mathbf{G}_0\|_2 + M\e^{4-2\a}\leq M(1+ {\eta})\e^2,
\end{equation}
with $G_0= \op_1[0,0]$. As a consequence, there exists $q_1^*(\mu_0,\mu_1,\eta)$ such that, if $q\in[0,q_1^*]$ then $F_0(R)+G_0(R)\geq 1/4$.

Let $W_0=\op_2[0,0]$. Then there exists $q_2^*(\mu_0,\mu_1,\eta)$ such that for $q\in[0,q_2^*]$
$$
\|W_0\|_2 \leq M \e^{2-\a} q^{-1} + M \eta\e  \le M(1+\eta)\e.
$$
\end{lemma}
We divide the proof of this lemma in two parts, the first one, in Section~\ref{sec:independentouterF} corresponds to the bound for $G_0$ and the second one, in Section~\ref{sec:independentouterW} corresponds to the bound for $W_0$.

\subsubsection{A bound for the norm of $G_0$ and its derivatives: }\label{sec:independentouterF}
Recall that $G_0=\op_1[0,0]$ as given in~\eqref{eq:primeraitouter}.
We start bounding $\|\mathbf{G}_0\|_2,\|\mathbf{G}_0'\|_2,\|\mathbf{G}_0''\|_2$, with $\mathbf{G}_0$ given in \eqref{def:G0cteouter}.
By~\eqref{asymptoticsKn} it is clear that, for $R\geq R_\m =  \e^\alpha$,
$$
|R^2 \mathbf{G}_0(R)| = \left |R^2 K_0\left (\frac{R\sqrt{2}}{\e }\right ) \Cout \right | \leq   M |\Cout| \sqrt{\e}
R_\m^{3/2}
e^{-\frac{R_\m\sqrt{2}}{\e}} \leq M|\Cout| \sqrt{\e} (\e^\alpha)^{3/2}
e^{-\frac{\sqrt{2}}{\e^{1-\alpha}}},
$$
if $0<q<q^*_0$, for $q_0^*=q^*_0(\mu_0,\mu_1)$. Therefore, using that $\Cout$ satisfies~\eqref{eq:condicioc1} we conclude that
$\|\mathbf{G}_0\|_2 \leq M  {\eta}\e^2$. In addition it is clear that
$\e \|\mathbf{G}_0'\|_2+\e^2 \|\mathbf{G}_0''\|_2 \leq M \|\mathbf{G}_0\|_2 \leq M {\eta}\e^2$, and thus
\begin{equation}\label{eq:fitesG0}
\|\mathbf{G}_0\|_2+\e \|\mathbf{G}_0'\|_2+\e^2 \|\mathbf{G}_0''\|_2 \leq M {\eta}\e^2\, .
\end{equation}

To deal with $\mathcal{S}_1 [\e^{-2} \mathcal{N}_1[0,0]]$ (see \eqref{def:N1outer}) we first bound
\begin{equation*}
\mathcal{F}_0(R)=  \mathcal{N}_1[0,0](R)=\e^2 \left (F_0''(R) + \frac{F_0'(R)}{R}  \right ).
\end{equation*}
By Corollary~\ref{lem:Kinq2_new}
$$
|R^2 \e^{-2} \mathcal{F}_0(R)| \leq M \e^{2(1-\a)},
$$
and applying Lemma~\ref{lem:G1} we obtain $\|\opg_1 (\e^{-2} \mathcal{F}_0(R))\|_2 \leq  C \e^{4-2\a}$, which gives:
$$
\|G_0\|_2 \le \|\mathbf{G}_0\|_2 + M\e^{4-2\a} \le M ( {\eta}\e^2+\e^{4-2\a}) \le M(1+ {\eta})\e ^2.
$$
Using Corollary~\ref{cor:L1} we obtain the bounds for the derivatives:
\begin{equation}
\label{eq:boundsDGO}
\e \|  G'_0\|_2+ \e^2 \| G''_0\|_2 \le M ( {\eta}\e^2+\e^{4-2\a}) \le M(1+ {\eta})\e ^2,
\end{equation}
and \eqref{eq:boundsG0} is proved.

To finish we notice that by Corollary~\ref{lem:Kinq2_new}, there exists $q_1^*(\mu_0,\mu_1,\eta)$ such that, if $q\in[0,q_1^*]$
\begin{equation}
\label{eq:boundFG}
F_0(R) + G_0(R) \geq \frac{1}{2} - M (1+{\eta})\frac{\e^2}{R^2}  \geq \frac{1}{2} - M(1+{\eta})\e^{2(1-\a)}  \geq \frac{1}{4}.
\end{equation}

\subsubsection{A bound for $\|W_0\|_2$}\label{sec:independentouterW}

We recall that $W_0=\mathcal{S}_2[\mathcal{N}_2[\opF_1[0,0],0]]=\mathcal{S}_2[\mathcal{N}_2[G_0,0]]$ where $\mathcal{N}_2$ is defined in~\eqref{def:N2outer}, namely
\begin{equation*}
\mathcal{N}_2[G_0,0] =  2V_0 \frac{F_0'+G_0'}{F_0+G_0} -q^2 \frac{1}{F_0+G_0}\left (F_0'' + G_0''+  \frac{F_0'+G_0'}{R}\right ).
\end{equation*}

By definition~\ref{def:A} of $\mathcal{A}$
$$
\opg_2\left [V_0 \frac{F_0'+G_0'}{F_0+G_0}\right] =
\mathcal{A}\left [\frac{F_0'+G_0'}{F_0+G_0}\right].
$$
Therefore, for $0<q<q_1^*(\mu_0,\mu_1,\eta)$, using Lemma~\ref{lem:AS2}, Corollary~\ref{lem:Kinq2_new} and bounds ~\eqref{eq:boundFG} and~\eqref{eq:boundsDGO},
\begin{align*}
\left \|\opg_2\left [V_0 \frac{F_0'+G_0'}{F_0+G_0} \right ]\right \|_2 & \leq \left \|
\frac{F_0'+G_0'}{F_0+G_0}\right \|_2\leq
M(k \e^{1-\a} + \e^{3-2\a}+\e \eta  )
\\ & \leq M (\e^{2-\a} q^{-1} + + \e^{3-2\a}+ \e \eta) \leq M(\e^{2-\a} q^{-1} + \e \eta),
\end{align*}
where we have used that $k \e^{1-\a}=\e q^{-1}\e^{1-\a} \leq  \e $. In the rest of the proof we will reduce the value of $q_1^*$, if necessary, without changing the notation.
In addition, by Corollary~\ref{lem:Kinq2_new} since
$$
q^2 \left |R^3 \frac{1}{F_0+G_0}\left (F_0'' + \frac{F_0'}{R}\right )\right |  \leq M q^2 k\e^{1-\a} = M q \e^{2-\a},
$$
we also have that by inequality \eqref{eq:lem47b} in Lemma~\ref{lem:W2},
$$
q^2\left \|\opg_2\left [\frac{1}{F_0+G_0}\left (F_0'' + \frac{F_0'}{R}\right ) \right ]\right \|_2 \leq M q\e^{2-\a} .
$$
To bound the last term in $\mathcal{W}_0$, we use the second statement of Lemma~\ref{lem:S22} with
$$
h = \frac{1}{F_0+G_0}, \qquad \hat{h}_2  = G_0'' + \frac{G_0'}{R}, \qquad \hat{h}_1 =  G_0'.
$$
Then
$$
\left \|\opg_2 \left [ \frac{1}{F_0+G_0} \left ( G_0'' + \frac{G_0'}{R} \right )\right ] \right \|_2 \leq  \left \| \frac{G_0'}{F_0+G_0} \right \|_2 + \|\opg_2 [h' G_0']\|_2 + 2 \left \|\mathcal{A}\left [\frac{G_0'}{F_0+G_0}\right ] \right \|_2.
$$
By bounds \eqref{eq:boundsDGO} and \eqref{eq:boundFG},
$$
\left \| \frac{G_0'}{F_0+G_0} \right \|_2   \leq  M\eta \e + M\e^{3-2\a} ,
$$
and as a consequence, by Lemma~\ref{lem:AS2},
$$
\left \|\mathcal{A}\left [\frac{G_0'}{F_0+G_0}\right ] \right \|_2 \leq   M\eta \e + M\e^{3-2\a}.
$$
By bound \eqref{eq:boundFG} and since $R\geq R_{\min}=\e^{\alpha}$:
\begin{align*}
| G_0'(R) h'(R) | &\leq |G_0'(R)| \frac{|F_0'(R)| + |G'_0(R)|}{|F_0(R) + G_0(R)|^2}
\leq
M \left (\frac{\e^{3-2\a}+\eta\e }{R^2} \right )
\frac{\e^{2-\a}q^{-1}+\e^{3-2\a}+\eta \e }{R^2}
\\ & \leq \frac{M}{R^3} \left(
\e^{4-3\a}  + \eta\e^{2-\a}\right),
\end{align*}
where we have used that $\varepsilon^{1-\alpha}/q\leq 1$. Then, using Lemma~\ref{lem:W2} $\|\opg_2 [h'G_0']\|_2 \leq \|h' G_0'\|_3 $ and therefore,
$\|q^2 \opg_2 [h' G_0' \|_2 \leq M   (q^2\e^{4-3\a} + q^2\eta\e^{2-\a})$.
We conclude that
\begin{equation*}
\|W_0\|_2 \leq M \e^{2-\a} q^{-1} + M\eta \e  \leq M(1+ \eta) \e.
\end{equation*}

\subsection{The contraction mapping}\label{subsec:lip}

In Lemma~\ref{lem:independentouter} we have proven that the independent term $(G_0,W_0)= \op[0,0]$ (defined in \eqref{eq:primeraitouter}) satisfies $\|G_0\|_2 + \e \|W_0\|_2 \leq M{(1+\eta) }\e^2$. In other words, the independent term belongs to the Banach space $\mathcal{X}_2 \times \mathcal{X}_2$ endowed  with the norm
$$
\llfloor(G,W)\rrfloor = \|G\|_2 + \e  \|W\|_2.
$$
Let
\begin{equation}\label{def:kappa0}
\kappa_0=\kappa_0(\mu_0,\mu_1,\eta)=\llfloor(G_0,W_0)\rrfloor \e^{-2}.
\end{equation}
Along this section we will prove the following result.

\begin{lemma}\label{lem:Lipschitzouter}
Let $\eta >0$, $0<\mu_0<\mu_1$
and take $\e=\mu e^{-\frac{\pi}{2nq}}$ with $\mu_0\le \mu \leq \mu_1$. Take $\kappa \geq 2\kappa_0$, where $\kappa_0$ is defined in \eqref{def:kappa0}, and $\Cout$ satisfying the condition~\eqref{eq:condicioc1}.
There exist $q_0=q_0(\mu_0,\mu_1,\eta)>0$ and $M=M(\mu_0,\mu_1,\eta)>0$ such that, for any $q\in[0,q_0]$ and $\alpha\in(0,1)$
satisfying $q^{-1}\e ^{1-\alpha}<1$, taking $R_\m\geq  \e^{\a}$,
if $ (G_1,W_1),  (G_2,W_2) \in \mathcal{X}_2 \times \mathcal{X}_2$ with $\llfloor(G_1,W_1)\rrfloor,\llfloor(G_2,W_2)\rrfloor \leq \kappa \e^2$ then
\begin{equation}
\label{eq:Lips}
    \llfloor \op[G_1,W_1]- \op[G_2,W_2] \rrfloor  \leq  M\e^{1-\a} q^{-1} \llfloor(G_1,W_1)-(G_2,W_2)\rrfloor.
\end{equation}
where the operator $\op$ is defined in \eqref{eqfxp}.

If moreover $\|G_1'\|_2, \|G_2'\|_2\leq \kappa \e$.
\begin{equation}\label{eq:caca}
\begin{aligned}
\e \|\mathcal{S}_2[\ \mathcal{N}_2[G_1, W_1] ]-
  \mathcal{S}_2 [ \mathcal{N}_2[G_2, W_2]] \|_2 \leq & M  \e^{2-\a} \|W_1 - W_2 \|_2 +
 M\e^{1-\a} \|G_1 - G_2 \|_2 \\ & + M\e \|G_1' - G_2'\|_2 ,
\end{aligned}
\end{equation}
with $\mathcal{S}_2$ defined in \eqref{def:inverse2outer} and $\mathcal{N}_2$ in \eqref{def:N2outer}.
Also,
\begin{equation}
\label{eq:lipF1outer}
\e\| \big (\op_1[G_1,W_1]-\op_1[G_2,W_2]\big )'\|_2 \leq M \e^{1-\a} q^{-1} \llfloor(G_1,W_1)-(G_2,W_2)\rrfloor.
\end{equation}
\end{lemma}
Next section is devoted to prove Theorem~\ref{prop:outerfixedpoint} from the above results and Lemma~\ref{lem:Lipschitzouter}. We postpone the proof of this lemma to Section~\ref{sec:provalemma}.

\subsection{Proof of Theorem \ref{prop:outerfixedpoint}} \label{sec:proofofTheoremouterfixedpoint}
Lemma \ref{lem:Lipschitzouter}, for $0<q<q_0$, gives us the Lipschitz constant of $ \op$ with the norm $\llfloor \cdot \rrfloor$ on $\mathcal{B}_{\kappa\e^2}$, the closed ball of $ \mathcal{X}_2\times \mathcal{X}_2$ of radius $\kappa \e ^2$. Indeed, the Lipschitz constant is $M\e^{1-\a}q^{-1}\leq 1/2$ if  $\e^{1-\a}q^{-1}<e_0:=1/(2M)$.
Then the operator $ \op$ is a contraction.
Moreover, if $(G,W)\in \mathcal{B}_{\kappa\e^2}$, it is clear that
$$
\llfloor \op[G,W]\rrfloor \leq \llfloor \op[G,W]-  \op[0,0]\rrfloor+ \llfloor \op[0,0]\rrfloor \leq \frac{1}{2}\llfloor(G,W)\rrfloor + \kappa_0 \e^2 \leq
\frac{1}{2} \kappa \e^2 + \frac{\kappa}{2} \e ^2\leq \kappa \e ^2.
$$
Then, the existence of a solution of the fixed point equation~\eqref{eqfxp}, namely $(G,W)= \op[G,W]$, belonging to $\mathcal{B}_{\kappa\e^2}$ is guaranteed by the Banach fixed point theorem.

Moreover, as
$$
\|G\|_2= \| \op_1[G,W] \|_2 \leq
\kappa \e^2  ,
$$
using \eqref{eq:lipF1outer},
one can easily see:
$$
\|G'\|_2=\|\big ( {\op}_1[G,W]\big )'\|_2 \leq \kappa \e,\qquad \|G''\|_2 =\|\big ( {\op}_1[G,W]\big )''\|_2 \leq \kappa.
$$

We introduce the auxiliary operator
$$
\widehat{\mathcal{F}}[G,W] =(\widehat{\mathcal{F}}_1,\widehat{\mathcal{F}}_2)[G,W]:= (\e^{-2} \mathcal{S}_1[\mathcal{N}_1[G,W]], -\mathcal{S}_2[\mathcal{N}_2[\widehat{\mathcal{F}}_1[G,W],W]]).
$$
Observe that $\widehat{\mathcal{F}}[G,W] = \mathcal{F}[G,W]$ for $\Cout=0$.
We denote by $(G^0, W^0)$, the solution of the fixed point equation $(G,W)= \widehat{\mathcal{F}}[G,W]$.
We point out that, by Lemma~\ref{lem:independentouter} and recalling that $\e^{1-\a}\leq q/2$, for $0<q\le q_0^*(\mu_0,\mu_1)$, we have:
$$
\llfloor \widehat{\mathcal{F}}[0,0] \rrfloor
\leq M (\e^{4-2\a} + \e^{3-\a}q^{-1}) \leq M \e^{3-\a}q^{-1}.
$$
Therefore, in this case, $\bar \kappa_0 = \kappa_0(\mu_0,\mu_1,0)= \e^{-2} \llfloor \widehat{\mathcal{F}}[0,0]\rrfloor \leq M\e^{1-\a} q^{-1}$, with $\kappa_0$ defined in~\eqref{def:kappa0}, and that implies that
$$
\llfloor (G^0, W^0) \rrfloor \leq  2\bar\kappa_0 \e^2 \leq 2 M \e^{3-\a}q^{-1}.
$$
Denoting by $M_0=2M$ (which only depend on $\mu_0,\mu_1$) the proof of first item of Theorem~\ref{prop:outerfixedpoint} is done.

Let now $(G,W)$ be the solution for a given $\Cout$ satisfying~\eqref{eq:condicioc1}. We
have that
\begin{align*}
G=\mathcal{F}_1[G,W]=& \mathbf{G}_0 + \widehat{\mathcal{F}}_1[G,W], \\
W= \mathcal{F}_2[G,W]=& -\mathcal{S}_2[\mathcal{N}_2[\mathcal{F}_1[G,W],W]] \\ =&
-\mathcal{S}_2[\mathcal{N}_2[\mathbf{G}_0+\widehat{\mathcal{F}}_1[G,W],W]] + \mathcal{S}_2[\mathcal{N}_2[\widehat{\mathcal{F}}_1[G,W],W]]
- \widehat{\mathcal{F}}_2[G,W].
\end{align*}
Therefore, using that $(G^0, W^0)=\widehat{\mathcal{F}}[G^0,W^0]$, we have that, using \eqref{eq:Lips} and \eqref{eq:caca}:
\begin{align*}
\llfloor (G ,W) - (G^0, W^0)\rrfloor  \leq & \|\mathbf{G}_0 \|_2 + \llfloor \widehat{\mathcal{F}}[G,W] - \widehat{\mathcal{F}}[G^0,W^0] \rrfloor \\
&+
\e \|\mathcal{S}_2[\mathcal{N}_2[\mathbf{G}_0+\widehat{\mathcal{F}}_1[G,W],W]] - \mathcal{S}_2[\mathcal{N}_2[\widehat{\mathcal{F}}_1[G,W],W]] \|_2 \\
\leq &
\|\mathbf{G}_0 \|_2 + M\e^{1-\a} q^{-1}
\llfloor (G,W) - (G^0, W^0)\rrfloor + M\e^{1-\a}\|\mathbf{G}_0\|_2 + M\e \|\mathbf{G}_0'\|_2 \\
\leq & M \|\mathbf{G}_0\|_2 + M\e \|\mathbf{G}_0'\|_2
+M\e^{1-\a} q^{-1}
\llfloor (G,W) - (G^0, W^0)\rrfloor .
\end{align*}
As a consequence, using that, by \eqref{eq:fitesG0}, $\|\mathbf{G}_0\|_2+\e \|\mathbf{G}_0'\|_2 \leq M \e ^2$, we obtain
$$
\llfloor (G,W) - (G^0, W^0)\rrfloor  \leq M \e^2.
$$
Then
\begin{align*}
\|G - \mathbf{G}_0 - G^0 \|_2  & =
\|\widehat{\mathcal{F}}_1[G,W] - \widehat{\mathcal{F}}_1[G^0,W^0]\|_2\leq M\e^{1-\a}q^{-1} \llfloor (G,W) - (G^0, W^0)\rrfloor  \\ & \leq M \e^{3-\a}q^{-1}  .
\end{align*}

The bounds for $\|(G^0)'\|_2$ and $\|G'-(G^0)'-\mathbf{G}_0'\|_2$, follow from the bound \eqref{eq:lipF1outer}, and an analogous expression for $\widehat{F}_1$, along with expression \eqref{eq:boundsDGO}. Denoting by $\widehat{G}^1=G-G^0-\mathbf{G}_0$, Theorem~\ref{prop:outerfixedpoint} is proven.

\subsection{Proof of Lemma \ref{lem:Lipschitzouter}}\label{sec:provalemma}

The proof of Lemma~\ref{lem:Lipschitzouter}  is divided into two parts, in Sections~\ref{sec:LipT1outer} we prove inequality \eqref{eq:Lips} and in Section~\ref{sec:LipT2outer} we prove \eqref{eq:caca} and \eqref{eq:lipF1outer}.

\subsubsection{The Lipschitz constant of $\op_1$}\label{sec:LipT1outer}

Let $(G_1,W_1), (G_2,W_2)\in \mathcal{X}_2 \times \mathcal{X}_2$ belonging to the closed ball of radius $\kappa \e^2$, that is
$\|(G_1,W_1)\|,\|(G_1,W_1)\|\leq \kappa \e^2$.
We have that, using Lemma~\ref{lem:G1},
\begin{equation}\label{lipT1R1outer}
\begin{aligned}
\|\op_1[G_1,W_1]-\op_1[G_2,W_2]\|_2 & = \e^{-2} \|\opg_1\big [\mathcal{N}_1 [G_1,W_1]- \mathcal{N}_1 [G_2,W_2]\big ]\|_2
\\ & \leq M  \|\mathcal{N}_1 (G_1,W_1)- \mathcal{N}_1 (G_2,W_2)\|_2.
\end{aligned}
\end{equation}
Then to compute the Lipschitz constant of $ \op_1$ it is enough to deal with the Lipschitz constant of $\mathcal{N}_1$.

Now we write $\eta(\lambda)=(1-\lambda) (G_1,W_1)+ \lambda(G_2,W_2)$ and, for any $R\geq R_\m=\e^\a$:
\begin{align*}
\mathcal{N}_1[G_2,W_2](R)-\mathcal{N}_1 [G_1,W_1] (R)=&  \int_{0}^1 \partial_G \mathcal{N}_1 [\eta(\lambda)](R) (G_2(R)-G_1(R))
\\ & + \int_0^1 \partial_W \mathcal{N}_1[\eta(\lambda)](R)(W_2(R)-W_1(R)) \dd \lambda.
\end{align*}
Then, since $\|\eta(\lambda)\|_2 \leq \kappa \e^2$ to bound the Lipschitz constant of $\mathcal{N}_1$ it is enough to bound $|\partial_G\mathcal{N}_1[G,W]|$ and $|\partial_W\mathcal{N}_1[G,W]|$ for $\|(G,W)\|_2\leq \kappa \e^2$.

We now recall that $\widehat{F}_0$ in~\eqref{def:F0hatouter} is defined as  $F_0^2 = 1+ \widehat{F}_0/2$. Then, since by Corollary~\ref{lem:Kinq2_new} $|kV_0(R)| \leq M\e^{1-\a}$ and $F_0^2 = 1-k^2 V_0^2-\e^2 n^2 R^{-2}$ we have that, using that $R\geq R_\m =  \e^{\a}$
\begin{equation}\label{boundhatF0outer}
|\widehat{F}_0(R)| \leq  M k^2 |V^2_0(R)| + M\e^2 R^{-2} \leq M \e^{2-2\a} .
\end{equation}
Then, if $|G(R)| \leq \kappa \e^{2} R^{-2} \leq M \e^{2-2\a}$:
\begin{equation}\label{boundF0+GG}
|F_0(R) + G(R)|\leq 1 + \O(\e^{2-2\a}) \leq 2,
\end{equation}
if $q$ is small enough.

We claim that, if $\|(G,W)\|_2 \leq \kappa \e^2$, then
\begin{equation}\label{bound:DR}
| \partial_{G} \mathcal{N}_1 [G,W](R) |  \leq M\e^{2-2\a} , \qquad |\partial_{W} \mathcal{N}_1 [G,W](R)|\leq Mk \e^{1-\a}.
\end{equation}
Indeed, we have that
$$
\partial_{G} \mathcal{N}_1 (G,W)=- \widehat{F}_0 - 6 F_0 G - 3 G^2 - 2 k^2 W V_0 - k^2 W^2,
$$
where $\mathcal{N}_1$ is given in \eqref{def:N1outer}.
Then, using~\eqref{boundhatF0outer}, that $|G(R)|\leq \kappa \e^2 R^{-2}$ and $|W(R)| \leq \kappa \e R^{-2}$
\begin{align*}
|\partial_{G} \mathcal{N}_1 [G,W]| & \leq M \left (\e^{2-2\a} + \kappa \frac{\e^2}{R^2} + \kappa^2 \frac{\e^4}{R^4} + \kappa k^2 |V_0(R)| \frac{\e}{R^2} +
\kappa^2 k^2 \frac{\e^2}{R^4}\right ) \\
& \leq M \left (\e^{2-2\a} + \kappa \e^{2-2\a} + \kappa^2 \e^{4-4\a} + \kappa k \e^{-\a} \e^{2-2\a} + \kappa^2 k^2 \e^{-2\a} \e^{2-2\a}\right ) \\
& \leq M e^{2-2\a} \left (1 + \kappa + \kappa^2 \e^{2-2\a} + \kappa  \e^{1-\a}/q  + \kappa ^2 q^{-2} \e^{2-2\a} \right ) \leq M\e^{2-2\a},
\end{align*}
where we have used again that $\e^{1-\a}/q\leq 1$.
With respect to $\partial_{W} \mathcal{N}_1[G,W]$, we have that:
$$
\partial_{W} \mathcal{N}_1 [G,W]=-2 k^2 V_0 (F_0+G) - 2 k^2  W(F_0+G).
$$
Then, using~\eqref{boundF0+GG}:
$$
|\partial_{W} \mathcal{N}_1 [G,W]| \leq M \left (k \e^{1-\a} + k^2 \frac{\e}{R^2}\right) \leq
M \left(k \e^{1-\a} +k^2\e^{1-2\a}\right)  \leq Mk \e^{1-\a}(1+\e^{1-\a}/q),
$$
provided $\e^{1-\alpha}/q<1$, and~\eqref{bound:DR} is proven.

Finally, using bounds~\eqref{bound:DR} of $\partial_W \mathcal{N}_1, \partial_G \mathcal{N}_2$:
$$
|\mathcal{N}_1[G_2,W_2](R)-\mathcal{N}_1 [G_1,W_1](R) |   \leq M \e^{2-2\a} |G_1(R)-G_2(R) | + M k \e^{1-\a}   |W_1(R) -W_2(R)|,
$$
and therefore,
\begin{align*}
\| \mathcal{N}_1(G_2,W_2)-\mathcal{N}_1 (G_1,W_1)\|_2 &\leq M \e^{2-2\a} \|G_1 - G_2\|_2 + M k \e^{1-\a}  \|W_1-W_2\|_2
\\ & \leq M \e^{1-\a} q^{-1} \llfloor (G_1,W_1)-(G_2,W_2)\rrfloor.
\end{align*}
This bound and~\eqref{lipT1R1outer} lead to the Lipschitz constant of $ \op_1$, which is $M \e^{1-\a}/q$.

From these computations we also deduce expression~\eqref{eq:lipF1outer} using  Corollary~\ref{cor:L1}.

\subsubsection{The Lipschitz constant of $ \op_2$}\label{sec:LipT2outer}

Now we deal with $\op_2[G,W]$ which is defined by
$$
\op_2 [G,W]= \opg_2 (\mathcal{N}_2 [\op_1[G,W],W]]).
$$
Recall that $\mathcal{N}_2$ was introduced at~\eqref{def:N2outer}:
\begin{align*}
\mathcal{N}_2[G,W](R) =& W^2 -  \frac{q^2}{F_0(R)+G(R)} \left (F_0''(R) + G''(R) + \frac{F_0'(R)+G'(R)}{R} \right ) \\ &+ 2 (V_0(R) + W)
\frac{F_0'(R)+G'(R)}{F_0(R)+ G(R)},
\end{align*}
We have to deal with each term of the difference
\begin{align*}
\opg_2  \big [\mathcal{N}_2 [\op_1[G_1,W_1],W_1] - \mathcal{N}_2 [\op_1[G_2,W_2],W_2]\big ]
\end{align*}
separating in a similar way as we did for computing the
norm of $W_0$ in Lemma~\ref{lem:independentouter}.
Along this proof we will use without special mention the first item of Lemma~\ref{lem:Lipschitzouter} (already proven) and the bounds in~\eqref{eq:lipF1outer}.

Take
$(G_1,W_1),(G_2,W_2) \in \mathcal{X}_2 \times \mathcal{X}_2$ with $\llfloor(G_1,W_1)\rrfloor, \llfloor (G_2,W_2)\rrfloor\leq \kappa \e^2$ and $\|G_1'\|_2, \|G_2'\|_2\leq \kappa \e$.
We first prove
\begin{equation}\label{boundlipS2N2outer}
\begin{aligned}
\e \|\mathcal{S}_2[\ \mathcal{N}_2[G_1, W_1] ]-
  \mathcal{S}_2 [ \mathcal{N}_2[G_2, W_2]] \|_2 \leq & M\e \e^{1-\a} \|W_1 - W_2 \|_2 +
 Mq^2\e^{1-\a} \|G_1 - G_2 \|_2 \\ & + Mq\e \|G_1' - G_2'\|_2 .
\end{aligned}
\end{equation}
We define $G_\lambda= (1-\lambda)G_2 + \lambda G_1$ and
$W_\lambda = (1-\lambda)W_2 + \lambda W_1$ and we notice that the operator $\mathcal{N}_2$ can be written as
$$
\mathcal{N}_2[G,W] = \widetilde{\mathcal{N}}_2[G,G',G'',W].
$$
By the mean's value theorem
\begin{align*}
\mathcal{N}_2[G_1,W_1]- \mathcal{N}_2[G_2,W_2] =&
(W_1-W_2) \int_{0}^1 \partial_W \widetilde{\mathcal{N}}_2[G_\lambda, G_\lambda',G_\lambda'',W_\lambda] \,d \lambda  \\ &+ (G_1-G_2)\int_{0}^1 \partial_G \widetilde{\mathcal{N}}_2[G_\lambda, G_\lambda',G_\lambda'',W_\lambda] \,d \lambda
\\ &+ (G_1'-G_2')\int_{0}^1 \partial_{G'} \widetilde{\mathcal{N}}_2[G_\lambda, G_\lambda',G_\lambda'',W_\lambda] \,d \lambda
\\ &+ (G_1''-G_2'')\int_{0}^1 \partial_{G''} \widetilde{\mathcal{N}}_2[G_\lambda, G_\lambda',G_\lambda'',W_\lambda] \,d \lambda
\\ =:& N_1+N_2+N_3+N_4.
\end{align*}
We start with $\e \mathcal{S}_2[N_1]$.
We have that
$\partial_W \widetilde{\mathcal{N}}_2 [G,G',G'',W] = 2 W+2\frac{F'_0+G'}{F_0+G}$ and therefore, using the bounds for $F_0, F_0'$ in Corollary~\ref{lem:Kinq2_new}:
\begin{align*}
\e |N_1(R) | &\leq  \e \|W_1 - W_2\|_2 \left(\frac{\e M}{R^4}+\frac{\e^{1-\a} k}{R^4}\right)
\leq  \e \|W_1 - W_2\|_2 \left(\frac{\e M}{R^4}+\frac{\e^{2-\a}q^{-1} }{R^4}\right)
\\ & \leq
M \e \|W_1 - W_2\|_2 \frac{\e^{1-\a}}{R^3}
\end{align*}
where we have used that $\e^{1-\a}/q< 1$.
Then, by Lemma~\ref{lem:W2},
$$
\e \| \mathcal{S}_2[N_1]\|_2 \leq \e M \|N_1\|_3  \leq M \e \e^{1-\a}\|W_1- W_1\|_2.
$$
We follow with $N_2$. It is clear that
\begin{align*}
\e \big |\partial_{G}\widetilde{\mathcal{N}}_2[G,G',G'',W](R)\big |  = & \frac{\e }{(F_0(R) + G(R))^2} \left |
q^2 \left (F_0''(R) + G''(R) + \frac{F_0'(R) + G'(R)}{R} \right )
\right .\\ & \left .- 2 (V_0(R) + W(R)) (F_0'(R) + G'(R))
 \right |. \\
\end{align*}
We use now that $k\e^{1-\a} \leq \e$ and that $\e R^{-2} \leq \e^{1- 2\a} \leq \e^{1-\a}k^{-1}$ and we obtain that
\begin{align*}
\e \big |R^2 \partial_{G}\widetilde{\mathcal{N}}_2[G,G',G'',W](R)\big | &\leq  M \e \left [q \e^{2(1-\a)} + q^2
 + \e ^{1-\a}q^2  +  \e^{2-2\a} + q\e^{1-\a}+
 q^{-1} \e^{3-3\a}+ \e^{2-2\a} \right ] \\ &\leq
M\e q^2,
\end{align*}
where again we have used that $\e^{1-\a}\leq q$.
This gives:
$$
\e \big |R \partial_{G}\widetilde{\mathcal{N}}_2[G,G',G'',W](R)\big | \leq M\e^{1-\a} q^2.
$$
Therefore
$$
\e |R^3 N_2(R) | \leq Mq^2\e^{1-\a}\|G_1- G_2\|_2 ,
$$
and we obtain $\e \|\mathcal{S}_2[ N_2]\|_2 \leq \e \|N_2\|_3 \leq M \e^{1-\a} q^2  \|G_1- G_2\|_2$.

With respect to $N_3$, we have that
\begin{align*}
\partial N_\lambda(R)&:=\partial_{G'} \widetilde{\mathcal{N}}_2 [G_\lambda,G_\lambda',G_\lambda'', W_\lambda ](R)  - 2 \frac{V_0(R)}{F_0(R)+G_\lambda(R)} \\ & = -\frac{q^2}{R(F_0(R)+ G_\lambda(R))} + 2\frac{W_\lambda(R)}{F_0(R)+G_\lambda(R)}.
\end{align*}
Then
$$
\e \big  | \partial N_\lambda(R)  \big |
\leq \frac{Mq^2 \e }{R} + \frac{M\e^2 }{R^2}
\leq M \e (q^2 + \e^{1-\a})\frac{1}{R}
\leq M\e q \frac{1}{R},
$$
that implies that
\begin{align*}
\e \big |R^3 \partial N_\lambda(R) \big |   |G_1'(R) - G_2'(R)|  \leq  M\e q \|G_1'-G_2'\|_2 ,
\end{align*}
and therefore
\begin{equation}\label{boundN31outer}
\e \left \| \mathcal{S}_2 \left [
(G_1' - G_2')\int_{0}^1 \partial N_\lambda d\lambda
\right ] \right \|_2 \leq \e\left  \|(G_1' - G_2') \int_{0}^1 \partial N_\lambda d\lambda  \right  \|_3 \leq M  \e q \|G_1'-G_2'\|_2.
\end{equation}
We point out that
$$
\mathcal{S}_2 \left [ V_0(R) (G_1'- G_2') \int_{0}^1 \frac{1}{F_0+ G_\lambda} d\lambda \right ]
= \mathcal{A} \left [  (G_1'- G_2') \int_{0}^1 \frac{1}{F_0+ G_\lambda} d\lambda \right ],
$$
and then
\begin{equation}\label{boundN32outer}
\e  \left \|\mathcal{S}_2 \left [ V_0(R) (G_1'- G_2') \int_{0}^1 \frac{1}{F_0+ G_\lambda} d\lambda \right ] \right \|_2 \leq \e \|G_1' - G_2'\|_2 .
\end{equation}
Bounds~\eqref{boundN31outer} and~\eqref{boundN32outer} imply $\e\|\mathcal{S}_2[N_3]\|_2 \leq M \e \|G_1' -G_2'\|_2$.

Finally we deal with $N_4$. Using Lemma~\ref{lem:S22} with
$$
h(R)= \int_{0}^1 \frac{d\lambda }{F_0+G_\lambda}, \qquad \hat{h}_2(R) = G_1''- G_2''  , \qquad \hat{h}_1 =  G_1'-G_2',
$$
we have that
$$
\e \|\opg_2 [N_4] \|_2 \leq  \e q^2 \| h \hat{h}_1  \|_2 + \e q^2 \|\opg_2 [h' \hat{h}_1]\|_2 + 2 \e q^2 \|\mathcal{A}[\hat{h}_1 h]  \|_2.
$$
Then,
we obtain
$$
 \e q^2\| h \hat{h}_1  \|_2 \leq M \e q^2  \|G_1'-G_2'\|_2 ,
$$
and by Lemma~\ref{lem:AS2},
$$
\e q^2\|\mathcal{A}[\hat{h}_1 h]\|_2 \leq  M \e q^2 \|G_1'-G_2'\|_2.
$$
In addition,
\begin{align*}
\e |h'(R) \hat{h}_1(R)| & \leq \e  |G_1'(R)-G_2'(R)| \int_{0}^1 \frac{|F_0'(R)| + |G_\lambda'(R)|}{|F_0(R)+ G_\lambda(R)|^2 }  d\lambda \\
&\leq M  \e  \frac{k \e^{1-\a} + \e }{R^4} \|G_1'-G_2'\|_2\\ &  \leq
M \e q \frac{1}{R^3} \|G_1'-G_2'\|_2.
\end{align*}
Then, using Lemma~\ref{lem:W2} $\|\opg_2 [h'\hat{h}_1]\|_2 \leq \|h' \hat{h_1}\|_3$, we obtain
$$
\e\|\opg_2[N_4] \|_2 \leq M  \e q  \|G_1'-G_2'\|_2,
$$
which finishes the proof of bound~\eqref{boundlipS2N2outer}.

Now we define $\varphi_1 =  \op_1[G_1,W_1]$, $\varphi_2= \op_1[G_2,W_2]$. By bound~\eqref{boundlipS2N2outer}, using that the Lipschitz constant of $\op_1$ is $M\e^{1-\a}/q$ and also \eqref{eq:lipF1outer}, we have that
\begin{align*}
\e \|\mathcal{S}_2[\mathcal{N}_2[\varphi_1,W_1]]
- \mathcal{S}_2[\mathcal{N}_2[\varphi_2,W_2]] \|_2\leq &M \e^{1-\a} \llfloor (G_1,W_1)-(G_2,W_2)\rrfloor + \e^{1-\a} \|\varphi_1-\varphi_2\|_2 \\&+ \e \|\varphi_1'-\varphi_2'\|_2 \\
\leq &M\e^{1-\a} \llfloor (G_1,W_1)-(G_2,W_2)\rrfloor \\&+ \e^{1-\a} q^{-1}
\llfloor (G_1,W_1)-(G_2,W_2)\rrfloor,
\end{align*}
and the proof of Lemma~\ref{lem:Lipschitzouter} is finished.

\section{Existence result in the inner region. Proof of Theorem~\ref{th:innermatching}}
\label{sec:proofInner}
We want to find solutions of~\eqref{originals} departing the origin that remain close to $(\fdin(r), \vdin(r))=(\f0(r),q\v0(r))$ defined by~\eqref{dominantterminner_0} where we recall that $\f0(r)$ is the unique solution of~\eqref{eq:introf0_0} and $\v0(r)$ is the solution of~\eqref{eq:introv0_0}:
\begin{equation}
\label{eq:f0v0}
\begin{aligned}
&\f0'' + \frac{\f0'}{r} - \f0\frac{n^2}{r^2} + \f0(1-\f0^2)=0, \qquad \f0(0), \qquad \lim_{r\to \infty} \f0(r)=1\\
 &\v0'+\frac{\v0}{r} + 2 \v0 \frac{\f0'}{\f0} + \esvi (1-\f0^2-k^2)=0,\qquad \v0(0)=0.
 \end{aligned}
\end{equation}
Then $\v0$ can be expressed (see~\eqref{expr:v_0}) as a function of $\f0(r)$ by writing
$$
\v0(r)=-\frac{1}{\esv r\f0^2(r)}\int_0^r \xi \f0^2(\xi) (1- \f0^2(\xi)- k^2) \dd \xi.
$$

The asymptotic and regularity properties of $\f0,\v0$ are given in Proposition~\ref{prop:f0v0_new} and  will be used along the proof of Theorem~\ref{th:innermatching}. Again, as in the previous section, Section~\ref{sec:provaOuter}, the proof of such result relies on a fixed point argument.

Let us now introduce the Banach spaces we shall be working with. For any $0<s_1$ and $\c>0$ we define $w(s)=\f0'(s/\sqrt{2})>0$,
$w_0(s)=\v0^2(s)\f0(s)>0$ and
$$
\Xin= \left\{\psi: [0,s_1]\to\mathbb{R}, \quad\psi\in\mathcal{C}^0( [0,s_1]),\quad \sup_{s\in [0,s_1]}\left|\frac{\psi(s)}{w(s)+\c w_0(s)}\right|<\infty\right\},
$$
endowed with the norm
$$
\|\psi\|=\sup_{s\in[0,s_1]} \left|\frac{\psi(s)}{w(s)+\c w_0(s)}\right|.
$$
We stress that, in $\Xin$ the norm $\|\cdot \|$ and
$$
\|\psi \|_{aux} = \sup_{s\in [0,s_*]} \frac{|\psi(s)|}{s^{n-1}} +
\sup_{s\in [s_*,s_1]} \left (\frac{1}{s^{3}}  + \c\frac{ |\log s|^2}{s^2} \right )^{-1}|\psi(s)|,
$$
for any given $s_*\in (0,s_1)$
are equivalent (see Lemma~\ref{prop:f0v0_new}).
We also introduce the Banach space
$$
\Yin= \left\{\psi: [0,s_1]\to\mathbb{R}, \quad\psi\in\mathcal{C}^0( [0,s_1]),\quad \|\psi\|_n<\infty \right\},
$$
where the norm $\|\cdot \|_n$ is defined by
$$
\|\psi \|_n= \sup_{s\in [0,s_*]} \frac{|\psi(s)|}{s^{n}} +
\sup_{s\in [s_*,s_1]} \left (\frac{1}{s^{3}}  +  \c \frac{ |\log s|^2}{s^2} \right )^{-1}|\psi(s)|,
$$
which satisfies that $\Yin \subset \Xin$.

Finally, for any fixed $m,l,\nu>0$, we define
$$
\Zin_{m}^{l,\nu}=\left\{\psi: [0,s_1]\to\mathbb{R}, \quad\psi\in\mathcal{C}^0( [0,s_1]),\quad \|\psi\|_{m}^{l,\nu}<\infty \right\},
$$
and the norm
$$
\|\psi \|_{m}^{l,\nu}= \sup_{s\in [0,s_*]} \frac{|\psi(s)|}{s^m}
+
\sup_{s\in [s_*,s_1]} \frac{|\psi(s)|s^{l}}{|\log s|^{\nu}}.
$$

From now on we will fix $s_*$ (independent of $q$  and $k$) as the minimum value which guarantees that, for $s\geq s_*$,
$\f0(s) \geq 1/2$ and the asymptotic expression~\eqref{asymKIn_new} is satisfied for $s\geq s_*$, namely
\begin{equation}\label{asympKIninner}
K_n(s)= \sqrt{\frac{\pi}{2 s}}e^{-s} \left (1+ \O\left (\frac{1}{s}\right )\right ),\qquad
I_n(s)= \sqrt{\frac{1}{2\pi s}}e^{s} \left (1+ \O\left (\frac{1}{s}\right ) \right ),
\qquad s\geq s_*.
\end{equation}

\subsection{The fixed point equation}
We denote by $\vb=v/q$ and we shall derive a system of two coupled fixed point equations equivalent to
\begin{subequations}
\begin{align}
f''+\frac{f'}{r}-f\frac{n^2}{r^2}+f(1-f^2-q^2\esvq \vb^2)&=0,\label{eqf}\\
f\vb'+f\frac{\vb}{r}+2\vb f'+ \esvi f(1-f^2-k^2)&=0. \label{eqv2}
\end{align}
\end{subequations}

We thus start by noting that since $q$ is small, we may write $(f,\vb)$ as a perturbation around $(\f0(r),\v0(r))$ of the form $(f,\vb)=(\f0(r)+g,\v0(r) + w)$. Therefore, using that $\f0,\v0$ are solutions of~\eqref{eq:f0v0}, equation~\eqref{eqf} can be expressed as
\begin{equation}
\label{Gin}
g''+\frac{g'}{r}-g\frac{n^2}{r^2}+g(1-3\f0^2(r))=\hat{H}[g,w],
\end{equation}
with
\[
\widehat{H}[g,w](r)=g^3+3g^2\f0(r)+q^2 \esvq(\v0(r) + w)^2  \left(g+\f0(r)\right),
\]
along with the initial condition $g(0)=0$. We also have that equation~\eqref{eqv2} can be written like:
\begin{equation}
\label{Vin}
w'+\frac{w}{r} + w\frac{\f0'}{\f0} = \esvi g(g+ 2 \f0) -
\frac{\v0 + w}{\f0(\f0 + g)}\left ( \f0 g' - \f0' g\right ) ,
\end{equation}
along with $w(0)=0$.

We now write the differential equations \eqref{Gin} and~\eqref{Vin} as a fixed point equation. We start by pointing out that, equivalently to what happens for the outer equations, one cannot explicitly solve the homogeneous linear problem associated to \eqref{Gin}. However, we shall conveniently modify the equation \eqref{Gin} to obtain a set of dominant linear terms at the left-hand-side for which we will have explicit solutions.

We first note that, as shown in \cite{Aguareles2011}, $\f0(r)$ very rapidly approaches the value of 1. Inspired by this, we define
\[
\widehat{\mathcal{E}}[g]:=g''+\frac{g'}{r}-g\frac{n^2}{r^2}+3g(1-\f0^2(r)),
\]
and therefore, \eqref{Gin} reads $\widehat{\mathcal{E}}[g]-2g =\widehat{H}[g,w](r)$, which motivates to perform the change $g=-\widehat{H}[0,0]/2+\Delta g$ into \eqref{Gin}.
Denoting by
$$
h_0=-\widehat{H}[0,0]/2=\frac12 q^2 \esvq\v0^2 \f0,
$$
$\Delta g$ is found to satisfy
\begin{equation*}
\Delta g''+\frac{\Delta g'}{r}-\Delta g\frac{n^2}{r^2}-2\Delta g= \widehat{H}[h_0+\Delta g]-\widehat{H}[0,0]-\widehat{\mathcal{E}}[h_0]-3\Delta g(1-\f0^2(r)),
\end{equation*}
along with $\Delta g(0)=0$.
Now we perform the change $s=\sqrt{2}r$ and we denote by
$\dg(s)=  \Delta g(s/\sqrt{2})$, $\dv(s)=w(s/\sqrt{2})$, $\ft_0(s)=\f0(s/\sqrt{2})$,
$\vt_0(s)=\v0(s/\sqrt{2})$ and $\tilde{h}_0(s)=h_0(s/\sqrt{2})$.
Therefore,
\begin{equation}
\label{dGins}
\dg''+\frac{\dg'}{s}-\dg\left(1+\frac{n^2}{s^2}\right)=\opN_1[\dg,\dv],
\end{equation}
where
\begin{equation}
\label{eq:N1}
\begin{split}
\opN_1[\dg,\dv](s)=& -\frac{3}{2} (1-\ft_0^2(s)) \dg + \frac{1}{2} \left ( {H}[\dg+\tilde{h}_0, \dv]- {H}[0,0]\right ) - \frac{1}{2}  {\mathcal{E}}[\tilde{h}_0],
\end{split}
\end{equation}
with
\begin{align*}
 {H}[g,\dv](s) &= g^3 + 3 g^2 \ft_0(s) + q^2 \esvq (\vt_0(s) + \dv)^2
(\ft_0(s) + g),
\\
{\mathcal{E}}[\tilde{h}_0](s)& =    \widehat{\mathcal{E}}[h_0](s\sqrt{2}) =2\left( \tilde{h}_0'' + \frac{\tilde{h}_0'}{s} - \frac{n^2}{s^2} \tilde{h}_0 \right)+
3 \tilde{h}_0 (1- \ft_0^2(s)).
\end{align*}

The homogeneous linear equation associated to~\eqref{dGins}, namely
$$
\dg'' + \frac{\dg'}{s} - \dg\left (1+ \frac{n^2}{s^2}\right ) =0,
$$
has solutions $K_n,I_n$, the modified Bessel functions. They satisfy that their wronskian is $1/s$. Therefore,   equation~\eqref{dGins} may also be written, for any $s_1>0$, like
\[
\begin{split}
\dg(s)& = K_n(s) \left (c_1 + \int_{s_1}^s \xi I_n(\xi) \opN_1[\dg,\dv](\xi)\dd \xi \right )+
I_n(s) \left (c_2- \int_{s_1}^s \xi K_n(\xi) \opN_1[\dg,\dv](\xi) \dd \xi \right ),\\
\dg'(s)& = K_n'(s) \left (c_1 + \int_{s_1}^s \xi I_n(\xi)\opN_1 [\dg,\dv](\xi)\dd \xi \right )+
I_n'(s) \left (c_2- \int_{s_1}^s \xi K_n(\xi) \opN_1[\dg,\dv](\xi) \dd \xi \right ),
\end{split}
\]
where $c_1,c_2$ are so far free parameters.
It is well known (see~\eqref{asymKIn0_new}) that $K_n(s) \to \infty$ and $I_n(s)$ is zero as $s\to 0$, if $n\geq 1$. Then, in order to have solutions bounded at $s=0$ we have to impose
$$
c_1 - \int_{0}^{s_1} \xi I_n(\xi) \opN_1[\dg,\dv](\xi)\dd \xi=0.
$$
Therefore,
\begin{equation*}
\dg(s) = K_n(s)\int_{0}^s \xi I_n(\xi) \opN_1[\dg,\dv](\xi)\dd \xi+
I_n(s) \left (c_2+ \int_{s}^{s_1} \xi K_n(\xi) \opN_1[\dg,\dv](\xi) \dd \xi \right ).
\end{equation*}
For any $s_1>0$, we introduce the linear operator
\begin{equation*}
\widehat{\mathcal{S}}_1[\psi](s) =  K_n(s)\int_{0}^s \xi I_n(\xi)\psi (\xi)\dd \xi+
I_n(s) \int_{s}^{s_1} \xi K_n(\xi) \psi(\xi) \dd \xi.
\end{equation*}
We have proven the following result:
\begin{lemma}\label{lem:S1inner}
For any $\Cin\in \R$ we define
$$
 \mathbf{\delta \widehat{g}}_0 (s)=I_n(s) \Cin.
$$
Then, if $\dg$ is a solution of~\eqref{dGins} satisfying $\dg(0)=0$, there exists $\Cin$ such that
\begin{equation}\label{fixedpoint:firstgpre}
\dg = \mathbf{\delta \widehat{g}}_0 + \widehat{\mathcal{S}}_1\circ\opN_1[\dg,\dv].
\end{equation}
\end{lemma}

We emphasize that $\opN_1$, given in~\eqref{eq:N1}, has linear terms in $\dg$. In fact, we decompose
$$
\opN_1[\dg,\dv]=\mathcal{L}[\dgt] + \opR_1[\dg,\dv],
$$
with
\begin{equation}\label{def:R1inner}
\begin{aligned}
\mathcal{L}[\dg](s)& =-\frac{3}{2} (1- \ft_0^2(s)) \dg(s) ,  \\
\opR_1[\dg,\dv](s)&= \frac{1}{2}\left ({H}[\dg + \tilde{h}_0,\dv ] - {H}[0,0] \right ) - \frac{1}{2} {\mathcal{E}}[\tilde{h}_0].
\end{aligned}
\end{equation}
Therefore equation~\eqref{fixedpoint:firstgpre} is rewritten as
\begin{equation}\label{fixedpoint:firstgpre2}
\dg=\mathbf{\delta \widehat{g}}_0 + \widehat{\mathcal{S}}_1\circ \mathcal{L}[\dg] + \widehat{\mathcal{S}}_1\circ \opR_1[\dg,\dv],
\end{equation}
with $\mathbf{\delta \widehat{g}}_0$ defined in Lemma~\ref{lem:S1inner}.
\begin{lemma}\label{lem:opTinner}
There exist $0<c,L\leq 1$ such that for any $0<s_*<s_1$  the linear operator $\opT:=\widehat{\mathcal{S}}_1 \circ \mathcal{L}$ satisfies that $\opT: \Xin\to \Xin $
with $\|\opT\| \leq L<1$.

As a consequence $\mathrm{Id} - \opT$ is invertible.
\end{lemma}
\begin{proof}
In~\cite{AgBaSe2016} it is proven that the linear operator
$$
\opTt[h](s) := \frac{3}{2} K_{n}(s) \int_0^s  \xi I_n(\xi) (1- \ft_0^2(\xi)) h (\xi) \dd s +
\frac{3}{2} I_{n}(s) \int_s^\infty  \xi K_n(\xi) (1- \ft_0^2(\xi)) h(\xi)\dd s,
$$
is contractive, in the Banach space defined by
$$
\widetilde{\Xin} = \left \{ \psi:[0,\infty)\to \R,\; \psi\in C^0[0,\infty), \|\psi\|_{w}:=\sup_{s\geq 0} \frac{|\psi(s)|}{w(s)} <\infty\right \}.
$$
The proof is based on the fact that
\begin{equation*}
\begin{split}
\vert \opTt[h](s)\vert \leq &  \frac{3}{2} K_{n}(s)\int_{0}^{s}\xi I_{n}(\xi)
(1-\ft_0^2(\xi))
\| h \|_w w(\xi) \, d\xi
\\&+\frac{3}{2} I_{n}(s)\int_{s}^{\infty}\xi K_{n}(\xi)
(1-\ft_0^2(\xi)) \| h\|_w w(\xi)\, d\xi
\\&\leq  \| h \|_w T(s) \, ,
\end{split}
\end{equation*}
where the function $T$ is defined by
\begin{equation*}
T(s) :=  \frac{3}{2} K_{n}(s)\int_{0}^{s}\xi I_{n}(\xi) (1-\ft_0^2(\xi)) w(\xi)\, d\xi
+\frac{3}{2} I_{n}(s)\int_{s}^{\infty}\xi K_{n}(\xi) (1-\ft_0^2(\xi))\w(\xi)\, d\xi.
\end{equation*}
and satisfies $\|T\|_w=\widetilde{L} <1$.

Let now $h\in \Xin$:
\begin{equation}\label{TinnTR}
\begin{split}
\vert \opT [h](s) \vert \leq &  \frac{3}{2} K_{n}(s)\| h\| \int_{0}^{s}\xi I_{n}(\xi)
(1-\ft_0^2(\xi))
(w(\xi)+\c w_0(\xi)) \, d\xi
\\&+\frac{3}{2} I_{n}(s)\| h\| \int_{s}^{s_1}\xi K_{n}(\xi)
(1-\ft_0^2(\xi))  (w(\xi)+\c w_0(\xi) )\, d\xi
\\ \leq  & \| h \| \big (T(s) + R(s)\big ),
\end{split}
\end{equation}
where
$$
R(s) = \frac{3\c }{2} K_{n}(s)\int_{0}^{s}\xi I_{n}(\xi)
(1-\ft_0^2(\xi)) w_0(\xi) \, d\xi
+\frac{3\c }{2} I_{n}(s)\int_{s}^{s_1}\xi K_{n}(\xi)
(1-\ft_0^2(\xi))  w_0(\xi)\, d\xi.
$$
When $s\in [0,s_*]$,
$$
R(s) \leq \c M \left ( s^{-n}\int_{0}^s \xi^{2n+3}\,d\xi +
s^{n} \int_{s}^{s_*} \xi^{2}\,d\xi
+ s^n \int_{s_*}^\infty \xi K_{n}(\xi) \frac{|\log \xi|^2}{\xi^{2}}\, d\xi\right ) \leq \c Ms^n.
$$

For $s\in [s_*,s_1]$, using that $1-\f0^2(s) =\O(s^{-2})$
\begin{align*}
R(s) &\leq \c M \left (\frac{e^{-s}}{\sqrt{s}} \int_{0}^{s*} \xi^{2n+3}\, d\xi +
\frac{e^{-s}}{\sqrt{s}} \int_{s_*}^s e^{\xi} \frac{|\log \xi|^2}{\xi^{7/2}} \,d\xi +
\frac{e^s}{\sqrt{s}} \int_{s}^{s_1} e^{-\xi}
\frac{|\log \xi|^2}{\xi^{7/2}} \,d\xi \right)
\\ & \leq
\c M\left (\frac{e^{-s}}{\sqrt{s}} +\frac{|\log s|^2}{s^4}\right ) \leq \c M \frac{1}{s^3} \leq \c M (w(s) + \c w_0(s)).
\end{align*}
Therefore, using \eqref{TinnTR} one obtains
$$
\|\opT[h]\| \leq \|h\| (\widetilde{L} + \c b_0),
$$
where $b_0$ is a constant which is independent on $\c$.

Taking
$
\c  \leq \min \left \{ 1, \frac{1-\widetilde{L}}{2b_0}\right \}
$
so that
$L:=\widetilde{L} + \c b_0 \leq  \frac{\widetilde{L}+1}{2}<1$, the proof is finished.
\end{proof}

As a consequence of this lemma, equation~\eqref{fixedpoint:firstgpre2} can be expressed as
\begin{equation}\label{fixedpoint:firstg}
\dg=\mathbf{\delta g}_0 + \mathcal{S}_1 \circ \opR_1[\dg,\dv],\qquad
\mathbf{\delta g}_0:=\big (\mathrm{Id}-\opT\big )^{-1}[\mathbf{\delta \widehat{g}}_0],\qquad\mathcal{S}_1:=\big (\mathrm{Id} - \opT\big )^{-1} \circ \widehat{\mathcal{S}}_1,
\end{equation}
and we recall that $\mathbf{\delta \widehat{g}}_0 $ was defined in Lemma~\ref{lem:S1inner}.
\begin{lemma}\label{lem:deltag0inner}
There exists a function $I(s)$ satisfying
$$
I'(s_1)K_n(s_1)- I(s_1)K_n'(s_1)=\frac{1}{s_1}, \qquad |I(s_1)|,|I'(s_1)|\leq M \frac{1}{\sqrt{s_1}} e^{s_1},
$$
such that $\mathbf{\dg}_0(s)=I(s) \Cin$.
\end{lemma}
\begin{proof}
Recall that
$$
\mathcal{T}[h](s)=-\frac{3}{2} K_n(s) \int_{0}^s \xi I_n(\xi) (1- \ft_0(\xi))h(\xi) \dd \xi - \frac{3}{2} I_n(s) \int_{s}^{s_1}  \xi K_n(\xi) (1- \ft_0(\xi))h(\xi) \dd \xi.
$$
Since $\mathbf{\delta g}_0=\big (\mathrm{Id}-\opT\big )^{-1}[\mathbf{\delta \widehat{g}}_0]$, by definition of the operator $\mathcal{T}$ it is clear that
$$
\mathbf{\dg}_0(s) = \sum_{m\geq 0} \mathcal{T}^m[\mathbf{\delta \widehat{g}}_0](s),
$$
and therefore, $\mathbf{\dg}_0(s)=I(s) \Cin$. Notice that if $\Cin=0$, one can take $I(s)=I_n(s)$ and we are done. Assume then that $\Cin\neq 0$. Then,
from $\mathbf{\dg}_0-\mathcal{T}[\mathbf{\dg}_0]=\delta \widehat{\mathbf{g}}_0 = I_n(s) \Cin$, one deduce that
\begin{align*}
I_n(s) & =
I(s)   -\frac{3}{2} K_n(s)   \int_{0}^s \xi I_n(\xi) (1- \ft_0(\xi))I(s)  \dd \xi - \frac{3}{2} I_n(s)  \int_{s}^{s_1}  \xi K_n(\xi) (1- \ft_0(\xi))I(s) \dd \xi ,\\
I_n'(s) & =
I'(s)   -\frac{3}{2} K'_n(s)  \int_{0}^s \xi I_n(\xi) (1- \ft_0(\xi))I(s)  \dd \xi - \frac{3}{2} I_n'(s)  \int_{s}^{s_1}  \xi K_n(\xi) (1- \ft_0(\xi))I(s) \dd \xi.
\end{align*}
Therefore
$$
I'(s_1) K_n(s_1) - I(s_1)K_n'(s_1)=I_n'(s_1)K_n(s_1) - I_n(s_1)K_n(s_1)=s^{-1}_1.
$$
To finish, we observe that $\|\mathbf{\dg_0}\| \leq M \|\delta \widehat{\mathbf{g}}_0\| = M\|I_n\| \Cin$.
That is, $\|I\| \leq M \|I_n\|$. Then, from the asymptotic expression of $I_n$ in~\eqref{asympKIninner}, we deduce that $I_n (s) (w(s) + c w_0(s))^{-1}$ is an increasing function and then we have that
$\|I_n\| = (w(s_1) + cw_0(s_1))^{-1} I_n(s_1)$ and then
$$
|(w(s_1)+ c w_0(s_1))^{-1} I(s_1)| \leq M (w(s_1) + cw_0(s_1))^{-1} I_n(s_1),
$$
that implies that $|I(s_1) | \leq M I_n(s_1) \leq Ms_1^{-1/2} e^{s_1}$. The bound for $|I'(s_1)|$ comes from~\eqref{asympKIninner} and the fact that $I'(s_1)=  \big [s_1^{-1} + I(s_1)K_n'(s_1) \big ] (K_n(s_1))^{-1}$.
\end{proof}

Now we deal with equation~\eqref{Vin} which, along with the initial condition $w(0)=0$, is equivalent to
\begin{equation*}
w(r) = \frac{1}{ r \f0^2(r)}\int_{0}^r
\xi \f0^2(\xi) \left [\esvi g(g+ 2 \f0) -
\frac{\v0 + w}{\f0(\f0 + g)}\left ( \f0 g' - \f0' g\right ) \right ]\, d\xi.
\end{equation*}
Therefore, recalling that $g=h_0+\Delta g$, the function $\dv(s) = w(s/\sqrt{2})$ satisfies
\begin{equation}\label{fixedpoint:firstv}
\dv(s) =   \mathcal{S}_2 \circ  \opR_2[\dg,\dv],
\end{equation}
with
\begin{equation}\label{def:S2}
\mathcal{S}_2[\psi](s) = \frac{\sqrt{2}}{2 s \ft_0^2(s)}\int_{0}^s \xi \ft^2_0(\xi) \psi(\xi) \,d\xi,
\end{equation}
and
\begin{equation}\label{def:R2tilde}
\begin{aligned}
 {\opR}_2[\dg,\dv](s)  =&
 \esvi (\tilde{h}_0 + \dg) (2 \ft_0 + \tilde{h}_0 + \dg) \\ &+
 \frac{\vt_0 + \dv}{\ft_0 (\ft_0 + \tilde{h}_0 + \dg)}
 \left [\ft_0 (\tilde{h}_0' + \dg') -\ft_0' (\tilde{h}_0 + \dg) \right ].
 \end{aligned}
\end{equation}
In view of~\eqref{fixedpoint:firstg} and~\eqref{fixedpoint:firstv}, we are looking for solutions of the fixed point equation associated. However, as for the \textit{outer region}, for several technical reasons, we consider instead the equivalent Gauss-Seidel version of the fixed point equation given by
\begin{equation}\label{fixedpointinner}
(\dg ,\dv )= \opF[\dg ,\dv ],
\end{equation}
where $\opF=(\opF_1,\opF_2)$ is
\begin{equation}\label{def:fixedpointoperatorinner}
\opF_1 =  \mathbf{\delta g_0} +
\opL_1\circ \opR_1 ,\qquad
\opF_2[\delta g, \delta v]= \opL_2 \circ  \opR_2[ \opF_1[\dg,\dv], \dv]
.
\end{equation}

\begin{remark}\label{rem:ho}
We note that $\ft_0 + \tilde{h}_0 +\dg >0$ in the definition domain of $s$ in order for the operator $ {\opN}_2$ to be well defined. The following bounds, which are a straightforward consequence of Proposition~\ref{prop:f0v0_new}, will be crucial to guarantee the well-posedness of $ {\opN}_2$:
\begin{equation*}
\begin{aligned}
&|\tilde{h}_0(s) |\leq Mq^2 s^{n+2} ,\quad s\to 0,\qquad |\tilde{h}_0(s)|\leq M q^2 \frac{|\log s|^2}{s^2},\quad s\gg 1, \\
& {\mathcal{E}}[\tilde{h}_0](s)\sim M q^2 s^{n} ,\quad s\to 0,\qquad | {\mathcal{E}}[\tilde{h}_0](s)|\leq M q^2 \frac{|\log s|^2}{s^4} \leq M q^2 \frac{1}{s^3},\quad s\gg 1.
\end{aligned}
\end{equation*}
Moreover $|\tilde{h}_0'(s)|\leq Mq^2 |\log s|^2 s^{-3}$ for $s \gg 1$.
\end{remark}
In what follows, we simplify the notation by dropping the symbol $\,\widetilde{\,}\,$ of $\ft_0,\vt_0$ and $\tilde{h}_0$.
Now we reformulate Theorem~\ref{th:innermatching} to adapt it to the fixed point setting.
\begin{theorem}  \label{prop:innerfixedpoint}
 Let $\eta >0$, $0<\mu_0<\mu_1$
and take $\e=\mu e^{-\frac{\pi}{2nq}}$ with $\mu_0\le \mu \leq \mu_1$.
There exist $q_0=q_0(\mu_0,\mu_1,\eta)>0$ and
$\rho_0=\rho_0(\mu_0,\mu_1,\eta)>0$, $M=M(\mu_0,\mu_1,\eta)>0$ such that, for any $q\in[0,q_0]$ and
\begin{equation*}
0<\rho<\rho_0,
\end{equation*}
 taking $s_1$ as:
 \begin{equation*}
 s_1=e^{\frac{\rho}{q}},
\end{equation*}
if $\Cin$ satisfies
\begin{equation}
\label{eq:condicioc2}
s_1^{3/2} e^{s_1} |\Cin|\le  \eta \rho^2    \ ,
\end{equation}
then there exists a family of solutions $(\dg(s,\Cin),\dv(s,\Cin))$ of the fixed point equation~\eqref{fixedpointinner} defined for $0\leq s \leq s_1$ which satisfy
$$
 \|\dg\| +  \|\dg'\|+\|\dv\|_1^{1,3} \leq M q^2.
$$
The function $\dg $ can be decomposed as
$$
\dg(s,\Cin)=  \dg_0(s,\Cin)+\dg_1(s,\Cin),
$$
with $\dg_0(s,\Cin)= I(s) \Cin + \widetilde{\dg}_0(s)$ and $I(s)$ is a function satisfying $I'(s_1) K_n(s_1) - I(s_1) K_n'(s_1)=s^{-1}_1$. Moreover
\begin{itemize}
    \item [(i)]there exists $q_*=q_*(\mu_0,\mu_1)$ and $M_0=M_0(\mu_0,\mu_1)$ such that for $q\in [0,q_0*]$
    $$
    \|\widetilde{\dg}_0\|+ \|\widetilde{\dg}_0'\| \leq M_0 q^2,
    $$
    \item[(ii)]and for $q\in [0,q_0]$
    $$
\|\dg_1\|,\,\|\dg_1'\|\leq M q^2 \rho^2.
$$
\end{itemize}
\end{theorem}
As we did in the \textit{outer region}, we prove this proposition in three main steps.
We first study the continuity of the linear operators $\opL_1,\opL_2$ in Section~\ref{subsec:linearinner} in the defined Banach spaces. After that, in Section~\ref{subsec:independentinner} we study
$\op[0,0]$
and finally, in Section~\ref{subsec:lipinner} we prove that the operator $\mathcal{F}$ is Lipschitz.

From now on, we fix $\eta,\mu_0,\mu_1$, we will take $q_0,\rho_0$ as small as we need and $\Cin$ a constant  satisfying~\eqref{eq:condicioc2}. As a convention, in the proof there appear a number of different constants, depending on $\eta,\mu_0,\mu_1$ but independent of $q$ which, to simplify the notation, will all be simply denoted as $M$.

\subsection{The linear operators} \label{subsec:linearinner}
The following results provide bounds and differentiability properties of the linear operators $\opL_1,\opL_2$ defined in \eqref{fixedpoint:firstg} and \eqref{def:R2tilde}.
\begin{lemma}\label{lem:Linear1inner}
Let  $s_1, c$ be such that $0<s_*<s_1$ and $0<\c\leq 1$, and let $\psi \in \Xin$. Then, the function $\opL_1[\psi]$ is a differentiable function in $(0,s_1)$ such that $\opL_1[\psi]\in \Yin\subset \Xin$, $\opL_1[\psi]' \in \Xin$ and
$$
\|\opL_1[\psi] \|_n \leq M\|\opL_1[\psi] \|\leq M\|\psi\|,\qquad  \|\opL_1[\psi]'\| \leq M \|\psi\|,
$$
for $M$ a constant independent of $s_1,s_0,\c$.
\end{lemma}
\begin{proof}
Let $\psi \in \Xin$. One has that
$$
\big |\opL_1[\psi](s)\big | \leq M\|\psi\| \left [
K_n(s) \int_{0}^s \xi I_n(\xi) (w(\xi)+\c w_0(\xi))d\xi +
I_n(s) \int_{s}^{s_1}\xi K_n(\xi)(w(\xi)+\c w_0(\xi)) d\xi\right ],
$$
where we have used that $\|\big (\mathrm{Id}-\opT\big )^{-1}\|\leq M$.
If $s\in [0,s_*]$, then
$$
\big |K_{n}(s)\big | \leq M s^{-n}, \qquad
\big |I_n(s) \big |\leq M s^{n}, \qquad w(s) + \c w_0(s) \leq M s^{n-1},
$$
and therefore,
\begin{equation*}
\begin{aligned}
\big | \opL_1[\psi](s) \big | &\leq M\| \psi\| \left (s^{-n} \int_{0}^s \xi^{2n}\,d\xi + s^n \int_{s}^{s_*} 1 \,d\xi +
s^n \int_{s_*}^{s_1}  \xi K_{n}(\xi) \, d\xi \right )
\leq M\|\psi\| s^n,
\end{aligned}
\end{equation*}
where we have used that
$$
\int_{s_*}^{s_1}  \xi K_{n}(\xi) \, d\xi \leq
\int_{s_*}^{\infty}  \xi K_{n}(\xi) \, d\xi
\leq M.
$$
When $s\in [s_*,s_1]$
\begin{align*}
\big | \opL_1[\psi](s) \big | \leq & M\| \psi\| \left (\frac{e^{-s}}{\sqrt{s}} \int_{0}^{s_*} \xi^{2n}\,d\xi +
\frac{e^{-s}}{\sqrt{s}} \int_{s_*}^{s} \sqrt{\xi} e^{\xi}
\left (\frac{1}{\xi^3} + \c\frac{(\log \xi)^2}{\xi^2}\right )
 \,d\xi \right .\\ & \left .+
\frac{e^s}{\sqrt{s}}\int_{s}^{s_1} \sqrt{\xi}e^{-\xi}
\left (\frac{1}{\xi^3} + \c  \frac{(\log \xi)^2}{\xi^2}\right )
\,d\xi \right )\\
& \leq M\| \psi\| \left (
\frac{1}{s^{3}} + \c \frac{|\log s|^2}{s^2} \right )
\leq M\| \psi\| (w(s) + cw_0(s)),
\end{align*}
which easily follows upon using that for any $\nu,l\in \N$,
$$
\int_{s_*}^s e^{\xi} \frac{|\log \xi|^l}{\xi^{\nu}}\, d\xi \leq Me^s \frac{|\log s|^{l}}{s^{\nu}},\qquad
\int_{s}^{s_1}e^{-\xi}\frac{|\log \xi|^l}{\xi^{\nu}}\,d\xi \leq M e^{-s}\frac{|\log s|^{l}}{s^{\nu}}.
$$
Therefore, $\|\opL_1[\psi]\|_n\leq  M \|\psi\|$.

As for $\opL_1[\psi]'$, we notice that
$$
 \big (\mathrm{Id}-\opT\big )\circ \opL_1[\psi]'(s) =  K_n'(s)\int_0^s \xi I_n(\xi) \psi(\xi)\, d\xi + I_n'(s) \int_{s}^{s_1} \xi K_n(\xi) \psi(\xi)\, d\xi,
$$
and so analogous computations as the ones for $\opL_1[\psi]$ lead to the result.
\end{proof}
\begin{lemma}\label{lem:Linear2inner}
Les us fix $s_1$ such that $0<s_*<s_1$.
Then if $\psi \in \Zin_{0}^{2,l}$, the function $\opL_2[\psi]$, defined in \eqref{def:S2}, is a differentiable function in $(0,s_1)$ such that $\opL_2[\psi]\in \Zin_{1}^{1,l+1}$  and
$$
\|\opL_2[\psi] \|_1^{1,l+1} \leq  M\|\psi\|_0^{2,l}.
$$
In addition, if $\psi \in \Zin_{0}^{\nu,l}$, with $\nu>2$, the function $\opL_2[\psi]$ is a differentiable function in $(0,s_1)$ such that $\opL_2[\psi]\in \Zin_{1}^{1,0}$ and
$$
\|\opL_2[\psi] \|_1^{1,0} \leq  M\|\psi\|_0^{\nu,l}.
$$
The constant $M>0$ does not depend on $s_1$.
\end{lemma}
\begin{proof}
Let $\psi\in \Zin_{0}^{2,l}$. We have that, if $s\in [0,s_*]$
$$
\left |\opL_2[\psi] \right |\leq \frac{\sqrt{2}}{2 s \f0^2(s)}\int_{0}^s \xi \f0^2(\xi)|\psi(\xi)|\, d\xi
\leq M \|\psi\|_{0}^{2,l} \frac{1}{s^{2n+1}} \int_{0}^s \xi^{2n+1} \, d\xi
\leq M \|\psi\|_{0}^{2,l} s.
$$
When $s\in [s_*,s_1]$,
\begin{align*}
\left |\opL_2[\psi] \right |\leq &\frac{1}{s \f0^2(s)}\int_{0}^{s_*} \xi \f0^2(\xi)|\psi(\xi)|\, d\xi + \frac{1}{s \f0^2(s)}\int_{s_*}^s \xi \f0^2(\xi)|\psi(\xi)|\, d\xi \\
\leq &  \frac{M}{s}\|\psi\|_0^{2,l}+ \frac{M}{s} \|\psi\|_0^{2,l}\int_{s_*}^s
\frac{(\log \xi)^l}{\xi}\, d\xi \leq M \|\psi\|_0^{2,l} \left (\frac{1}{s}+ \frac{|\log s|^{l+1}}{s} \right ).
\end{align*}
Finally, let $\psi \in \Zin_{0}^{\nu,l}$ with $\nu>2$. Then for $s\in [0,s_*]$
$$
\left |\opL_2[\psi] \right |\leq \frac{1}{s \f0^2(s)}\int_{0}^s \xi \f0^2(\xi)|\psi(\xi)|\, d\xi
\leq M \|\psi\|_{0}^{\nu,l} \frac{1}{s^{2n+1}} \int_{0}^s \xi^{2n+1} \, d\xi
\leq M \|\psi\|_{0}^{\nu,l} s,
$$
and if $s\in [s_*,s_1]$,
\begin{align*}
\left |\opL_2[\psi] \right |\leq &\frac{1}{s \f0^2(s)}\int_{0}^{s_*} \xi \f0^2(\xi)|\psi(\xi)|\, d\xi + \frac{1}{s \f0^2(s)}\int_{s_*}^s \xi \f0^2(\xi)|\psi(\xi)|\, d\xi \\
\leq &  \frac{M}{s} \|\psi\|_{0}^{\nu,l}+ \frac{M}{s}\|\psi\|_{0}^{\nu,l} \int_{s_*}^s
\frac{(\log \xi)^l}{\xi^{\nu-1}}\, d\xi  \leq  \|\psi\|_{0}^{\nu,l} \frac{M}{s}.
\end{align*}
\end{proof}

\subsection{The independent term}\label{subsec:independentinner}
We now deal with the first iteration of the fixed point procedure given by the equation~\eqref{fixedpointinner}, namely we study $\mathcal{F}[0,0]$.
\begin{lemma}\label{lem:independentterminner}
 Let $0<\c\leq 1$ as in Lemma~\ref{lem:opTinner},
let  $0<\mu_0<\mu_1$
and take $\e=\mu e^{-\frac{\pi}{2nq}}$ with $\mu_0\le \mu \leq \mu_1$.
There exist $q^*=q^*(\mu_0,\mu_1)>0$ and  $M=M(\mu_0,\mu_1)>0$ such that, for any $q\in[0,q^*]$ and $0<\rho<\frac{\pi}{2n}$, for  $0<s_*<s_1\leq e^{\frac{\rho}{q}}$, given $\eta>0$ and $\Cin$ satisfying~\eqref{eq:condicioc2}, the function $(\dg_0, \dv_0)=\opF[0,0]$ belongs to $\Xin \times \Zin_{1}^{1,3}$, $\dg_0$ is a differentiable function belonging to $\Xin$ and
$$
\|\dg_0'\|,\|\dg_0\| \leq M(1+\eta)  q^2   ,\qquad \|\dv_0\|_{1}^{1,3} \leq
M(1+\eta)q^2 \esvi .
$$
Furthermore, $\dg_0\in \Yin $ with
$\|\dg_0\|_n \leq  M(1+\eta)q^2 $, and $\dv_0 \in \Zin_{1}^{1,1}$ with
$\|\dv_0\|_{1}^{1,1} \leq M \rho^2\esvi$.
\end{lemma}
\begin{proof}
Notice that $s_1k<1$ if $q$ is small enough.
We have that $\dg_0= \mathbf{\delta g}_0+\opL_1 \circ \opR_1[0,0]$. We recall that $\mathbf{\delta g}_0 (s)= \big (\mathrm{Id}- \mathcal{T})^{-1}[ \mathbf{\delta \widehat{g}}_0]$ where $\mathbf{\delta \widehat{g}}_0(s)=  I_{n}(s)\Cin$.
Using that $I_n$ is an increasing positive function, that the norms $\| \cdot \|, \|\cdot \|_{aux}$ are equivalent  and that $I_n(s)=\O(s^n)$ as $s\to 0$,
$$
 \|\mathbf{\delta g}_0\|_n \leq M  \|\mathbf{\delta \widehat{g}}_0\| \leq M |\Cin| I_n(s_1)
(w(s_1)+c w_0(s_1))^{-1} \leq M |\Cin| I_n(s_1)
\left (\frac{1}{s_1^3} + \c \frac{|\log s_1|^2}{s_1^2}\right )^{-1}.
$$
Since $s_1>s_*$ the asymptotic expression ~\eqref{asympKIninner} for $I_n(s_1)$ applies and then, since $\Cin$ satisfies~\eqref{eq:condicioc2} we conclude that
$\|\mathbf{\delta g}_0\| \leq M  {\eta} q^2 $.

We now compute $\opR_1[0,0]$ (see \eqref{def:R1inner}):
\begin{align*}
\opR_1[0,0] &= -\frac{1}{2} \mathcal{E}[h_0] + \frac{1}{2} \big (H[h_0,0]-H[0,0] \big ) \\
    &= -\frac{1}{2} \mathcal{E}[h_0] +
    \frac{1}{2} \big (h_0^3 + 3 h_0^2f_0 + q^2 \esvq \v0^2  h_0 \big ).
\end{align*}
Therefore, using the estimates for $\f0,\v0,h_0$ and $\mathcal{E}[h_0]$ in Proposition~\ref{prop:f0v0_new} and Remark~\ref{rem:ho} we have that
$$
\sup_{s\in [0,s_*]} |\opR_1[0,0](s)|\leq Mq^2  s^n, \qquad
\sup_{s\in [s_*,s_1]} |\opR_1[0,0](s)|\leq M\frac{q^2|\log s|^2}{s^4} + M\frac{q^4|\log s|^4}{s^4}.
$$
Using that for any $l\in \Z$, $|\log s|^l s^{-1}$ is bounded if $s\in (2,s_1)$ and that $s^{-3} \leq M w(s)$ we have that
$$
\sup_{s\in [s_*,s_1]} |\opR_1[0,0](s)| \leq Mq^2 \frac{1}{s^3}
\leq Mq^2 (w(s)+\c w_0(s)).
$$
As a consequence $\opR_1[0,0] \in \Yin \subset \Xin$,
$\|\opR_1[0,0]\|\leq C q^2$ and using Lemmas~\ref{lem:opTinner} and~\ref{lem:Linear1inner}
$$
\big \|\opL_1\big [ \opR_1[0,0]\big ] \big \|_n \leq M \big \|\opL_1\big [ \opR_1[0,0]\big ] \big \|\leq
M \|\opR_1[0,0]\| \leq Mq^2 .
$$
Moreover, $\big \|\opL_1\big [ \opR_1[0,0]\big ]' \big \|\leq Mq^2$.

We deal now with $\dv_0$. First we notice that
$\f0+ h_0+\dg_0>0$. Indeed, we have that, for $s\in [0,s_*]$ $\f0(s)\geq M|s|^{n}$ for some positive
constant $M$ (see Proposition~\ref{prop:f0v0_new}). Therefore, if $q$ is small enough:
$$
\f0(s)+h_0(s) + \dg_0(s) \geq Cs^{n} - M q^2 |s|^{n+2} - M q^2 |s|^{n}>0.
$$
For $s\geq s_*$ since $\f0(s) \geq 1/2$, taking $q$ small enough:
$$
\f0(s)+h_0(s) + \dg_0(s) \geq \frac{1}{2} - Mq^2 \frac{|\log s|^2}{s^2} - Mq^2 \frac{1}{s^3}-Mq^2 \frac{|\log s|^2}{s^2}>\frac{1}{4}.
$$
We conclude then that $\dv_0$ is well defined. Now we are going to prove that it belongs to $\Zin_{1}^{1,3}$. By definition $\dv_0=\opF_2[0,0]=\opL_2 \circ \opR_2[\dg_0,0]$ with $\opR_2$ defined by~\eqref{def:R2tilde}:
$$
\opR_2[\dg_0,0]= \esvi (h_0 +\dg_0)(2 \f0 +h_0+\dg_0)+
\frac{\v0}{\f0(\f0+h_0 + \dg_0)}
\big [\f0 (h_0'+ \dg_0') - \f0'( h_0+ \dg_0)\big ].
$$
Therefore, using that $\dg_0 \in \Yin$, for $s\in [0,s_*]$ we have that
$$
\big |\opR_2[\dg_0,0](s) \big |\leq M(1+{\eta})^2 \esvi (s^{2n}+1) \leq M(1+ {\eta}) q^2 \esvi.
$$
On the other hand, for $s\in [s_*,s_1]$,
$$
\big |\opR_2[\dg_0,0](s) \big |\leq M(1+ {\eta})\frac{q^2 |\log s|^2 }{\esv s^2 } +
M (1+ {\eta})\frac{q^2 |\log s|^3  }{\esv s^3}
\leq  M(1+ {\eta})\frac{q^2 |\log s|^2 }{\esv s^2 }.
$$
As a consequence $\opR_2[\dg_0,0]\in \Zin_{0}^{2,2}$ with norm
$\big \|\opR_2[\dg_0,0]\big \|_0^{2,2} \leq M(1+ {\eta}) q^2 \esvi $. Therefore, by Lemma~\ref{lem:Linear2inner}
$\dv_0 \in \Zin_{1}^{1,3}$ with norm $\|\dv_0\|_{1}^{1,3} \leq M(1+ {\eta}) q^2 \esvi $,
and thus,  for $s\leq s_1 \leq e^{\frac{\rho}{q}} $
$$
|\dv_0(s)|\leq M(1+ {\eta})q^2\frac{|\log s|^3}{\esv  s} \leq M(1+ {\eta}) \rho^2\frac{|\log s|}{\esv s}.
$$
\end{proof}
\subsection{The contraction mapping}\label{subsec:lipinner}
In what follows we shall show that the fixed point equation~\eqref{fixedpointinner} is a contraction in a suitable Banach space. We define the norm
$$
\llfloor (\dg,\dv)\rrfloor = \|\dg\|+ \|\dv\|_{1}^{1,3},
$$
in the product space $\Xin \times \Zin_{1}^{1,3}$ and
we notice that, under the conditions of Lemma~\ref{lem:independentterminner}, we have proved that  $\llfloor (\dg_0,\dv_0)\rrfloor \leq \kappa_0 q^2$, where $\kappa_0=\kappa_0(\mu_0,\mu_1,\eta)$.

\begin{lemma}\label{lem:lipFinner}
Let $\mu_0,\mu_,\eta,\Cin$ and $\mu$ as in Lemma~\ref{lem:independentterminner}
and take $\e=\mu e^{-\frac{\pi}{2nq}}$ with $\mu_0\le \mu \leq \mu_1$.
There exist $q_0=q_0(\mu_0,\mu_1,\eta)>0$ and  $M=M(\mu_0,\mu_1,\eta)>0$ such that, for any $q\in[0,q_0]$,  $0<\rho< \frac{\pi}{2n}$ and $0<s_*<s_1\leq e^{\frac{\rho}{q}}$, we have that if
$(\dg_1,\dv_1),(\dg_2,\dv_2)\in \Xin \times \Zin_{1}^{1,3}$ satisfying $\llfloor (\dg_1,\dv_1)\rrfloor, \llfloor (\dg_2,\dv_2)\rrfloor \leq 2\kappa_0 q^2$,
then
\begin{enumerate}
\item with respect to $\mathcal{F}_1$
$$
\|\opF_1[\dg_1,\dv_1] - \opF_1[\dg_2,\dv_2]\|
\leq  Mq^2 \|\dg_1 - \dg_2\| + M \c^{-1} \rho^2 \|\dv_1 - \dv_2\|_{1}^{1,3}.
$$
\item and for $\mathcal{F}_2$
$$
\|\opF_2[\dg_1,\dv_1] - \opF_2[\dg_2,\dv_2]\|
\leq M q^2  \esvi \|\dg_1 - \dg_2\| + M( \rho^2 \c^{-1} + q^2) \|\dv_1 -\dv_2\|_{1}^{1,3}.
$$
\end{enumerate}
\end{lemma}

The remaining part of this section is devoted to prove Theorem~\ref{prop:innerfixedpoint} (Section~\ref{subsec:proofofTheorem} below) and Lemma~\ref{lem:lipFinner} whose proof is divided into two technical sections, Sections~\ref{Lip:inner:technica:1} and~\ref{Lip:inner:technica:2}.

\subsection{Proof of Theorem~\ref{prop:innerfixedpoint}}\label{subsec:proofofTheorem}
The proof of the result is a straightforward consequence of the previous analysis.
We define $\mathcal{B}=\{(\dg,\dv)\in \Xin\times \mathcal{Z}_{1}^{2,3}, \; \llfloor (\dg,\dv)\rrfloor \leq 2\kappa_0 q^2\}$.
The Lipschitz constant of $\opF$ in $\mathcal{B}$, $\mathrm{lip }\opF $, satisfies that
$$
\mathrm{lip }\opF \leq M(\mu_0,\mu_1,\eta)\max\{q^2, \c^{-1} \rho^2 \} \leq \frac{1}{2},
$$
provided $q$ is small enough and $\c^{-1} \rho^2 <1/2$,
so that $\opF$ is a contraction.
In addition, if $\llfloor (\dg,\dv)\rrfloor \leq 2\kappa_0 q^2$, then
\begin{align*}
\llfloor \opF[\dg,\dv]\rrfloor & \leq \llfloor \opF[0,0] \rrfloor + \llfloor \opF[\dg,\dv]-\opF[0,0]\rrfloor
\leq \kappa_0 q^2  + \frac{1}{2} \llfloor (\dg,\dv)\rrfloor \\ & \leq \kappa_0 q^2 + \frac{1}{2}2\kappa_0 q^2 \leq 2\kappa_0q^2 .
\end{align*}
Therefore the operator $\opF$ sends $\mathcal{B}$ to itself.  The fixed point theorem assures the existence of solutions $(\dg,\dv) \in \mathcal{B}$, consequently satisfies:
$$
\llfloor (\dg,\dv)\rrfloor = \llfloor \opF[\dg,\dv]\rrfloor \leq 2\kappa_0 q^2,
$$
and, if $(\dg,\dv)=\mathcal{F}[\dg,\dv]$, then $\dg_1=\mathcal{F}_1[\dg,\dv]-\mathcal{F}_1[0,0]$ satisfies
$$
\|\mathcal{F}_1[\dg,\dv]-\mathcal{F}_1[0,0]\| \leq Mq^2 \|\dg\| + Mc^{-1} \rho^2 \|\dv\|_1^{1,3} \leq Mc^{-1}\rho^2 q^2,
$$
provided $q\ll \rho $.
The bound for $\|\dg_1'\|$ follows from the previous bound and Lemma~\ref{lem:Linear1inner}.
Therefore, also using Lemma~\ref{lem:deltag0inner},  Theorem~\ref{prop:innerfixedpoint} is now proven.

\subsection{Proof of Lemma~\ref{lem:lipFinner}  }

\subsubsection{The Lipschitz constant for $\mathcal{F}_1$}\label{Lip:inner:technica:1}

Let $(\dg_1,\dv_1), (\dg_2,\dv_2)$ belonging to $\mathcal{X}\times \mathcal{Z}_1^{1,3}$, be such that
$\llfloor(\dg_1,\dv_1)\rrfloor, \llfloor (\dg_2,\dv_2)\rrfloor \leq 2\kappa_0 q^2$.

From the definition of $\opF_1$ provided in \eqref{def:fixedpointoperatorinner}, definition of $\mathcal{R}_1 $ in~\eqref{def:R1inner} and by Lemma~\ref{lem:Linear1inner}, we have that
\begin{equation*}
   \|\opF_1[\dg_1,\dv_1] - \opF_1[\dg_2,\dv_2]\|  \leq
   M\|H[h_0 + \dg_1 , \dv_1] - H[h_0 + \dg_2 , \dv_2]\|.
\end{equation*}
Let
$\dg(\lambda)= \dg_2 + \lambda(\dg_1- \dg_2)$ and
$\dv(\lambda)= \dv_2 + \lambda(\dv_1 - \dv_2)$. Using the mean's value theorem:
\begin{equation}\label{meanHinner}
\begin{aligned}
H[h_0 + \dg_1 , \dv_1](s) - &H[h_0 + \dg_2 , \dv_2](s) = \\
&\int_{0}^1 \partial_1 H[h_0+\dg(\lambda),\dv(\lambda)](s) \big (\dg_1(s) - \dg_2(s)\big )\,d\lambda \\
&+\int_{0}^1 \partial_2 H[h_0+\dg(\lambda),\dv(\lambda)](s) \big (\dv_1(s) - \dv_2(s)\big )\,d\lambda.
\end{aligned}
\end{equation}
We have that $\|\dg(\lambda)\|\leq A q^2$
$\|\dv(\lambda)\|_{1}^{1,3}\leq B q^2$
and
\begin{align*}
\partial_1 H[h_0+\dg(\lambda),\dv(\lambda)](s)  = &3 (h_0(s)+\dg(\lambda)(s))^2  +6 (h_0(s) +\dg(\lambda)(s))\f0(s) \\ & +q^2 \esvq (\v0(s)+\dv(\lambda)(s))^2, \\
\partial_2 H[h_0+\dg(\lambda),\dv(\lambda)](s)  =&
2 q^2 \esvq (v_0(s) + \dv(\lambda)(s))(\f0(s) + h_0(s) + \dg(s)).
\end{align*}
Then, recalling that $\|h_0\|_{n+2}^{2,2}\leq M q^2$, we obtain that, if $s\in [0,s_*]$,
\begin{align*}
\big | \partial_1 H[h_0+\dg(\lambda),\dv(\lambda)](s)\big | \leq&  Mq^4  s^{2n-2} + M q^2 s^{n-1}
+M q^2 s^2  \leq M q^2 , \\
\big | \partial_2 H[h_0+\dg(\lambda),\dv(\lambda)](s)\big |\leq & Mq^2
\esv s^{n},
\end{align*}
and for $s\in [s_*,s_1]$, noticing that,
$$
|\v0(s) + \delta v(\lambda)(s)| \leq M(\frac{|\log s|}{s} + q^2 \frac{|\log s|^3}{s}) \leq M\frac{|\log s|}{s}  (1+ \rho^2 ) \leq
M\frac{|\log s|}{s}.
$$
Then
\begin{align*}
\big | \partial_1 H[h_0+\dg(\lambda),\dv(\lambda)](s)\big | \leq &
Mq^4 \frac{|\log s|^4}{s^4} + Mq^2 \frac{|\log s|^2}{s^2} +
M q^2 \frac{|\log s|^2}{s^2}\\  \leq &M q^2 \frac{|\log s|^2}{s^2}, \\
\big | \partial_2 H[h_0+\dg(\lambda),\dv(\lambda)](s)\big | \leq & M q^2 \esv  \frac{|\log s|}{s}.
\end{align*}

Using all these bounds in~\eqref{meanHinner} one finds that, for $s\in [0,s_*]$
\begin{align*}
\big |H[h_0 + \dg_1 , \dv_1](s) - H[h_0 + \dg_2 , \dv_2](s) \big |
    & \leq Mq^2   s^{n-1} \|\dg_1 - \dg_2\| +  M q^2 s^{n+1} \|\dv_1-\dv_2\|_1^{1,3} \\ & \leq M q^2 \esv  s^{n-1}   \big( \|\dg_1 - \dg_2\| +
    \|\dv_1 - \dv_2\|_{1}^{1,3}\big )\\
    & \leq M q^2 \esv    \big( \|\dg_1 - \dg_2\| +
    \|\dv_1 - \dv_2\|_{1}^{1,3}\big ),
\end{align*}
and for $s\in [s_*,s_1]$
\begin{align*}
\big |H[h_0 + \dg_1 , \dv_1](s) - H[h_0 + \dg_2 , \dv_2](s) \big |
     \leq &
Mq^2  \frac{|\log s|^2}{s^2} (w(s)+ \c w_0(s)) \|\dg_1 - \dg_2\| \\ &
+ M q^2 \esv \frac{|\log s|^4}{s^2} \|\dv_1-\dv_2\|_{1}^{1,3}.
\end{align*}
Notice that, for $s_*<s<s_1 $
$$
\frac{|\log s|^2}{s^2} \big(w(s) + \c w_0(s)\big )^{-1} \le M\frac{|\log s|^2}{s^2} \left (\frac{1}{s^3}+ \c \frac{|\log s|^2}{s^2}\right )^{-1} \leq M\left (\frac{1}{s |\log s|^2 } + \c\right )^{-1} \leq M
\frac{1}{\c}.
$$
In addition, if $s_1\leq e^{\frac{\rho}{q}}$, then
$$
q^2 |\log s|^2 \leq \rho^2.
$$
Therefore, since $|\log s|^2\leq Ms^2$,
\begin{align*}
\big |H[h_0 + \dg_1 , \dv_1](s) -& H[h_0 + \dg_2 , \dv_2](s) \big |
\leq M (w(s) + \c w_0(s)) q^2 \|\dg_1 - \dg_2\| \\&+
M \c^{-1} \rho^2 (w(s)+\c w_0(s)) \|\dv_1 - \dv_2\|_{1}^{1,3},
\end{align*}
which proves the first item in Lemma~\ref{lem:lipFinner}.

\begin{remark}
As a consequence, using Lemma \ref{lem:opTinner}, if $\dg,\dv\in \Xin \times \Zin_{1}^{1,3}$ with $\|\dg\|\leq  2\kappa_0 q^2 $, $\|\dv\|_{1}^{1,3} \leq  2 \kappa_0 q^2 $,
$$
\|
\opF_1[\dg,\dv]\|\leq \|\opF_1[0,0]\|+\|\opF_1[\dg,\dv]-
\opF_1[0,0]\|\leq   \kappa _0 q^{2} + M q^2\| \dg\| + Mc^{-1}\rho \|\dv\|_1^{1,3}\leq
2 \kappa_0 q^2,
$$
if $q$ is small enough.  The bound for the derivative is a consequence of Lemma~\ref{lem:Linear1inner}.
\end{remark}

\subsubsection{The Lipschitz constant for $\mathcal{F}_2$}\label{Lip:inner:technica:2}

We recall that $\opF_2[\dg,\dv]=\opL_2\circ \opR_2[\mathcal{F}_1[\dg,\dv],\dv]$ where the operator $\opR_2$, defined in~\eqref{def:R2tilde}.
We rewrite $\opR_2=\opP+ \opP_1 \cdot \opP_2$ with
\begin{align*}
\opP[\dg,\dv]&=(h_0+\dg)(2\f0+h_0+\dg), \\
\opP_1[\dg,\dv] &= \frac{\v0+\dv}{\f0(\f0+h_0+\dg)},
\\
\opP_2[\dg] &= \f0(h_0'+\dg') - \f0'(h_0+\dg).
\end{align*}
For $(\dg_1,\dv_1), (\dg_2,\dv_2)$ be such that
$\llfloor(\dg_1,\dv_1)\rrfloor, \llfloor (\dg_2,\dv_2)\rrfloor \leq 2\kappa_0 q^2$, 
we denote $\overline{\dg}_j = \opF_1[\dg_j,\dv_j]$, $j=1,2$.

We recall that  $\|h_0\|_{n+2}^{2,2}\leq M q^2$ and we shall deal separately with $\opP,\opP_1\cdot \opP_2$. Starting with $\opP$,
\begin{align*}
\big |\opP[\overline{\dg}_1,\dv_1](s)- \opP[\overline{\dg}_2,\dv_2](s)\big | \leq &
\esvi\big[2 |\overline{\dg}_1(s)-\overline{\dg}_2(s)|\cdot |\f0(s)+h_0(s)| \\ &+ |\overline{\dg}_1(s)+\overline{\dg}_2(s)|\cdot |\overline{\dg}_1(s)-\overline{\dg}_2(s)| \big ].
\end{align*}
Therefore, when $s\in[0,s_*]$,
$$
\big |\opP{}[\overline{\dg}_1,\dv_1](s)- \opP[\overline{\dg}_2,\dv_2](s)\big |\leq
M\esvi \|\overline{\dg}_1-\overline{\dg}_2\|s^{2n-2}\leq
M\esvi \|\overline{\dg}_1-\overline{\dg}_2\|,
$$
and for $s\in [s_*,s_1]$
\begin{align*}
\big |\opP[\overline{\dg}_1,\dv_1](s)- \opP[\overline{\dg}_2,\dv_2](s)\big |& \leq
M\esvi\|\overline{\dg}_1-\overline{\dg}_2\| (w(s) + \c w_0(s)) \\ & \leq
M\esvi\|\overline{\dg}_1-\overline{\dg}_2\|
\left (\frac{1}{s^3} + \c \frac{|\log s|^2}{s^2}\right ).
\end{align*}
As a consequence
$$
\|\opP[\overline{\dg}_1,\dv_1](s)- \opP[\overline{\dg}_2,\dv_2]\|_{0}^{2,2} \leq M\esvi \|\overline{\dg}_1 - \overline{\dg}_2\|,
$$
and by Lemma~\ref{lem:Linear2inner} and the first item of Lemma~\ref{lem:lipFinner},
\begin{equation}\label{bound:N1innerLip}
\|\opL_2\big [\opP[\overline{\dg}_1,\dv_1]\big ]- \opL_2\big [\opP[\overline{\dg}_2,\dv_2]\big ]\|_{1}^{1,3} \leq  M\esvi q^2  \|\dg_1 - \dg_2\| +M \c^{-1}\rho^2 \|\dv_1-\dv_2\|_{1}^{1,3}.
\end{equation}

Now we deal with $\widehat{\opP}:=\opP_1\cdot \opP_2$.
Using the mean value Theorem as described in~\eqref{meanHinner} yields:
\begin{equation}\label{def:incrementN2}
\begin{aligned}
\widehat{\opP}[\overline{\dg}_1,\dv_1]&- \widehat{\opP}[\overline{\dg}_2,\dv_2] =
\opP_1[\overline{\dg}_1,\dv_1]  \big(\opP_2[\overline{\dg}_1]- \opP_2[\overline{\dg}_2]\big ) \\ &+ \opP_2[\overline{\dg}_2] \big  (\opP_1[\overline{\dg}_1,\dv_1]- \opP_1[\overline{\dg}_2,\dv_2]\big ) \\
&=\opP_1[\overline{\dg}_1,\dv_1] \big (\f0(\overline{\dg}_1'-\overline{\dg}_2') -
\f0'(\overline{\dg}_1-\overline{\dg}_2)\big ) \\
&+\opP_2[\overline{\dg}_2] \left((\overline{\dg}_1-\overline{\dg}_2) \int_{0}^1
\partial_1 \opP_1[\overline{\dg}(\lambda) ,\dv(\lambda)] \, d\lambda
\right.\\
& +\left.(\dv_1-\dv_2)\int_{0}^1
\partial_2 \opP_1[\overline{\dg}(\lambda) ,\dv(\lambda)]\right) \, d\lambda ,
\end{aligned}
\end{equation}
where we denote by
$
\overline{\dg}(\lambda)=\lambda \overline{\dg}_1+(1-\lambda)\overline{\dg}_2$ and analogously for $\dv(\lambda)$.
We emphasize now that $\overline{\dg}_j$ is a differentiable function since
$\overline{\dg}_j=\opF_1[\dg_j,\dv_j] = \opL_1\circ \opR_1 [\dg_j,\dv_j]$ and by
Lemma~\ref{lem:Linear1inner}, the linear operator $\opL_1$ converts continuous functions into differentiable ones. Moreover, $\overline{\dg}_j \in \Yin$ and this implies that
for $s\in[0,s_*]$
$$
\f0(s)+h_0(s)+\overline{\dg}(s) \geq M s^n,
$$
while for $s\in [s_*,s_1]$, using that $\f0(s)\geq 1/2$,
we have that $\f0(s)+h_0(s)+\overline{\dg}(s) \geq 1/4$ if $q$ is small enough. Taking this into account one can now bound the terms in ~\eqref{def:incrementN2}.

For $s\in [0,s_*]$
\begin{align*}
\big |\opP_1[\overline{\dg}_1,\dv_1](s)\f0(s)(\overline{\dg}_1'(s)-\overline{\dg}_2'(s)) \big |
&\leq M\esvi \|\overline{\dg}_1'-\overline{\dg}_2'\|
\leq M\esvi \|\overline{\dg}_1-\overline{\dg}_2\|,\\
\big |\opP_1[\overline{\dg}_1,\dv_1](s) \f0'(s)(\overline{\dg}_1(s)-\overline{\dg}_2(s))\big |
&\leq M\esvi \|\overline{\dg}_1-\overline{\dg}_2\|,
\end{align*}
and
\begin{align*}
\left |\opP_2[\overline{\dg}_2](s) (\overline{\dg}_1(s)-\overline{\dg}_2(s)) \int_{0}^1
\partial_1 \opP_1[\overline{\dg}(\lambda) ,\dv(\lambda)](s) \, d\lambda \right | &\leq
M\esvi q^2
\|\overline{\dg}_1-\overline{\dg}_2\| ,\\
\left| \opP_2[\overline{\dg}_2](s) (\dv_1(s)-\dv_2(s)) \int_{0}^1
\partial_2 \opP_1[\overline{\dg}(\lambda) ,\dv(\lambda)](s)\, d\lambda\right | &\leq
Mq^2
\|\dv_1 - \dv_2\|_{1}^{1,3}.
\end{align*}
Then for $s\in [0,s_*]$
\begin{equation}\label{bound:N2inner}
\begin{aligned}
\big |\widehat{\opP}[\overline{\dg}_1,\dv_1](s)- \widehat{\opP}[\overline{\dg}_2,\dv_2](s)\big | \leq
M\esvi \|\overline{\dg}_1-\overline{\dg}_2\| +
Mq^2
\|\dv_1 - \dv_2\|_{1}^{1,3} .
\end{aligned}
\end{equation}
When $s\in [s_*,s_1]$, using that $s_1=e^{\frac{\rho}{q}}$ and that
$$
|\dv_j(s)| \leq 2\kappa_0 q^2  |\log s|^3 s^{-1} \leq 2\kappa_0 \rho^2  |\log s|s^{-1},
$$
we obtain that
\begin{align*}
\big |\opP_1[\overline{\dg}_1,\dv_1](s)\f0(s)(\overline{\dg}_1'(s)-\overline{\dg}_2'(s)) \big |
&\leq M\esvi \frac{|\log s|^3}{s^3} \|\overline{\dg}_1'-\overline{\dg}_2'\|
\leq M\esvi \frac{|\log s|^3}{s^3} \|\overline{\dg}_1-\overline{\dg}_2\,\\
\big |\opP_1[\overline{\dg}_1,\dv_1](s) \f0'(s)(\overline{\dg}_1(s)-\overline{\dg}_2(s))\big |
&\leq M\esvi
\frac{|\log s|^3}{s^6}\|\overline{\dg}_1-\overline{\dg}_2\|,
\end{align*}
and
\begin{align*}
\left |\opP_2[\overline{\dg}_2](s) (\overline{\dg}_1(s)-\overline{\dg}_2(s)) \int_{0}^1
\partial_1 \opP_1[\overline{\dg}(\lambda) ,\dv(\lambda)](s) \, d\lambda \right | &\leq
M\esvi q^2   \frac{|\log s|^5}{s^5} \|\overline{\dg}_1-\overline{\dg}_2\| \\
\left| \opP_2[\overline{\dg}_2](s) (\dv_1(s)-\dv_2(s)) \int_{0}^1
\partial_2 \opP_1[\overline{\dg}(\lambda) ,\dv(\lambda)](s)\, d\lambda\right | &\leq
Mq^2  \frac{|\log s|^5}{s^3} \|\dv_1 - \dv_2\|_{1}^{1,3}.
\end{align*}
Then for $s\in [s_*,s_1]$, increasing $s_*$, if necessary
\begin{equation}\label{bound:N2inner2}
\begin{aligned}
\big |\widehat{\opP}[\overline{\dg}_1,\dv_1](s)- \widehat{\opP}[\overline{\dg}_2,\dv_2](s)\big |  \leq
\frac{M}{\esv s^{5/2} } \|\overline{\dg}_1-\overline{\dg}_2\| +
Mq^2  \frac{1}{s^{5/2}}\|\dv_1 - \dv_2\|_{1}^{1,3} .
\end{aligned}
\end{equation}
By bounds~\eqref{bound:N2inner} and~\eqref{bound:N2inner2},
we have that
$$
\|\widehat{\opP}[\opF_1[\dg_1,\dv_1],\dv_1]- \widehat{\opP}[\opF_1[\dg_2,\dv_2],\dv_2]\|_{0}^{5/2,0} \leq M\|\overline{\dg}_1 - \overline{\dg}_2 \| + Mq^2  \|\dv_1-\dv_2\| .
$$
We use now Lemma~\ref{lem:Linear2inner}, that $\| \cdot \|_{1}^{1,3} \leq \| \cdot \|_{1}^{1,0}$ and  again the first item in Lemma~\ref{lem:lipFinner} to conclude
\begin{align*}
\left \|\opL_2\big [ \widehat{\opP}[\opF_1[\dg_1,\dv_1],\dv_1]\big ]- \opL_2\big [\widehat{\opP}[\opF_1[\dg_2,\dv_2],\dv_2]\big ] \right \|_{1}^{1,3} \leq &
M \esvi q^2   \|\dg_1 - \dg_2\| \\ &+ M(\rho^2 \c^{-1}  +q^2 ) \|\dv_1 -\dv_2\|_{1}^{1,3}.
\end{align*}
Finally, also by the bound in~\eqref{bound:N1innerLip}, since $\opR_2=\opP+ \widehat{\opP}$, the second item of Lemma~\ref{lem:lipFinner} is proven.

\appendix
\section{The dominant solutions in the outer region. Proof of Proposition~\ref{prop:dominantouter}}\label{sec:dominantouter}

Along this section we will work with \emph{outer variables} (see~\eqref{eq:canvi_0}) namely $R=kqr$ and, according to definition~\eqref{dominanttermouter_0} and~\eqref{def:epsilon},
$$
F_0(R)=F_0(R;k,q)=\fdout(R/\e), \qquad V_0(R)=V_0(R;q) = k^{-1} \vdout(R/\e),\qquad \e=kq.
$$
We also recall that, $V_0(R)=K'_{inq}(R)/K_{inq}(R)$ (see~\eqref{eq:V0_0}), and $F_0$ was defined in~\eqref{eq:F0_0}.

The proof of Proposition~\ref{prop:dominantouter} requires a thorough analysis, among other things, of the Bessel function $K_{inq}$. We separate it into different subsections which correspond to the different items in the Proposition.

\subsection{The asymptotic behaviour of $\fdout,\vout$ for $kqr\gg 1$}
This short section corresponds to the first item. {Consider $q<\frac12$,}
using the asymptotic expansions~\eqref{asymKIn_new} for $K_{inq}$, we have that
\begin{equation}\label{asympV0_proof_new}
V_0(R) = \frac{K_{inq}'(R)}{K_{inq}(R)}=- \frac{\displaystyle{1 + \frac{c_1}{R} + \O\left (\frac{1}{R^2}\right )}}{\displaystyle{1 + \frac{\overline{c}_1}{R} +
\O\left (\frac{1}{R^2}\right )}}
=- 1 -\frac{c_1}{R}+\frac{\overline{c}_1}{R} + \O\left (\frac{1}{R^2}\right ), \qquad \textrm{as $R\to \infty$,}
\end{equation}
with
$$
\overline{c}_1 - c_1= \frac{4(inq)^2 -1}{8} -\frac{4(inq)^2 +3}{8} = -\frac{1}{2},
$$
and the claim is proved. This expansion is valide for $R\geq R_0$ with $R_0$ independent of $q$. The expansion for $F_0$ is:
\begin{equation*}
\begin{aligned}
F_0(R) &= \sqrt{1-k^2 V_0^2-\e^2 \frac{n^2}{R^2}} = \sqrt{1- k^2 \left (1 + \frac{1}{R} + \O\left (\frac{1}{R^2} \right ) \right )- \e^2 \frac{n^2}{R^2}} \label{defF0appendix}
\\ &=\sqrt{1-k^2} \sqrt{1 - \frac{k^2}{R(1-k^2)} + \O\left (\frac{k^2}{R^2} \right )}
\\ &=\sqrt{1-k^2} \left (1 - \frac{k^2}{2R(1-k^2)} + \O\left (\frac{k^2}{R^2}\right )\right ),
\end{aligned}
\end{equation*}
where we have also used that $k=\e/q$ is small (compare with~\eqref{eq:asymptoticsinfinityR_0}). Going back to the original variables we obtain the result.

\subsection{Asymptotic expression of $\vdout$ for $2  e^{-\frac{n}{2nq}}\leq kqr\leq (qn)^2 $}
Now we deal with the asymptotic expression in~\eqref{vdoutmaching_new} (item~\ref{item2dominantouter}) which in \emph{outer variables} reads as:
\begin{equation}\label{firstasV0appendix}
V_0(R) = \frac{nq}{R} \mathrm{cotan} \left (nq \log \left (\frac{R}{2}\right ) -\theta_{0,nq}\right ) [1+ \mathcal{O}(q^2)], \qquad 2 e^{-\frac{\pi}{2nq}} \leq R \leq q^2n^2,
\end{equation}
with $\theta_{0,nq}=\mathrm{arg} \Gamma(1+inq)=-\gamma nq + \mathcal{O}(q^2)$ and $\gamma$ the Euler's constant.

Let $\nu=nq$.
We first recall some properties of $K_{i\nu}$ with $\nu>0$, see~\cite{dunster, Olverhand}. For $x\in \R$ (in fact the formula is also satisfied for some complex domains), we have that
\begin{equation}\label{KI}
K_{i\nu}(x)= -\frac{\pi i}{2\sinh (\nu \pi)} \big [I_{-i\nu} (x) - I_{i\nu}(x)\big ],
\qquad I_{\eta }(x) = \left (\frac{x}{2}\right )^{\eta}
\sum_{k\geq 0} \left (\frac{x^2}{2}\right )^{k} \frac{1}{k! \Gamma(\eta + k + 1)},
\end{equation}
where $\Gamma(z)$ is the Euler Gamma function
\[
\Gamma(z)=\int_0^\infty t^{z-1} \exp(-t)\,\dd t.
\]
Using that
\begin{equation}\label{formulaGamma}
\Gamma(1+k+\nu i)= (k+\nu i) \cdots (1+\nu i)\Gamma(1+\nu i), \qquad
|\Gamma(1\pm i \nu)|^2 = \frac{\pi \nu}{\sinh(\pi \nu)},
\end{equation}
and denoting $\theta_{k,\nu} =\arg(\Gamma(1+k+i\nu))$,
from~\eqref{KI} we deduce that
\begin{equation}\label{Kserie}
K_{i\nu}(x) = -\frac{1}{\nu} \left ( \frac{\nu \pi}{\sinh \pi \nu}\right )^{1/2}
\sum_{k\geq 0} \left (\frac{x^2}{2}\right )^{k}\frac{\sin \left (\nu \log \left (\frac{x}{2}\right ) - \theta_{k,\nu}\right )}
{k! \big [(k^2 + \nu^2) \cdots (1+ \nu^2)\big ]^{1/2}}.
\end{equation}
By convention,  when $k=0$, $k! \big [(k^2 + \nu^2) \cdots (1+ \nu^2)\big ]^{1/2}=1$.

By formula~\eqref{formulaGamma}, we have that
$$
\arg (\Gamma(1+k+\nu i))= \sum_{l=1}^k \arg(l+\nu i) + \arg(\Gamma(1+\nu i)).
$$
Now we notice that
\begin{equation}\label{ap:theta0nu}
-\theta_{0,\nu} = -\arg \Gamma(1+ \nu i)=\gamma \nu + \O(\nu^2),
\end{equation}
being $\gamma$ the Euler's constant. Indeed, it is well known (\cite{abramowitz}) that
$$
\log \Gamma(1+z)=-\log(1+z) +z(1-\gamma) + \O(z^2).
$$
Then
\begin{align*}
\Gamma(1+i\nu) & = \frac{1}{1+i\nu}  e^{i\nu(1-\gamma)+\O(\nu^2)} = (1 - i\nu +\O(\nu^2))(1+ i\nu(1-\gamma) + \O(\nu^2))\\& = 1 -\gamma i \nu +\O(\nu^2),
\end{align*}
and henceforth, $\arg \Gamma(1+i\nu)=-\gamma \nu + \O(\nu^2)$ as we wanted to check.

We use the expansion~\eqref{Kserie} for $K_{i\nu}$ which has a decomposition
\begin{equation}\label{der0K}
K_{i\nu}(x)= \frac{1}{\nu} \left [\frac{\nu \pi}{\sinh \nu \pi}\right ]^{1/2} \left \{
-\sin\left (\nu \log \left (\frac{x}{2}\right )- \theta_{0,\nu}\right ) + h(x) \right \},
\end{equation}
with $h(x)$ satisfying that $|h(x)|\leq C |x|^2$, $|h'(x)| \leq C|x|$ and $|h''(x)|\leq C$. Therefore
\begin{equation}\label{der1K}
K'_{i\nu}(x) = \left [\frac{\nu \pi}{\sinh \nu \pi}\right ]^{1/2}
\left \{ -\frac{1}{x} \cos \left (\nu \log \left (\frac{x}{2}\right )- \theta_{0,\nu}\right ) +\frac{h'(x)}{\nu}\right \},
\end{equation}
and as a consequence
$$
V_0(R) = \frac{nq}{R } \frac{\cos \left (nq \log \left (\frac{R}{2}\right ) - \theta_{0,nq}\right ) - (nq)^{-1}R h'(R)}
{\sin \left (nq \log \left (\frac{R }{2}\right ) - \theta_{0,nq}\right ) - h(R )} ,\\
$$
with $|h(x)|, |xh'(x)|\leq Cx^2$ and $\theta_{0,nq}= \mathrm{arg} \Gamma(1+inq) = -nq \gamma + \O(q^2)$.

We notice now that when
$2 e^{-\frac{\pi}{2\nu}}\leq x\leq \nu^2$
$$
-\frac{\pi}{2} + \nu \gamma + \O(\nu^2) \leq \nu \log \left (\frac{x}{2}\right ) - \theta_{0,\nu} \leq  -2 \nu |\log \nu| (1+ \O(|\log \nu|^{-1}).
$$
Then, taking $\nu=nq$ we deduce that, for $2  e^{-\frac{n}{2nq}}\leq R\leq (qn)^2 $
\begin{equation}\label{ap:V0ab}
\begin{aligned}
a(R) &:=\frac{Rh'(R)}{nq \cos \left (nq \log \left (\frac{R}{2}\right ) - \theta_{0,nq}\right )} \leq C (nq)^4 \frac{1}{(nq)^2 \gamma} (1+ \O(q^2)) \leq C(nq)^2, \\
|b(R)| &:=\left|\frac{h(R)}{\sin \left (nq \log \left (\frac{R}{2}\right ) - \theta_{0,nq}\right )} \right|
\leq C (nq)^4 \frac{1}{q|\log q|} (1+ \O(|\log q|^{-1})
\leq C(nq)^3
\end{aligned}
\end{equation}
and therefore
\begin{equation}\label{ap:V0Rq2}
V_0(R)=\frac{nq}{R} \mathrm{cotan} \left (nq \log \left (\frac{R}{2}\right ) - \theta_{0,nq}\right )\frac{1-a(R)}{1-b(R)}.
\end{equation}
The result in~\eqref{firstasV0appendix} (and consequently item~\ref{item2dominantouter} of Proposition~\ref{prop:dominantouter}) follows from~\eqref{ap:V0Rq2} and~\eqref{ap:V0ab}.

\subsection{Monotonicity of $\vdout$ and $\fdout$}

This section is devoted to prove item~\ref{item3dominantouter} in Proposition~\ref{prop:dominantouter}.
Fisrt we will see that $\vdout,\fdout$ are increasing functions for $2  e^{-\frac{n}{2nq}}\leq kqr $.
It is equivalent to prove that $V_0'(R), F_0'(R)$ are increasing functions in the corresponding domain $2  e^{-\frac{n}{2nq}}\leq R $.

We begin by $V_0$. Using expansion~\eqref{asymKIn_new} for $K_{inq}$ and the corresponding expansions for $K_{inq}', K_{inq}''$, we have that for $R\gg 1$
$$
V_0'(R)=\frac{1}{2R^2} + \O\left (\frac{1}{R^3}\right ),
$$
so that $V_0'(R)>0$ if $R\gg 1$.

Assume then that there exists {$R_*>2  e^{-\frac{n}{2nq}}$} such that
$V'_0(R_*)=0$ and take the larger $R_*$ critical point.
That is $V'_0(R)\neq 0$ if $R>R_*$. Notice that,
using that $V_0(R) \to -1$ as $R\to \infty$ and $V_0'(R)>0$ if {$R\gg R_*$} we deduce that $V_0(R_*)<-1$ and $V''_0(R_*)\geq 0$, indeed, if $V''_0(R_*)<0$,
it should be a maximum which is a contradiction. Then, since $V_0$ is solution of~\eqref{Ricatti_0}:
$$
\frac{V_0(R_*)}{R_*} + V_0^2(R_*) + \frac{q^2n^2}{R_*^2} - 1=0,
$$
or equivalently
\begin{equation*}
V_0(R_*) = v_{\pm}(R_*):=\frac{1}{2} \left [ -\frac{1}{R_*} \pm \sqrt{\frac{1}{R_*^2} + 4\left (1 - \frac{q^2n^2}{R_*^2}\right)}\right ]
= \frac{1}{2} \left [ -\frac{1}{R_*} \pm \sqrt{\frac{1}{R_*^2} (1-4q^2n^2)  + 4}\right ] .
\end{equation*}
Note that, when $q$ is small enough, $v_{\pm}(R)$ are defined for all $R>0$, and
$$
\lim_{R\to 0}v_{\pm}(R)=-\infty, \qquad \lim_{R \to \infty}v_{+}(R)=1, \qquad \lim_{R \to \infty}v_{-}(R)=-1,
$$
$v_{-}(R)<v_+(R)$ for all $R>0$. We also have that $V_0(R)<-1$, $v_-(R)< V_0(R)<v_{+}(R)$ if $R\gg 1$.

We emphasize that, differentiating equation~\eqref{Ricatti_0}, we obtain that
$$
V_0''(R) + \frac{V'_0(R)}{R} - \frac{V_0(R)}{R^2} + 2 V_0 V_0' - 2 \frac{q^2 n^2}{R^3}=0.
$$
Evaluating at $R=R_*$ we have that
$$
V_0''(R_*) - \frac{V_0(R_*)}{R_*^2} - 2 \frac{q^2 n^2}{R_*^3} =0 \Longleftrightarrow V_0''(R_*)= \frac{V_0(R_*)}{R_*^2} + 2 \frac{q^2 n^2}{R_*^3}.
$$
That is, assuming that $V_0(R_*)=v_-(R_*)$, we obtain:
$$
V_0''(R_*) = \frac{1}{2R_*^2}  \left [ -\frac{1}{R_*} -\sqrt{\frac{1}{R_*^2} + 4\left (1 - \frac{q^2n^2}{R_*^2}\right)}\right ]+ 2 \frac{q^2 n^2}{R_*^3},
$$
and it is clear that, if $q$ is small enough, $V_0''(R_*)<0$ and therefore we have a contradiction with the fact that
$R_*$ can not be a maximum.
We conclude then that $V_0(R_*)=v_+(R_*)$. In this case, $V_0(R_*)<-1$ if and only if
$$
-1+\frac{1}{2R_*}  > \frac{1}{2} \sqrt{\frac{1}{R_*^2} + 4\left (1 - \frac{q^2n^2}{R_*^2}\right)},
$$
which implies that $R_*<1/2$ and
$$
1-\frac{1}{R_*} > 1-\frac{q^2n^2}{R_*^2} \Longleftrightarrow R_*<q^2n^2.
$$
We recall that $V_0''(R_*)>0$. Therefore,using again that
\begin{equation}\label{boundV0R*}
V_0''(R_*)= \frac{V_0(R_*)}{R_*^2} + 2 \frac{q^2 n^2}{R_*^3}>0 \Rightarrow V_0(R_*) > -2 \frac{q^2 n^2}{R_*}.
\end{equation}
Since $2 e^{-\frac{\pi}{2nq}}< R_*<q^2n^2$, using~\eqref{ap:V0Rq2}, we rewrite $V_0(R_*)$ as:
$$
V_0(R_*) = \frac{nq}{R_*} \frac{\cos \left (nq \log \left (\frac{R_*}{2}\right ) - \theta_{0,nq}\right )}{\sin \left (nq \log \left (\frac{R_*}{2}\right ) - \theta_{0,nq}\right )} \frac{1+ a(R_*)}{1+b(R_*)}.
$$
Using~\eqref{ap:V0ab} and that the function $\cos(x)/\sin(x)$ is a decreasing function if
$x\in [-\pi/2,0]$, we have that,
\begin{align*}
V_0(R_*) &\leq \frac{nq}{R_*} \frac{1+ a(R_*)}{1+b(R_*)}
\frac{\cos \left (nq \log \left (\frac{(nq)^2}{2}\right ) - \theta_{0,nq}\right )}{\sin \left (nq \log \left (\frac{(nq)^2}{2}\right ) - \theta_{0,nq}\right )} =
\frac{nq}{R_*} \frac{1}{2nq \log (nq)}(1+ \O(q^2|\log q|^2)) \\ &= - \frac{1}{2 R_* |\log (nq)|} (1+ \O(q^2|\log q^2))
\end{align*}
which is a contradiction with~\eqref{boundV0R*}. Then we conclude that $V_0'(R)>0$ for {$ 2 e^{-\frac{\pi}{2nq}}< R$}.

Note that since we have proved that $V'_0(R)>0$ for $R\geq 2 e^{-\frac{\pi}{2\nu}}$ then
by~\eqref{asympV0_proof_new} $V_0(R)=-1 - \frac{1}{2R} + \O(1/R^2)$ if $R\gg 1$ which implies that $V_0(R) \to -1$ when $R\to \infty$ and hence
$V_0(R)<-1$ in the same domain.

Differentiating the expression for $F_0$ (see for instance~\ref{defF0appendix}) and using that $V'_0(R)>0$ we easily obtain that $F_0'(R)>0$.

Going back to the original variables, item~\ref{item3dominantouter} of Proposition~\ref{prop:dominantouter} is proven.

\subsection{Bounds for $\vdout$ and $\fdout$}

This section is devoted to prove the the  bounds for $\vdout$ and $\fdout$ and its derivatives given in item~\ref{item4dominantouter} of Proposition~\ref{prop:dominantouter}.

Let us first provide a technical lemma whose proof is postponed to the end of this section.
\begin{lemma}\label{lem:Kinq_der}
There exists $q_0>0$, such that  if $0<q<q_0$,
the modified Bessel function $K_{inq}(R)$ satisfies:
\begin{equation*}
K_{inq}(R)>0, \qquad K'_{inq}(R)<0, \qquad K''_{inq}(R)>0, \qquad
\text{for all } \, {R\geq 2 e^2 e^{-\frac{\pi}{2qn}}.}
\end{equation*}
\end{lemma}

 We point out that, in \emph{outer variables}, in order to prove the bounds in items~\ref{item3dominantouter} and~\ref{item4dominantouter}, it is enough to prove the following result (see also Corollary~\ref{lem:Kinq2_new}):
\begin{lemma}
Let
$\a \in (0,1)$. There exists $q_0=q_0(\a) >0$ and a constant $M>0$ such that for any $0<q<q_0$ and $R\in [R_\m,+\infty)$ with $R_\m$ satisfying $2 e^2 e^{-\frac{\pi}{2qn}}\leq R_\m \leq \e^{\a}$, where $\e=kq$, one has
$$
|kV_0(R)|, \ |k V_0'(R)R|, \ |kV''(R)R^2|\leq M \e R_\m^{-1},
$$
and
$$
|R (V_0(R)+1)|,|R^2V_0'(R)|, |R^3 V_0''(R)|\leq M .
$$
With respect to $F_0$, we have that $
F_0(R) \geq 1/2$
and
$$
|F'_0(R) R^2|, |F''_0(R) R^3| \leq  M k \e R_\m^{-1},\qquad |1-F_0(R)|, |F_0'(R)R|, |F_0''(R)R^2|\leq M \e^2 R_\m^{-2}.
$$
\end{lemma}
\begin{proof}
Because of item~\ref{item3dominantouter} of Proposition~\ref{prop:dominantouter}, $V_0$ is an increasing and negative function on  $[2 e^{-\frac{\pi}{2qn}},\infty]$ and therefore in $[R_\m,\infty)$. Therefore we have that $|kV_0(R)|\leq k |V_0(R_\m)|$. We notice that, from~\eqref{der0K} and~\eqref{der1K}
\begin{align*}
 V_0(R_\m) &= \frac{K'_{inq}(R_\m)}{K_{inq}(R_\m)}=-
\frac{-\frac{1}{R_\m} \cos \left (nq \log \left (\frac{R_\m}{2}\right )- \theta_{0,nq}\right ) +\frac{h'(R_\m)}{nq}}
{\frac{1}{nq}\left \{
-\sin\left (nq \log \left (\frac{R_\m}{2}\right )- \theta_{0,nq}\right ) + h(R_\m) \right \}} \\
&= \frac{nq}{R_\m}
\frac {\cos \left (nq\log \left (\frac{R_\m}{2}\right )- \theta_{0,nq}\right ) - R_\m \frac{h'(R_\m)}{nq}}
{\sin\left (nq \log \left (\frac{R_\m}{2}\right )- \theta_{0,nq}\right ) - h(R_\m) },
\end{align*}
with $h(R)$ satisfying that $|h(R)|\leq M |R|^2$ and $|h'(R)| \leq M|R|$. We recall that $ \e=kq =\mu e^{-\frac{\pi}{2nq}}$ and $-\theta_{0,nq}= \gamma nq + \O(q^2)$. Then, since $R_\m \leq \e^{\a}\ll q$
$$
| kV_0(R_\m) | \leq    k \frac{nq}{R_\m}\frac{1+ \O(R_\m)}
{\left |\sin \left(-\frac{\pi \a }{2} + \O(q) \right ) \right |+ \O(R_\m^2)}
\leq
M \e R_\m^{-1}.
$$

Define now the function $g(R) = R V_0(R)$. We want to see that, for $R\ge R_\m$, $g'(R)\ne 0$.
Assume that, for some $R_*$, the function has a critical point, namely, $g'(R_*)=R_*V_0'(R_*) + V_0(R_*) =0$. Then
using the equation \eqref{Ricatti_0} satisfied by $V_0$ we get:
$$
V_0^2 (R_*) -1 +\frac{q^2n^2}{R_*^2} =0 \Longleftrightarrow V_0^2 (R_*)=1 -  \frac{q^2n^2}{R_*^2},
$$
which is a contradiction with the fact that $V_0(R)<-1$.
Therefore, $g'(R) = RV'_0(R)+ V_0(R) \ne 0$ for $R\ge R_\m$.

Recall that, for $R\gg 1$,
$$
V_0(R) = -1 - \frac{1}{2R} + \O\left (\frac{1}{R^2}\right ).
$$
As a consequence
$$
g'(R) = -1 - \O(R^{-2}) \to -1 ,\qquad \text{as  } R\to \infty
$$
and therefore, $g'(R)<0$ for all $R\geq R_\m$.

Then $g(R_1)\leq g(R_2)$ if $R_1 \geq R_2$ and using that $g(R)<0$,
we conclude that $| g(R_2)|\leq |g(R_1)|$ when $R_1 \geq R_2$.
On the other hand, $|R(V_0(R)+1)| \leq M$ when $R\geq R_0$ if $R_0$ is big enough (but independent of $q$). Thus,
if $R_\m \leq R \leq R_0$,
$$
|R(V_0(R)+1)| = |RV_0(R)| -R\leq R |V_0(R)| \leq R_0 |V_0(R_0)| \leq M \e R_\m^{-1} .
$$

With respect to $V'_0(R)$ we have that $|R^2 V'_0(R) |\leq M$ if $R\geq R_0$ with $R_0$ big enough.
Take now $R_{\m} \leq  R \leq R_0$. We
recall that $V_0(R)=K'_{inq}(R) / K_{inq}(R)<0$ and we notice that
$$
0<V'_0(R) = \frac{K_{inq}''(R)}{K_{inq}(R)} - \left (\frac{K'_{inq}(R)}{K_{inq}(R)}\right )^2 \leq \frac{K_{inq}''(R)}{K_{inq}(R)}.
$$
The modified Bessel function $K_{inq}$ satisfies the linear differential equation
$$
K_{inq}'' + \frac{K'_{inq}(R)}{R} - K_{inq}(R) \left (1- \frac{n^2q^2}{R^2}\right )=0.
$$
Then, using that, by Lemma~\ref{lem:Kinq_der}, for $R\geq R_\m \ge 2e^2 e^{-\frac{\pi}{2nq}}$ we know that $K_{inq}(R)>0$, $K_{inq}'(R)<0$ and $K_{inq}''(R)>0$ and therefore:
$$
0<K''_{inq}(R) = -\frac{K'_{inq}(R)}{R} + K_{inq}(R) \left (1- \frac{n^2q^2}{R^2}\right ) \leq -\frac{K'_{inq}(R)}{R} + K_{inq}(R).
$$
Therefore, if $R_{\m} \leq  R \leq R_0$
$$
|R^2 V'_0(R)| = R^2 V'_0(R) \leq -R \frac{K'_{inq}(R)}{ K_{inq}(R)} + R^2 = R|V_0(R)| +R^2 \leq R_0|V_0(R_0)| +R_0^2 \leq M.
$$
In addition, using that $V_0$ satisfies equation~\eqref{Ricatti_0}
$$
0<k R V'_0(R) =-k V_0(R) -kR(V_0^2(R)-1) - k \frac{q^2n^2}{R} \leq -kV_0(R) \leq M \e R_{\m}^{-1}.
$$

Now we deal with $V''_0(R)$. We have that, when $R\geq R_0$ with $R_0$ big enough (but independent of $q$), $|R^3 V_0''(R)|\leq M$. For
$R_\m \leq R\leq R_0$,
$$
V_0''(R) = \frac{V_0}{R^2 } - \frac{V'_0}{R} + 2 V_0 V'_0(R) + \frac{n^2q^2}{R^3}.
$$
Therefore, using that $|RV_0(R)|$ and $|R^2 V_0'(R)|\leq M$ for $R_\m \leq R\leq R_0$ we obtain:
$$
|R^3 V_0''(R)| \leq M.
$$
Moreover, using that $kq R^{-1}_{\m} \leq \e^{1-\a}$
$$
|kR^2 V_0''(R) | \leq |k V_0(R) | + |kR V_0'(R)| + 2 k |V_0(R) | |R^2 V_0'(R)| +k \frac{n^2 q^2}{R}  \leq M \e R_{\m}^{-1}.
$$

Now we deal with the properties of $F_0$ and its derivatives.
Since $|kV_0(R)| \leq M \e^{1-\alpha}$ and $F_0(R) = \sqrt{1-k^2 V_0^2-\e^2 n^2 R^{-2}}$,
we have that
$$
F_0(R)= 1- \sum_{n\geq 1} a_n B_0(R)^{n} , \qquad a_n>0,
$$
with
$$
B_0(R)=  k^2 V_0^2(R) +\frac{\e^2 n^2}{R^2} .
$$
Then
\begin{align*}
    F_0'(R)&=-\sum_{n\geq 1} na_n B_0^{n-1}(R) B_0'(R) , \\
    F_0''(R) & -\sum_{n\geq 1 } na_n \big [(n-1)B_0^{n-2}(R) (B_0'(R))^2  + B_0^{n-1}(R) B_0''(R) \big ].
\end{align*}
Using the properties for $V_0$, we deduce from the above expression, the corresponding ones for $F_0$.
\end{proof}

To finish the proof of Proposition~\ref{prop:dominantouter} we prove Lemma~\ref{lem:Kinq_der}.
\begin{proof}[Proof of Lemma~\ref{lem:Kinq_der}]
We take $\nu=nq \le \nu_0$.
Besides expression~\eqref{Kserie} of $K_{i\nu}$ we also have the integral expression:
\begin{equation}
\label{KInt}
K_{i\nu}(x)=\int_0^\infty \exp (-x\cosh t)\cos(\nu t)\dd t ,
\end{equation}
from which we deduce that $K_{i\nu}(x)$ is real if $x\in \R$.

Notice that, from the asymptotic expression (see~\eqref{asymKIn_new}), there exists $x_0$ only depending on $\nu_0$ such that $\forall x\ge x_0$ :
\begin{equation}
\label{eqAsympt}
K_{i\nu}(x)= \sqrt{\frac{\pi}{2 x}}e^{-x} \left (1+ \O\left (\frac{1}{x}\right )\right ) >0,\qquad
K_{i\nu}'(x)= -\sqrt{\frac{\pi}{2 x}}e^{-x} \left (1+ \O\left (\frac{1}{x}\right )\right )<0,
\end{equation}
therefore, we only need to prove that $K_{i\nu}''(x)>0$.

We first claim that $K_{i\nu}''(x)>0$ if $x\geq \nu^2 $ and $\nu>0$. Indeed, differentiating twice the expression~\eqref{KInt}:
$$
K''_{i\nu}(x) =  \int_{0}^\infty  \exp (-x\cosh t) \cosh^2 t \cos(\nu t)  \dd t .
$$
For $0\leq \nu t\leq \frac{\pi}{4}$, we have that $\cos(\nu t) \geq \frac{\sqrt{2}}{2}$ and then, also using that
$e^t \leq 2 \cosh t\leq e^t + 1 \leq 2 e^t$ for $t\geq 0$, we obtain
\begin{align*}
K''_{i\nu}(x)  & \geq \frac{\sqrt{2}}{2} \int_0^{\frac{\pi}{4\nu}} \exp(-x\cosh (t)) \cosh^2 t\dd t - \int_{\frac{\pi}{4\nu}}^{\infty}
\exp(-x\cosh (t)) \cosh^2 t\dd t  \\
& \geq \frac{\sqrt{2}}{8} \int_0^{\frac{\pi}{4\nu}} \exp\left (-x \frac{e^t+1}{2} \right ) e^{2t}\dd t -
 \int_{\frac{\pi}{4\nu}}^{\infty} \exp \left (-x \frac{e^t}{2}\right )e^{2t} \dd t
 \\ &= \frac{\sqrt{2}}{8} \exp\left (- \frac{x}{2}\right )\int_0^{\frac{\pi}{4\nu}} \exp\left (- \frac{x}{2}e^t \right ) e^{2t}\dd t -
 \int_{\frac{\pi}{4\nu}}^{\infty} \exp \left (-\frac{x}{2}e^t\right )e^{2t} \dd t .
\end{align*}
Note that, performing the obvious change $u=e^t$:
\begin{align*}
\int \exp \left (-x \frac{e^t}{2}\right )e^{2t} \dd t &= \int \exp \left (-\frac{x}{2} u\right ) u \dd u
=-\frac{2}{x} \exp \left (-\frac{x}{2} u\right ) u + \frac{2}{x} \int \exp \left (-\frac{x}{2} u\right ) \dd u
\\ &= -\frac{2}{x} \exp \left (-\frac{x}{2} u\right ) u - \frac{4}{x^2} \exp \left (-\frac{x}{2} u\right )
\\ &= -\frac{2}{x} \exp \left (-\frac{x}{2} e^t\right ) e^t - \frac{4}{x^2} \exp \left (-\frac{x}{2} e^t\right )
\\ & =-  \frac{2}{x^2}\exp \left (-\frac{x}{2} e^t\right ) \left [x e^t + 2\right ]  =:-F(t).
\end{align*}
We obtain then
$$
K''_{i\nu}(x)   \geq \left [F(0) - F\left (\frac{\pi}{4\nu}\right )\right] \frac{\sqrt{2}}{8} e^{-x/2} - F\left (\frac{\pi}{4\nu}\right ).
$$
In order to check that $K''_{i\nu}(x)>0$, we have to prove that
$$
F(0) > F\left (\frac{\pi}{4\nu}\right ) \left [1 + \frac{8}{\sqrt{2}} e^{x/2}\right ].
$$
Since $x\geq 0$, it is enough to check that
$$
2  > \exp \left ( -\frac{x}{2} \left (e^{\frac{\pi}{4 \nu}} -1\right)\right ) \left (x e^{\frac{\pi}{4 \nu}} +2 \right )
\left (1+\frac{8}{\sqrt{2}}e^{x/2}\right ).
$$
On one hand, $x\geq \nu^2$ with $\nu$ small enough, implies that $2 \leq \nu^2 e^{\frac{\pi}{4 \nu}} \leq x e^{\frac{\pi}{4 \nu}}$.
On the other hand, it is clear that $1 \leq e^{x/2}$ if $x>0$ and $x\leq e^x$. Therefore, the above inequality is satisfied if
$$
2 > 6\frac{8}{\sqrt{2}} \exp \left ( -\frac{x}{2} \left (e^{\frac{\pi}{4 \nu}} -1\right)\right )e^x e^{\frac{\pi}{4 \nu}}e^{x/2}
\Longleftrightarrow \frac{\sqrt{2}}{24} >  \exp \left (-\frac{x}{2} \left (e^{\frac{\pi}{4 \nu}} -4\right) + \frac{\pi}{4\nu}\right ),
$$
for all $x\geq \nu ^2$. Thus we need $\nu$ to satisfy:
$$
\frac{\sqrt{2}}{24} >  \exp \left (-\frac{\nu^2}{2} \left (e^{\frac{\pi}{4 \nu}} -4\right) + \frac{\pi}{4\nu}\right ),
$$
which is true if $\nu$ is small enough.

In conclusion, we have proved that, for $\nu>0$ small enough and $x\ge \nu^2$, the function $K_\nu$ satisfies $K''_{i\nu}(x)>0$. It remains to prove that $K''_{i\nu}(x)>0$ if $x\leq \nu^2$. From~\eqref{der0K} and~\eqref{der1K} we have that

\begin{equation}\label{der2K}
K''_{i\nu}(x) = \left [\frac{\nu \pi}{\sinh \nu \pi}\right ]^{1/2} \left \{
\frac{\nu}{x^2} \sin \left (\nu \log \left (\frac{x}{2}\right )- \theta_{0,\nu}\right ) + \frac{1}{x^2} \cos \left (\nu \log \left (\frac{x}{2}\right )- \theta_{0,\nu}\right )+ \frac{h''(x)}{\nu}\right \}.
\end{equation}
For $2 e^2 e^{-\frac{\pi}{2\nu}}\leq x \leq \nu^2$, it is clear from~\eqref{ap:theta0nu}
\begin{align*}
& \nu \log \left (\frac{x}{2}\right )- \theta_{0,\nu} < 2 \nu\log \nu + (\gamma-\log 2) \nu +\O(\nu^2) <0, \\
& \nu \log \left (\frac{x}{2}\right )- \theta_{0,\nu} > 2\nu  - \frac{\pi}{2} + \gamma \nu + \O(\nu^2) > -\frac{\pi}{2},
\end{align*}
if we take $\nu$ small enough. Therefore, if $\nu$ is small enough,
\begin{align*}
\cos \left (\nu \log \left (\frac{x}{2}\right )- \theta_{0,\nu}\right ) & \geq \cos \left (
 - \frac{\pi}{2} + 2\nu + \gamma \nu + \O(\nu^2) \right )
= \sin \left ( (2+ \gamma) \nu + \O(\nu^2) \right )
\\ & \geq \left (1 + \frac{\gamma}{2} \right )\nu ,
\\
\sin \left (\nu \log \left (\frac{x}{2}\right )- \theta_{0,\nu}\right ) & \geq -1.
\end{align*}
Then, from expression~\eqref{der2K} of $K''_{i\nu}(x)$
$$
K''_{i\nu}(x) \geq
\left [\frac{\nu \pi}{x^4 \sinh \nu \pi}\right ]^{1/2} \left \{
\left (1+\frac{\gamma}{2}\right ) \nu  - \nu - C \frac{x^2 }{\nu} \right \}
\geq
\left [\frac{\nu \pi}{x^4 \sinh \nu \pi}\right ]^{1/2} \left \{
\frac{\gamma}{2} \nu  - C \nu^3 \right \}
>0,
$$
if  $\nu$ is small enough.
Therefore, we have just shown that $K''_{i\nu}(x)\geq 0$ if $x\geq 2e^2  e^{-\frac{\pi}{2\nu}}$. This result along with the asymptotic expressions \eqref{eqAsympt} provides the sign for $K'_{i\nu}$ and $K_{i\nu}$.
\end{proof}

\section{The dominant solutions in the inner region. Proof of Proposition~\ref{prop:f0v0_new}} \label{sec:inner_dominant}

We now prove the asymptotic properties of $\fdin, \vdin$ defined in~\eqref{dominantterminner_0}.
As we have already pointed out, the properties of $\f0=\f0^{\mathrm{in}}$ and $\partial_r \f0^{\mathrm{in}}$ are all provided in~\cite{Aguareles2011}.
With respect to the properties of $\v0^{\mathrm{in}}(r)=q\v0(r)$, with $\v0$ in~\eqref{expr:v_0}, in the second item, in~\cite{AgBaSe2016}  the function
$$
\overline{\v0}(r)=-\frac{1}{r\f0^2(r)}\int_{0}^r \xi \f0^2(\xi)(1-\f0^2(\xi))\dd \xi,
$$
was considered and the same asymptotic properties of $\overline{\v0}$ was considered as the ones stated in the second item but for all $r>0$.
We introduce
$$
\Delta \v0(r;k):=\v0(r;k) - \overline{\v0}(r)  = k^2 \frac{1}{r\f0^2(r)}\int_{0}^r \xi \f0^2 (\xi) \dd \xi.
$$
Note that, if $r\sim 0$, using that $\f0(r) \sim \alpha_0 r^n $
$$
\Delta \v0(r;k) \sim \frac{1}{2n+2} k^2 r, \qquad \partial_r \Delta \v0(r;k) \sim k^2 c ,
$$
for some constant $c$. Then it is clear that, for $r\sim 0$, the properties of $\v0^{\mathrm{in}}(r;k,q)$ are deduced from the analogous ones for $\overline{\v0}(r)$ proven in~\cite{AgBaSe2016}.

When $kr \leq n/\sqrt{2}$ and $r\gg 1$, we have that $1/2\leq f_0(r) \leq 1$. Then
$$
|\Delta \v0(r;k) |\leq M k^2 r.
$$
As a consequence $|\Delta \v0(r;k)|\leq M \frac{n^2}{2r} \leq M|\log r| r^{-1}$ if $kr\leq n/\sqrt{2}$. In~\cite{Aguareles2011} was already proven $|\overline{\v0}(r) |\leq M |\log r| r^{-1}$. Therefore this property (and analogously the one for $\v0'$) is satisfied.

It only remains to check that $\v0<0$. From its definition~\eqref{expr:v_0} it is enough to check that $1-k^2-\f0^2(r)>0$ for $0\leq r \leq \frac{n}{k\sqrt{2}}$.
We first notice that
there exists $r_0\gg 1$ such that
$$
1- \f0^2(r) \geq \frac{n^2}{2r^2},\qquad r\gg r_0.
$$
Therefore, for $kr \leq n/\sqrt{2}$ and $r\gg r_0$, we have that
$1- k^2 - \f0^2(r)\geq 0$.
Since $\f0$ is an increasing function, we have that
$1-k^2 - \f0^2(r)\geq 0$ for all $r\geq 0$ such that
$kr\leq n/\sqrt{2}$.

Now we prove the third item. We first deal with the asymptotic expression of $\vdin=q\v0$.
We use the asymptotic expressions of $\f0(r)$ already proven in the first item, namely $\f0(r) =1 - \frac{n^2}{2r^2} + \mathcal{O}(r^{-4})$ as $r\to \infty$.
We write
\begin{align*}
\v0(r) = &   -\frac{1}{r\f0^2(r)} \int_{0}^{r}  \xi \f0^2(\xi) (1- \f0^2(\xi) ) \dd \xi +\frac{k^2}{r\f0^2(r)}\int_{0}^r \xi \f0^2(\xi) \dd \xi=:\v0^1(r) + \v0^2(r).
\end{align*}
We take $r_* \gg 1$. It is clear that
$$
\frac{k^2}{r\f0^2(r)}\int_{0}^r \xi \f0^2(\xi) \dd \xi =
\frac{k^2}{r\f0(r)} \int_{0}^{r_*} \xi \f0^2(\xi) \dd \xi +
\frac{k^2}{r\f0(r)} \int_ {r_*}^r \xi \f0^2(\xi) \dd \xi.
$$
Notice that
$$
\frac{k^2}{r\f0(r)} \int_{0}^{r_*} \xi \f0^2(\xi) \dd \xi = k^2 \mathcal{O}(r^{-1}),
$$
and, using that $\f0^2(r)= 1-\frac{n^2}{r^2} + \mathcal{O}(r^{-4})$ if $r,r_* \gg 1$,
$$
\frac{k^2}{r\f0(r)} \int_ {r_*}^r \xi \f0^2(\xi) \dd \xi =
k^2 \frac{r^2 - r_*^2}{2r} -\frac{k^2 n^2 \log r}{r} + k^2 \mathcal{O}(r^{-1}).
$$

Consider now $r_*\gg 1$ and let us define
\begin{align*}
\Delta \v0(r,r_*) :=& \v0^1(r) +\frac{n^2}{r\f0^2(r)} \log \left ( \frac{r}{r_*} \right ) +\frac{1}{r\f0^2(r)} \int_{0}^{r_*} \xi \f0^2(\xi)(1-\f0^2(\xi))\dd \xi \\ =&
\frac{1}{r\f0^2(r)} \int_{r}^{r_*} \xi \f0^2(\xi)(1-\f0^2(\xi))\dd \xi + \frac{n^2}{r\f0^2(r)} \log \left ( \frac{r}{r_*} \right ).
\end{align*}
It is clear, using again that $\f0^2(r)= 1-\frac{n^2}{2r^2} + \mathcal{O}(r^{-4})$
\begin{align*}
\Delta \v0(r,r_*) = &  \frac{1}{r\f0^2(r)}\int_{r}^{r_*} \frac{n^2}{\xi} + \mathcal{O}\left (\frac{1}{\xi^3}\right )\dd \xi
+ \frac{n^2}{r\f0^2(r)} \log \left ( \frac{r}{r_*} \right )
\\ &= \mathcal{O}(r^{-3}) + \mathcal{O}(r^{-1} r_*^{-2}).
\end{align*}
Therefore, taking $r_*\to \infty$, we have that
\begin{align*}
\mathcal{O}(r^{-3} ) &=v_0^1(r) + \frac{n^2}{r\f0^2(r) } \log r + \frac{1}{r\f0^2(r)}\lim_{r_* \to \infty}
\left (-n^2   \log r_* +
 \int_{0}^{r_*} \xi \f0^2 (\xi) (1- \f0^2(\xi)\dd \xi \right )\\
 & =v_0^1(r)+\frac{1}{r\f0^2(r)}\left(n^2\log r+ C_n\right) = v_0^1(r)+\frac{1}{r}\left(n^2\log r+ C_n\right)\left(1+\mathcal{O}(r^{-2})\right)
\\ &=v_0^1(r) +\frac{n^2}{r} \log r + \frac{C_n}{r} +\mathcal{O}(r^{-3}\log r),
\end{align*}
with $C_n$ as defined in Theorem~\ref{thm:main}.
Collecting all these estimates, the proof of~\eqref{prop:v0innermatching} is complete.

\section*{Acknowledgements}
This work is part of the grant PID-2021-122954NB-100 (funding T.M. Seara and I. Baldomá) and the projects  PID2020-115023RB-I00 and PDC2021-121088-I00 (funding M. Aguareles) all of them financed
MCIN/AEI/ 10.13039/501100011033/ and by “ERDF A way of making Europe”. M. Aguareles is also funded by TED2021-131455A-I00 .  T. M. S.
are supported by the Catalan Institution for Research and Advanced Studies via an ICREA Academia Prize 2019. This work is also supported by the Spanish State
Research Agency, through the Severo Ochoa and María de Maeztu Program for Centers and Units of
Excellence in R\&D (CEX2020-001084-M).

\bibliographystyle{alpha}
\bibliography{qneq0}
\end{document}